\documentclass[a4paper,12pt]{article} 
\usepackage{graphicx}
\usepackage{amsmath,amsfonts,amsthm,mathtools,amssymb,framed,ascmac,relsize}
\usepackage[top=7em,left=5.5em,right=5.5em]{geometry}

\theoremstyle{plain} 
\newtheorem{theorem}{Theorem}[section]
\newtheorem{corollary}[theorem]{Corollary}
\newtheorem{proposition}[theorem]{Proposition}
\newtheorem{lemma}[theorem]{Lemma}

\theoremstyle{definition} 
\newtheorem{definition}[theorem]{Definition}
\newtheorem{remark}[theorem]{Remark}
\newtheorem{example}[theorem]{Example}

\numberwithin{equation}{section} 

\usepackage{newtxtext,newtxmath}
\usepackage{pgfplots} 
\usepackage{tikz}
\usepackage{bm}
\usetikzlibrary{intersections,calc,arrows.meta}
\pgfplotsset{compat=newest}

\definecolor{new_blue}{rgb}{0,0,0.6}
\usepackage[colorlinks=true,linkcolor=new_blue,citecolor=new_blue,setpagesize=false]{hyperref}
\usepackage{caption}
\usepackage{subcaption}
\allowdisplaybreaks

\usepackage{enumitem}
\setenumerate[1]{label={\textnormal{(\arabic*)}}, itemsep=0em,topsep=0.3em}

\usepackage[
    backend=biber,
    sorting=nyt,
    style=alphabetic,
    url=false,
    isbn=false,
    giveninits=true,
    date=year,
    backref=true,
      maxnames=5
]{biblatex}

\renewbibmacro{in:}{%
  \ifboolexpr{%
     test {\ifentrytype{article}}%
     or
     test {\ifentrytype{inproceedings}}%
  }{}{\printtext{\bibstring{in}\intitlepunct}}%
}

\newcommand{\RR}{\mathbf{R}}
\newcommand{\CC}{\mathbf{C}}
\newcommand{\NN}{\mathbf{N}} 

\newcommand{\Hup}{\mathbf{C}^+} 
\newcommand{\DD}{\mathbf{D}}

\newcommand{\ee}{\varepsilon}

\newcommand{\di}[1]{\mathop{\mathrm{d}#1}}

\newcommand{\Prob}{\mathbb{P}}
\newcommand{\Dil}{D}   
\newcommand{\scp}{\mathrm{sc}} 
\newcommand{\acp}{\mathrm{ac}} 
\newcommand{\ppp}{\mathrm{pp}} 

\newcommand{\cl}[1]{\mathrm{cl}({#1})}

\newcommand{\iu}{\mathrm{i}}  
\newcommand{\conj}{\overline}  
\newcommand{\Aut}{\mathrm{Aut}}   
\newcommand{\inv}{{\langle-1\rangle}} 
\newcommand{\Arg}{\mathrm{Arg}\,}
\newcommand{\Argg}{\mathrm{Arg}}
\newcommand{\nontanglim}{\mathlarger{\sphericalangle}\!\lim} 
\newcommand{\nontangtou}{\underset{\mathlarger{\mathlarger{\sphericalangle}}}{\to}} 
\renewcommand\Delta{\mathlarger{\mathlarger{\rhd}}}  
\renewcommand\Gamma{\mathlarger{\mathlarger{\bigtriangledown}}} 

\newcommand{\multconv}{\boxtimes}
\newcommand{\addconv}{\boxplus}
\newcommand{\MP}{{\bm\pi}} 
\newcommand\sof{\sigma}   

\title{Free multiplicative convolution with an arbitrary measure on the real line}

\author{
    OCTAVIO ARIZMENDI\thanks{Centro de Investigaci\'on en Matem\'aticas. A.C., Jalisco S/N, Col. Valenciana CP: 36023 Guanajuato, Gto, M\'exico \\ 
Email address: octavius@cimat.mx}, 
    TAKAHIRO HASEBE\thanks{
Department of Mathematics, Hokkaido University. Science Bldg.~No.~3, Kita 10, Nishi 8, Kita-Ku, Sapporo, Hokkaido, 060-0810, Japan \\
    Email address:
thasebe@math.sci.hokudai.ac.jp}, AND 
    YU KITAGAWA\thanks{
Department of Mathematics, Hokkaido University. Science Bldg.~No.~3, Kita 10, Nishi 8, Kita-Ku, Sapporo, Hokkaido, 060-0810, Japan \\ Current address: Department of Mathematics, 
Faculty of Science, 
Kyoto University. Kitashirakawa Oiwake-cho, Sakyo-ku,
Kyoto 606-8502, Japan \\ Email address:
kitagawa.yu.57z@st.kyoto-u.ac.jp}
}

\begin{document}
\maketitle

\begin{abstract}
    We develop analytic tools for studying the free multiplicative convolution of any measure on the real line and any measure on the nonnegative real line. More precisely, we construct the subordination functions and the $S$-transform of an arbitrary probability measure. The important multiplicativity of $S$-transform is proved with the help of subordination functions. We then apply the $S$-transform to establish convolution identities for stable laws, which had been considered in the literature only for the positive and symmetric cases. Subordination functions are also used in order to extend Belinschi--Nica's semigroup of homomorphisms, and to establish regularity properties of free multiplicative convolution, in particular, the absence of singular continuous part and analyticity of the density. 
\end{abstract}



\tableofcontents

\section{Introduction and statement of results}

Free independence was introduced by Voiculescu  \cite{V2006symmetries} from considerations in operator algebras. In this theory, two binary operations between probability measures are fundamental: free additive and free multiplicative convolutions that correspond to the addition and multiplication of freely independent operators. 
More precisely, for two probability measures $\mu,\nu$ on $\RR$, their \emph{free additive convolution} $\mu\addconv\nu$ is the law of $X+Y$ where $X,Y$ are self-adjoint operators that are freely independent, affiliated to a von Neumann algebra with a normal faithful tracial state, and have distributions $\mu,\nu$, respectively. Similarly, if moreover $Y$ is positive, or equivalently $\nu$ is supported on the nonnegative real line $\RR_{\ge0}$, then the \emph{free multiplicative convolution} $\mu\multconv\nu$ is defined as the law of $\sqrt{Y}X\sqrt{Y}$. For further details the reader is referred to \cite{BV1993free}. 

The free additive convolution can be calculated by the use of the \emph{$R$-transform} since it  serves as a linearizing transformation, in a similar way as the characteristic function is used to calculate the classical convolution. Although the $R$-transform is a fundamental tool to understand free convolution, one should note that, in practice, the calculation of the free convolution of two specified probability measures is very difficult since to access the $R$-transform one needs to find the inverse under composition of a certain analytic function. One way to circumvent this problem is to use \emph{subordination functions}, which can be obtained by establishing fixed point equations. 

For the case of multiplicative convolution on the nonnegative real axis, a similar situation happens: the \emph{$S$-transform} is multiplicative with respect to free multiplicative convolution and there exist subordination functions which satisfy a fixed point equation on a proper domain of the complex plane. However, until now, the general theory lacked a way to calculate the free multiplicative convolution of a measure on $\RR$ and a measure on $\RR_{\ge0}$, since neither an $S$-transform for general probability measures nor subordination functions are available.  

Our main results solve these two problems: we establish the existence of subordination functions for free multiplicative convolution and provide an $S$-transform for general measures which allow to calculate and analyze free multiplicative convolutions. We then address some applications of these analytic tools to stable laws and regularity properties of free convolution. 

In Sections \ref{sec:S-transform}--\ref{sec:R} below, we summarize the main results of this paper.

\subsection{Basic notations, concepts and transforms}\label{sec:basic_notation}
We introduce some basic notations  used throughout the paper. Let us denote by $\Prob(\RR)$ and $\Prob(\RR_{\ge0})$ the sets of probability measures on $\RR$ and on $\RR_{\ge0}$, respectively. For $z\in\CC \setminus\{0\}$, we define $\arg z$ such that $\arg z=0$ for $z\in(0,+\infty)$, and it increases counterclockwise up to $ 2\pi$. We write $\Hup\coloneqq\{z\in\CC\colon\Im z>0\}$ for the complex upper half-plane and $\CC^-\coloneqq\{z\in\CC\colon\Im z<0\}$ for the lower half-plane. Let $\cl{E}$ denote the topological closure of a subset $E\subseteq \CC$. 

  Given $x_0\in\RR$ and $0<\theta<\pi/2$, we define a domain $\Gamma_{x_0,\theta}\subseteq\Hup$ as 
      \begin{align}
          \Gamma_{x_0,\theta}\coloneqq \{z\in\Hup\setminus\{0\}\ \colon\ \theta<\arg(z-x_0)<\pi-\theta\}. \label{eq:nabla}
      \end{align}
For a function $f\colon\Hup\to\CC\cup\{\infty\}$ and $x_0\in\RR$, we say that \emph{$f$ has a nontangential limit $\zeta \in \CC\cup\{\infty\}$ at $x_0$} if $\lim_{z\to x_0, z\in\Gamma_{x_0,\theta}} f(z)=\zeta$ for any $\theta\in(0,\pi/2)$. The limit $\zeta$ is denoted by $\nontanglim_{z\to x_0}f(z)$ or $\lim_{z\nontangtou x_0}f(z)$. Similarly, we say that  $f$ has a nontangential limit $\zeta$ at $\infty$ if $g(z)\coloneq f(-1/z)$ has the nontangential limit $\zeta$ at $0$. For further information when $f$ is analytic, see Section \ref{sec:boundary}.

In order to state our results we also need to introduce some analytic functions which are often used in free probability. First, we define the \emph{Cauchy transform} of $\mu \in \Prob(\RR)$ as
      \begin{align}\label{eq:Cauchy}
          G_\mu(z)\coloneqq\int_{\RR}\frac{1}{z-x}\di{\mu(x)},\qquad z\in\CC\setminus\RR.
      \end{align}
      We often restrict the domain of $G_\mu$ to either $\Hup$ or $\CC^-$ depending on the context. Since the values on $\Hup$ and $\CC^-$ are just the reflection with respect to $\RR$, the two selections of domains are essentially the same. 
We will also need the following related transforms 
 \begin{align}
F_\mu(z)\coloneqq\frac{1}{G_\mu (z)},\qquad \psi_\mu(z)&\coloneqq\frac{1}{z}G_\mu\left(\frac1{z}\right)-1, \qquad
          \eta_\mu(z)\coloneqq1-\frac{z}{G_\mu\left(1/z\right)}. \label{eq:F_psi_eta}
      \end{align}
The last one will be referred to as the \emph{$\eta$-transform} of $\mu$. For further information on these transforms, see Section \ref{sec:Cauchy}.

\subsection{$S$-transform}\label{sec:S-transform}
As mentioned above, an important analytic tool for computing the free multiplicative convolution of two probability measures is Voiculescu's $S$-transform. It was introduced in \cite{V1987multiplication} for
distributions with nonzero mean and bounded support and was further studied by Bercovici and Voiculescu
\cite{BV1993free} in the case of probability measures on $\RR_{\ge0}$  with unbounded support. 
Later, Raj Rao and Speicher \cite{RS07} defined an $S$-transform in the case of measures having zero mean and all finite moments. Their approach is combinatorial by the use of Puiseux series and cannot be extended to general measures. 
Arizmendi and P\'{e}rez-Abreu \cite{AP2009transform} considered the case of symmetric measures with unbounded support;  we shall mention that although the approach was analytical, the authors relied heavily on an identity relating the Cauchy transforms of a symmetric measure and of its pushforward to the nonnegative real line by the map $x\mapsto x^2$, which is not suitable for the general case. There has been an attempt to define an $S$-transform for general probability measures on $\RR$ in Guerrero's PhD thesis \cite{thesisGuerrero}, but some desirable properties, especially the multiplicativity of $S$-transform, remained open.

In this paper we extend the $S$-transform to general probability measures on $\RR$ with possibly unbounded support. We follow an analytic approach that is very different from the previous ones.  The idea comes from the observation in \cite[Lemma 7.1]{AFPU} that for measures $\mu$ supported on the nonnegative real line the identity $G_{\mu^{\addconv 1/t}}(0) = -t S_\mu(-t)$ holds whenever $t\in (0,1-\mu(\{0\}))$; see Section \ref{sec:conv_powers} for the definition of $\mu^{\addconv 1/t}$. 
For our purposes, 
it is more useful to consider the reciprocal of the $S$-transform, which is called the \emph{$T$-transform}, first introduced in \cite{dykema2007multilinear}. We define the $T$-transform of any probability measure $\mu$ on $\RR$ by 
\begin{align}\label{eq:T-transform}
T_\mu(u) \coloneq u \lim_{\substack{z\to0\\z \in \CC^-}}F_{\mu^{\addconv (-1/u)}} (z), \qquad -1<u<0. 
\end{align}
As we see in Section \ref{sec:T} this limit  exists finitely for all $u$, and is nonzero if and only if $\mu(\{0\})-1 < u <0$. We then define the $S$-transform $S_\mu \coloneq 1/T_\mu$ on $(\mu(\{0\})-1,0)$, which is consistent with the definitions in the literature for restricted classes of $\mu$. An advantage of this definition is that $F_{\mu^{\addconv t}}, t>1$, is already well studied in the literature, in particular in \cite{BB2005partially}. We will establish the following. 

\begin{theorem} \label{thm:T}
\begin{enumerate}
\item\label{itemT0} For any $\mu \in \Prob(\RR)$, $T_\mu$ is a continuous map from $(-1,0)$ into $\Hup\cup\RR$.  
\item\label{itemT1} For any  $\mu \in \Prob(\RR)$ and $\nu \in \Prob(\RR_{\ge0})$ we have 
\[
T_{\mu\multconv \nu}(u) = T_\mu(u) T_\nu(u), \qquad  u \in(-1,0).
\] 
\item\label{itemT2} If $\lambda, \mu \in \Prob(\RR)$ and  $T_\lambda = T_\mu$ on $(- \ee,0)$ for some $0<\ee <1$, then $\lambda =\mu$. 
\item \label{itemT3} If $\mu\ne \delta_0$ then 
\[
\psi_\mu\left(\frac{u}{1+u} S_\mu(u)  \right) = u,\qquad u \in (\mu(\{0\})-1, 0).  
\]
In case $S_\mu(u) \in \RR$, the left hand side is to be interpreted as the nontangential limit of $\psi_\mu(z)$ as $z\nontangtou \frac{u}{1+u} S_\mu(u)~(z\in \Hup)$. 
\item\label{itemT4} If $\mu_n, \mu \in \Prob(\RR)~(n \in \NN)$ and $\mu_n $ converges to $\mu$ weakly, then $\lim_{n\to\infty } T_{\mu_n}(u) = T_\mu(u)$ for all $0<u<1$. 
\item\label{itemT5} Assume that $\mu_n \in \Prob(\RR)~(n \in \NN)$ and $T_{\mu_n}$ converges to a $\CC$-valued function $T$ pointwisely on $(-\ee,0)$ for some $\ee\in(0,1)$. Then $\mu_n$ converges weakly to some $\mu \in \Prob(\RR)$ and $T_\mu = T$ on $(-\ee,0)$. 
\end{enumerate}
\end{theorem}
The proof will be derived in Section \ref{sec:T} since it uses some properties of subordination functions developed in Sections  \ref{sec:ex_of_subord} and \ref{sec:subord2}. Although the above results are natural extensions of the known ones, we shall mention concerning \ref{itemT0} that $T_\mu$ is not necessarily analytic in $(-1,0)$ by contrast to the cases of positive and symmetric measures; see Example \ref{ex:SS}.  
The above convergence results \ref{itemT4} and  \ref{itemT5}, together with multiplicativity  \ref{itemT1}, give an alternative complex-analytic proof of the following weak continuity of free convolution, known e.g.~in \cite[Corollary 5.3.35]{AGZ2010introduction}.

\begin{corollary} \label{cor:wconti}
Free multiplicative convolution $\multconv\colon \Prob(\RR)\times \Prob(\RR_{\ge0})\to \Prob(\RR)$ is continuous with respect to weak convergence. 
\end{corollary}

\subsection{Subordination functions}

A crucial concept for studying free convolutions is subordination relations for the Cauchy transform or $\eta$-transform, which were initially established in \cite{V1993analogue} for the free additive convolution of compactly supported probability measures and later extended to general cases \cite{B1998process, V2000coalgebra,V2002analytic,BSTV2015operator,BMS2017analytic}. The original approach used in the proofs was subsequently replaced by purely analytic methods  \cite{BB2007new,CG2011arithmetic}. In particular, these subordination functions can be characterized as fixed points (more precisely, the \emph{Denjoy--Wolff points}) of certain parametrized analytic functions. This observation plays a pivotal role in deriving regularity properties such as the Lebesgue decomposition of free convolution.

Our second main result is  the existence of subordination functions for $\eta_{\mu\multconv\nu}$, when $\mu\in\Prob(\RR)\setminus\{
    \delta_0\}$ and $ \nu\in\Prob(\RR_{\geq0})\setminus\{\delta_0\}$. For probability measures with compact support, it has already been established under a more general operator-valued setting \cite{BSTV2015operator}.

\begin{theorem}
    \label{thm:subordination}
    Consider any two probability measures $\mu\in\Prob(\RR)\setminus\{
    \delta_0\}, \nu\in\Prob(\RR_{\geq0})\setminus\{\delta_0\}$, and their $\eta$-transforms $\eta_\mu, \eta_\nu$. Then, there exists a unique pair of analytic functions $\omega_1\colon\Hup\to\Hup$ and $\omega_2\colon\Hup\to\CC\setminus[0,+\infty)$ satisfying the following:
    \begin{enumerate}
      \item $\eta_{\mu\multconv\nu}(z) =\eta_\mu(\omega_1(z))=\eta_\nu(\omega_2(z))= \omega_1(z)\omega_2(z)/z$,
      \item\label{item:sub2} $\arg z\leq \arg \omega_2(z)\leq\arg z+\pi$, or equivalently, $\omega_2(z)/z$ is a self-map of $\Hup$, 
      \item\label{item:sub3} $\nontanglim_{z\to 0}\omega_1(z)=\nontanglim_{z\to 0}\omega_2(z)=0$. 
  \end{enumerate}
\end{theorem}
The proof is given in Section \ref{sec:ex_of_subord}. 
For measures $\mu \in \Prob(\RR)$ and $\nu \in \Prob(\RR_{\ge0})\setminus\{\delta_0\}$, one can see from Theorem  \ref{thm:subordination} \ref{item:sub2} and \ref{item:sub3} that
the subordination function $\omega_2$ is the $\eta$-transform of a unique probability measure on $\RR$. Let us denote this probability measure by $\mathbb{B}_{\nu}(\mu)$ (if $\mu=\delta_0$ then we simply set $\mathbb{B}_{\nu}(\mu)\coloneq\delta_0$). On the other hand, $\omega_1$ need not be an $\eta$-transform of a probability measure, see Remark  \ref{eq:omega_1_not_eta}. 

Regarding $\mathbb{B}_{\rho}(\cdot)$ as a self-map $\mathbb{B}_{\rho}\colon\Prob(\RR)\to \Prob(\RR)$ one can show  remarkable properties  which generalize the ones known for the Belinschi--Nica semigroup  \cite{BN2008remarkable}. This was done in \cite{AH2016free} when the domain of the map $\mathbb{B}_\rho$ is restricted to $\Prob(\RR_{\ge0})$. 
With the subordination functions above in hand, now we can show that the main results in \cite{AH2016free} are valid without any restriction as we state in the following theorems. 
 Recall from \cite{Fra09} that, for measures $\rho,\sigma \in \Prob(\RR_{\ge0})\setminus\{\delta_0\}$,  their \emph{multiplicative monotone convolution} $\rho \circlearrowright \sigma \in  \Prob(\RR_{\ge0})\setminus\{\delta_0\}$ is characterized by 
\begin{equation}\label{eq:monotone}
    \eta_{\rho \circlearrowright \sigma} = \eta_\rho \circ \eta_\sigma. 
\end{equation}
The following two results will be proved in Section \ref{sec:subordination_hom}. 
\begin{theorem} 
\label{thm:subordination_hom}
For $\mu \in \Prob(\RR)$, $\nu \in \Prob(\RR_{\ge0})$, $\rho,\sigma \in \Prob(\RR_{\ge0})\setminus\{\delta_0\}$,  we have
\begin{align}
&\mathbb{B}_{\rho}(\mathbb{B}_\sigma(\mu)) = \mathbb{B}_{\sigma \circlearrowright \rho}(\mu),  \label{eq:hom2} \\
&\mathbb{B}_{\rho}(\mu \multconv \nu) = \mathbb{B}_{\rho}(\mu) \multconv \mathbb{B}_{\rho}(\nu).  \label{eq:hom}
\end{align}
\end{theorem}


\begin{theorem} 
\label{thm:pde_for_semigroup}
Let $\mu \in\Prob(\RR)$ and $(\nu_t)_{t\ge0} \subseteq \Prob(\RR_{\ge0})$ be a weakly continuous $\circlearrowright$-convolution semigroup with $\nu_0=\delta_1$. Let $z K(z)$ be the infinitesimal generator of the one-parameter compositional semigroup $(\eta_{\nu_t})_{t\ge0}$, i.e., $z K(z) \coloneq \left. \frac{\di{}}{\di t}\right|_{t=0} \eta_{\nu_t}(z)$. Let $\omega=\omega(t, z) \coloneq \eta_{\mathbb{B}_{\nu_t}(\mu)}(z)$. Then we have
 \begin{equation}\label{eq:omega1}
 \frac{\partial \omega}{\partial t} = K(\omega) \left( -\omega + z \frac{\partial \omega}{\partial z} \right),  
 \end{equation}
 or equivalently for $H = H(t,z) \coloneq z - F_{\mathbb{B}_{\nu_t}(\mu)}(z) = z \omega(t, 1/z)$ we have 
 \begin{equation*}\label{eq:H}
  \frac{\partial H}{\partial t}  =  - z K \left( \frac{H}{z}\right)   \frac{\partial H}{\partial z}. 
 \end{equation*}
\end{theorem}

\begin{remark} The infinitesimal generator $zK(z)$ exists by Berkson and Porta's theorem  \cite{BP1978semigroups}. The class of  functions $K(z)$ for compactly supported $(\nu_t)_{t\ge0}$ was characterized by Bercovici  \cite[Theorem 3.8]{B2005multiplicative}. In the general case, a characterization of $K$ is unknown. 
\end{remark}

\subsection{Boolean and free stable laws}

We now discuss the application of $S$-transform to free and Boolean stable laws, which were introduced in \cite{BV1993free} and \cite{SW1997Boolean}, respectively. Boolean stable laws surprisingly connect different kinds of convolutions and distributions. 
We consider the Cauchy transform $G_\mu$ as a function on $\CC^-$, so that its compositional inverse map $G_\mu^\inv(z) = F_\mu^\inv(1/z)$ can be defined on a subdomain $\Gamma_{0,\theta} \cap \{|z|<r\}$ of the upper half-plane  for some $\theta \in(0,\pi/2)$ and $r>0$ that depend on $\mu$; see \cite{BV1993free} for details. 
Let $C_\mu$ be the \emph{free cumulant transform} of $\mu$ defined by 
\[
C_\mu(z) = zG^\inv_\mu(z)-1,  \qquad z \in \Gamma_{0,\theta}\cap\{|z|<r\}. 
\]
Note that the \emph{Voiculescu transform} $\varphi_\mu$ in  \cite{BV1993free} and $R$-transform $R_\mu$ in \cite{V1986addition} are related to $C_\mu$ as $C_\mu(z)= z \varphi_\mu(1/z) =z R_\mu(z)$. 

Let $\mathfrak{A}$ be the set of admissible parameters
\[
\mathfrak{A}\coloneq((0,1] \times [0,1]) \cup \{(\alpha,\rho): 1 < \alpha \le 2, 1-1/\alpha \le \rho \le 1/\alpha \}. 
\]
We consider the \emph{free stable law $\mathbf{f}_{\alpha,\rho}$} and \emph{Boolean stable law $\mathbf{b}_{\alpha,\rho}~((\alpha,\rho) \in \mathfrak{A})$}  determined by 
\begin{equation}\label{eq:stable}
C_{\mathbf{f}_{\alpha,\rho}} (z) =  \eta_{\mathbf{b}_{\alpha,\rho}}(z) = -(e^{-\iu\rho\pi}z)^{\alpha},  \qquad  z \in \Hup,  
\end{equation}
where the power function $w\mapsto w^\alpha$ is defined on $\CC\setminus(-\infty,0]$ as the principal branch. The parameter $\alpha$ is called the stability index. The parameter $\rho$ stands for the asymmetry and appears as $\rho = \mathbf{f}_{\alpha,\rho} ([0,+\infty)) = \mathbf{b}_{\alpha,\rho} ([0,+\infty))$. Moreover, $\mathbf{f}_{\alpha,1-\rho}$ is the pushforward of $\mathbf{f}_{\alpha,\rho}$ by the reflection $x\mapsto -x$, and the same holds for $\mathbf{b}_{\alpha,1-\rho},\mathbf{b}_{\alpha,\rho}$. 
Except for $\mathbf{f}_{1,0}=\delta_{-1}$ and $\mathbf{f}_{1,1}=\delta_1$, the free stable laws are all Lebesgue absolutely continuous and the densities are investigated in  \cite{BP1999stable,D2011kanter,HK2014free,HSW20}.  The Boolean stable laws are investigated in  \cite{AH2014classical,AH2016classical,HS2015unimodality} in details. According to  \cite{HS2015unimodality}, for $0< \alpha < 1$ and $0 \le \rho \le1$, $\mathbf{b}_{\alpha,\rho}$ is Lebesgue absolutely continuous and its density function is given by 
\begin{equation}\label{eq:density_boole}
q_{\alpha,\rho}(x)\mathbf{1}_{(0,+\infty)}(x) + q_{\alpha,1-\rho}(-x)\mathbf{1}_{(-\infty,0)}(x), 
\end{equation}
where $q_{\alpha,\rho}(x)$ is the function
\begin{equation*}
q_{\alpha,\rho}(x) = \frac{\sin(\alpha\rho\pi)}{\pi}\cdot \frac{x^{\alpha-1}}{x^{2\alpha} + 2x^\alpha\cos(\alpha\rho\pi) +1}, \qquad x>0. 
\end{equation*}
Note that the definition of $\mathbf{b}_{1,\rho}$ differs from that of \cite{HS2015unimodality}; in our definition $\alpha \mapsto \mathbf{b}_{\alpha,\rho}$ is weakly continuous at $\alpha=1$, and consequently, $\mathbf{b}_{1,\rho}$ is a Cauchy distribution for $0<\rho <1$, $\mathbf{b}_{1,1}$ is the trivial measure $\delta_1$ and $\mathbf{b}_{1,0}$ is $\delta_{-1}$.  Finally, it holds that $\mathbf{f}_{1,\rho} = \mathbf{b}_{1,\rho}$.

The following formulas generalize results in \cite[Theorem 4.14]{AH2016classical} (reformulations of results in  \cite{AH2014classical,AP2009transform,BP1999stable}) where the asymmetry parameter $\rho$ was limited to the three values $0, \frac1{2}$ and $1$. The proof is just to use the $S$-transform, see Section \ref{sec:boole1}.

\begin{theorem}\label{thm:reproducing} For any $(\alpha,\rho) \in \mathfrak{A}$ and $0< \beta \le 1$ we have
\begin{align*}
&\mathbf{f}_{\alpha,\rho} \multconv (\mathbf{f}_{\beta,1})^{\multconv \frac{1}{\alpha}} = \mathbf{f}_{\alpha\beta,\rho},   \\
&\mathbf{b}_{\alpha,\rho} \multconv (\mathbf{b}_{\beta,1})^{\multconv \frac{1}{\alpha}} = \mathbf{b}_{\alpha\beta,\rho}.  
\end{align*}
Note that $(\mathbf{f}_{\beta,1})^{\multconv \frac{1}{\alpha}} = \mathbf{f}_{\frac{\alpha\beta}{1-\beta+\alpha\beta},1}$ and $(\mathbf{b}_{\beta,1})^{\multconv \frac{1}{\alpha}} = \mathbf{b}_{\frac{\alpha\beta}{1-\beta+\alpha\beta},1}$ that can be easily confirmed by calculating the $S$-transforms (see Section \ref{sec:conv_powers} for the definition of $\multconv {\frac{1}{\alpha}}$). 
\end{theorem}
Yet more companion identities that connect Boolean, free and even classical stable laws will be given in Section \ref{sec:boole1}. 

The next result connects the free multiplicative convolution to the classical multiplicative convolution $\circledast$. To be precise, $\mu \circledast \nu$ is the law of the product $XY$, where $X$ and $Y$ are $\RR$-valued classical independent random variables having distributions $\mu$ and $\nu$, respectively. Moreover, we introduce the notation $\nu^\alpha~(\alpha>0)$ for the law of $Y^\alpha$ when $Y\ge0$ has distribution $\nu$. If $Y>0$, i.e., $\nu$ does not have an atom at 0, then we also allow $\alpha$ to be any real number. Under  the notation above, the measures of the form $\mathbf{b}_{\alpha,\rho} \circledast( \nu^{\frac1{\alpha}})$ will be referred to as \emph{Boolean stable mixtures}. In \cite[Proposition 4.3]{AH2016classical} it was shown that 
\begin{equation}\label{eq:eta_b}
\eta_{\mathbf{b}_{\alpha,\rho} \circledast (\nu^{\frac1{\alpha}})} (z) = \eta_\nu(-(e^{-\iu\rho\pi}z)^\alpha) = \eta_\nu(\eta_{\mathbf{b}_{\alpha,\rho}}(z)). 
\end{equation} 
Our $S$-transform allows us to prove the following remarkable formula that generalizes \cite[Theorem 4.5]{AH2016classical} to arbitrary $(\alpha,\rho)\in \mathfrak{A}$. The proof will be given in Section \ref{sec:boole1}.

\begin{theorem}\label{thm:mixture} Let $(\alpha,\rho) \in\mathfrak{A}$ and $\nu\in \Prob(\RR_{\ge0})$ be such that $\nu^{\multconv \frac1{\alpha}}$ exists as a probability measure on $\RR_{\ge0}$ (see Section \ref{sec:conv_powers} for the precise meaning). Then we have
\begin{equation*}\label{eq:identity}
\mathbf{b}_{\alpha,\rho} \circledast( \nu^{\frac1{\alpha}}) = \mathbf{b}_{\alpha,\rho} \multconv (\nu^{\multconv\frac1{\alpha}}). 
\end{equation*}
\end{theorem}

We also give a characterization of Boolean stable mixtures in terms of a certain symmetry of $S$-transform, and investigate Boolean stable mixtures as compound free Poisson distributions, see Sections \ref{sec:characBoole} and  \ref{sec:compound} for details.

\subsection{Regularity of free convolution}\label{sec:R}
Despite many correspondences and analogous structures between classical and free probability, free independence exhibits several mathematical properties that are significantly different from those of classical independence.  One notable example is the regularity of free convolution. 
In the early stages of research, regularity results for free convolution were obtained concerning free entropy, the H\"{o}lder continuity of distribution functions, and the locations and weights of atoms  \cite{V1993analogue,BV1993free,BV1998regularity}. Over the course of time, these findings were extended in various directions, including the Lebesgue decomposition of free convolutions \cite{B2005complex,B2006note, B2008lebesgue,BW2008freely, B2014infty, J2021regularity,BBH2023regularity}, studies on the support of free convolution  \cite{H2015supports, BES2020support, MS2022support, M2024density}, and investigations into free L\'{e}vy processes  \cite{HW2022regularity}.

In Section \ref{sec:cont_ext}, we show that these regularity results related to the Lebesgue decomposition also hold for the product of positive and self-adjoint freely independent random variables.

\begin{theorem}
    \label{thm:leb_decomp}
    Suppose that $\mu \in \Prob(\RR)$ and $\nu \in \Prob(\RR_{\geq0})$ are  nondegenerate.
    \begin{enumerate}
        \item\label{item:leb_decomp1} $\omega_1$ and $\omega_2$ extend to continuous maps from $(\Hup\cup\RR)\setminus\{0\}$ into $\CC\cup\{\infty\}$.
        \item\label{item:leb_decomp2} Let $a \in \RR \setminus \{0\}$. We have $(\mu \multconv \nu) (\{a\}) > 0$ if and only if there exist $b, c \in \RR$ such that $a=bc$ and $\mu(\{b\})+\nu(\{c\})>1$. In this case, we have $(\mu\multconv\nu)(\{a\})=\mu(\{b\})+\nu(\{c\})-1$.
        \item\label{item:leb_decomp3} $(\mu\multconv\nu)(\{0\})=\max\{\mu(\{0\}),\nu(\{0\})\}$.
        \item\label{item:leb_decomp4} The singular continuous part of $\mu\multconv\nu$ is zero.
        \item\label{item:leb_decomp5} The absolutely continuous part of $\mu\multconv\nu$, denoted by $(\mu\multconv\nu)^{\acp}$, is nonzero, and its density function $\frac{\di{(\mu\multconv\nu)^{\acp}}}{\di{x}}$ can be chosen to be analytic  on $\RR\setminus\{0\}$ wherever it is positive and finite.
    \end{enumerate}
\end{theorem}

\begin{remark}\label{rem:regularity}
    Theorem  \ref{thm:leb_decomp} \ref{item:leb_decomp2} and \ref{item:leb_decomp3} have already been established in \cite[Corollary 6.6]{ACSY2024universality} by very different methods, but they are included here to exhibit that the general case can also be derived from the existence of subordinations functions and $T$-transforms. In particular, \ref{item:leb_decomp3} is an easy consequence of Theorem \ref{thm:T} \ref{itemT1} and the fact that $T_\mu(u)\ne 0$ if and only if $ \mu(\{0\})-1 < u <0$; see Proposition \ref{prop:omega} and discussion afterwards.
\end{remark}
Note that, in some special cases, part of the theorem above is known or can be proved from known results in the literature: the case of measures on $\RR_{\ge0}$  \cite{BV1993free,B2003Atoms,B2005complex,J2021regularity}; the case when $\mu$ is symmetric because the pushforward map $x\mapsto x^2$ can reduce the problem to the nonnegative real line \cite{AP2009transform}; the case when $\nu$ is the law of a projection because a formula for such a free multiplicative convolution exists \cite[Exercise 14.21]{NS2006lectures} and regularity properties follow from the work \cite{BB2004atoms,BB2005partially}; the case when $\nu$ is a Marchenko--Pastur law, because the free multiplicative convolution is then a compound free Poisson distribution and regularity properties can be retrieved from \cite{BB2004atoms,BB2005partially,HS2017unimodality}.

 \subsection*{Outline of the paper }The rest of the paper is organized as follows. In Section \ref{sec:pre}, we provide an overview of fundamental results in free probability and introduce analytic tools that are extensively used in the subsequent arguments. Section \ref{sec:ex_of_subord} establishes the existence of subordination functions (Theorem \ref{thm:subordination}) with the aid of classical function theory. In Section \ref{sec:subord2}, we prove the two remarkable properties of the second subordination functions in Theorems \ref{thm:subordination_hom} and a differential equation in Theorem \ref{thm:pde_for_semigroup}. Using subordination functions, we establish Theorem \ref{thm:T} in Section \ref{sec:T}, which extends the use of the $S$-transform to arbitrary probability measures. Section \ref{sec:S} investigates several applications of these results to stable laws and shows Theorems \ref{thm:reproducing},  \ref{thm:mixture} and some more results. Section \ref{sec:cont_ext} is devoted to the study of regularity properties of free multiplicative convolution and provides the proof of Theorem \ref{thm:leb_decomp}. In Appendix \ref{app:ext_at_origin}, we study the existence of a bounded and continuous  version of the absolutely continuous part of free  convolution. Finally, we present explicit formulas for free convolution of two-point measures in Appendix \ref{app:specific_example}.

\section{Preliminaries}
\label{sec:pre}

Section \ref{sec:basic_notation} has provided minimal tools and notation, mainly from complex analysis, needed to state the main results. Here we introduce more needed to prove the main results.

\subsection{Cauchy transform and $\eta$-transform}\label{sec:Cauchy}

To provide the analytic description of free convolution, we collect some fundamental results on the Cauchy transform \eqref{eq:Cauchy} and $\eta$-transform in \eqref{eq:F_psi_eta} that will be used in the subsequent analysis.

 The Cauchy transform allows us to investigate properties of measures through analytic methods. More precisely, there is a one-to-one correspondence between probability measures and a certain class of analytic functions from $\Hup$ to $\CC^- $ (see Proposition \ref{prop:characF} below). To recover the original measure, one needs to examine the behavior of its Cauchy transform near the real line in the nontangential way. The following results are referred to as the \emph{Stieltjes inversion}.   
For a proof, see, for instance, \cite[Lemma 1.10]{B2006note},  \cite[Proposition 8, Chapter 3]{MS2017free} and,  \cite[Theorem 11.6]{S2005trace}. Note that, by Theorem  \ref{thm:Lind} below, $\lim_{\ee\to0^+}G_\mu(x+\iu\ee)=\nontanglim_{z\to x}G_\mu(z)$ if either limit exists.
  
  \begin{theorem}
      \label{thm:ct_nontang}
      Let $\mu$ be a probability measure on $\RR$, and consider its Lebesgue decomposition and we write the singular continuous part as $\mu^{\scp}$ and the absolutely continuous part as $\mu^{\acp}$. 
      \begin{enumerate}
          \item For almost all $x\in\RR$ with respect to $\mu^{\scp}$, the nontangential limit of $G_\mu(z)$ at $x$ exists and equals $\infty$.
          \item\label{item:G2} For any $x\in\RR$, we have $\mu(\{x\}) = \nontanglim_{z\to x} (z - x) G_\mu(z)$.
          \item For almost all $x\in\RR$ with respect to $\mu^{\acp}$, we have $\pi\frac{\di{\mu_{\acp}}}{\di{x}}=-\nontanglim_{z\to x}\Im G_\mu(z)$.
      \end{enumerate}
  \end{theorem}
  
  Moreover, the weak convergence of probability measures translates into the convergence of their Cauchy transforms (or $\eta$-transforms via \eqref{eq:F_psi_eta}), which can be easily verified by Vitali's theorem or Montel's theorem. For a detailed proof, see e.g.~\cite[Proposition 3.7]{HHM}. 
  \begin{theorem}
      \label{thm:ct_unifconv}
      Suppose that $\mu_n,\mu \in \Prob(\RR), n\in\NN$, and $A$ is a subset of $\Hup$ containing an accumulation point. Then the following are equivalent. 
 \begin{enumerate}     
    \item $\{\mu_n\}_{n=1}^\infty$ converges weakly to $\mu$, 
    \item $G_{\mu_n}(z)$ converges to $G_\mu(z)$ uniformly on every compact subset of $\Hup$, 
    \item $G_{\mu_n}(z)$ converges to $G_\mu(z)$ pointwisely on $A$. 
    \end{enumerate}
  \end{theorem}

The following proposition from \cite[Proposition 5.2]{BV1993free} characterizes the $F$-transform of probability measures.

\begin{proposition}\label{prop:characF} Let \( F\colon \Hup \to \Hup\cup\RR \) be an analytic function. Then the following assertions are equivalent.

\begin{enumerate}
    \item There exists a probability measure \( \mu \) on \( \RR \) such that \( F=F_{\mu} \) in \( \Hup \).
    
    \item \label{item:F2}
    $\nontanglim_{z\to\infty} F(z)/z = 1.
    $ 
    
    
    \item\label{item:F3} {\normalfont(Pick--Nevanlinna representation)} There exist \( a \in \RR \) and a finite measure \( \rho \) on \( \RR \) such that
\begin{equation}\label{eq:PN}    
    F(z) = a + z + \int_{-\infty}^{\infty} \frac{1 + tz}{t - z} \rho(\di{t}), \qquad z\in \Hup.
    \end{equation}
\end{enumerate}
\end{proposition}
From part \ref{item:F3} in the previous proposition it is clear that $\Im F_\mu(z)\geq \Im z$ for $z\in \Hup$ and by symmetry $\Im F_\mu(z)\leq \Im z$ for $z\in\CC^-$. The inequalities are strict unless $\mu=\delta_a$ for some $a\in \RR$.
  
Now we present fundamental properties of the $\eta$-transform. Some of the following results have already been proved in \cite[Proposition 3.2]{AH2016classical}.

  \begin{proposition}\label{prop:prop_of_eta}
     For any probability measure $\mu\in\Prob(\RR)$ such that $\mu\neq\delta_0$, we have the following:
      \begin{enumerate}
          \item\label{item:eta1} $\eta_\mu$ is analytic on $\Hup$.
          
          \item\label{item:eta2} For any $z\in\Hup$, $0<\arg z\leq \arg\eta_\mu(z)\leq \arg z+\pi$. In particular, $\eta_\mu(\Hup)\subseteq \CC\setminus[0,+\infty)$. Furthermore,  $\arg z=\arg\eta_\mu(z)$ for some $z\in\Hup$ if and only if $\mu=\delta_a$ for some $a>0$, and $\arg\eta_\mu(z)=\arg z+\pi$ for some $z\in\Hup$ if and only if $\mu=\delta_a$ for some $a<0$.

       \item\label{eq:eta_at_infty_zero} 
       $\nontanglim_{z\to\infty}\eta_\mu(z)=1-1/\mu(\{0\})$ and $\nontanglim_{z\to0}\eta_\mu(z)=0.$  %
       
          
      \end{enumerate}
      
   Moreover, we have the following converse statement: if an analytic map $\eta \colon \Hup \to \CC \setminus [0,+\infty)$ satisfies $\arg z \le \arg \eta(z) \le \arg z + \pi$ and $\nontanglim_{z\to0}\eta(z)=0$, then there is a unique $\mu \in \Prob(\RR) \setminus\{\delta_0\}$ such that $\eta=\eta_\mu$.  
  \end{proposition}


    

      
      
          


  \begin{proof}
     \ref{item:eta1}  follows immediately from the definition of the $\eta$-transform, and  \ref{item:eta2} can be easily seen from the inequality $\Im z\ge\Im F_\mu(z) $ for $z\in\CC^-$. 
       \ref{eq:eta_at_infty_zero} is equivalent to 
        \begin{align}
    \nontanglim_{z\to\infty}\psi_\mu(z)&=\mu(\{0\})-1,\quad \nontanglim_{z\to0}\psi_\mu(z)=0,\label{eq:psi_at_infty_zero}
      \end{align} 
      which can be proved by the dominated convergence theorem. 
      The last converse statement is known in \cite[Proposition 3.2]{AH2016classical}, which follows easily by Proposition \ref{prop:characF} \ref{item:F2} applied to $F(z)\coloneqq z(1-\eta(1/z))$. 
  \end{proof}

 We now turn to the case of measures supported on the nonnegative real line. In this case, the $\eta$-transform is naturally defined on $\CC\setminus[0,+\infty)$. Correspondingly, in order to consider nontangential limits at $0$ and $\infty$,  the domain 
      \begin{align*}
          \Delta_{\theta}\coloneq \{z\in\CC\setminus\{0\}\ \colon\ \theta<\arg z< 2\pi-\theta\} \subseteq \CC\setminus[0,+\infty), \qquad \theta \in(0,\pi), 
      \end{align*}
is a proper replacement of $\Gamma_{0,\theta}$ defined in \eqref{eq:nabla}.   
  \begin{proposition} \label{prop:prop_of_eta0}
      For any probability measure $\nu\in\Prob(\RR_{\geq0})$ such that $\nu\neq\delta_0$, we have the following:
      \begin{enumerate}
          \item $\eta_\nu(z)$ extends to an analytic self-map of $\CC\setminus[0,+\infty)$ such that $\eta_\nu(\conj{z})=\conj{\eta_\nu(z)}$ holds. 
          \item For any $z\in\Hup$, $0<\arg z\leq \arg \eta_\nu(z)< \pi$. In particular, $\eta_\nu(\Hup)\subseteq\Hup$. 
          Moreover, there exists some $z\in\Hup$ such that $\arg z=\arg \eta_\nu(z)$ if and only if $\nu=\delta_a$ for some $a>0$. 
          \item\label{item:eta_positive3} For any $0<\theta<\pi$, we have $\lim_{z\to\infty,z\in\Delta_\theta}\eta_\nu(z)=1-1/\nu(\{0\})$ and $\lim_{z\to0,z\in\Delta_\theta}\eta_\nu(z)=0$.
      \end{enumerate}
  \end{proposition}
  
  \begin{proof}
      The statements (1) and (2) were proved in  \cite[Proposition 2.2]{BB2005partially}, and one can check (3) by the dominated convergence theorem in an analogous way to Proposition   \ref{prop:prop_of_eta}  \ref{eq:eta_at_infty_zero}. 
  \end{proof}

\subsection{Free convolution powers and Boolean convolution powers}
\label{sec:conv_powers}

For any positive integer $n$,  convolution powers with respect to $\addconv$ and $\multconv$ are defined as 
\begin{align*}
    \mu^{\addconv n} \coloneqq \underbrace{\mu\addconv\mu\addconv\cdots\addconv\mu}_{n \text{ times}}, \qquad 
    \nu^{\multconv n}\coloneqq \underbrace{\nu\multconv\nu\multconv\cdots\multconv\nu}_{n \text{ times}}
\end{align*}
for arbitrary $\mu\in\Prob(\RR)$ and $\nu\in\Prob(\RR_{\geq 0})$.
These convolution powers can be interpolated and extended to semigroups $\{\mu^{\addconv t}:t\ge1\}\subset \Prob(\RR)$ \cite{NS1996multiplication,BB2004atoms} and $\{\nu^{\multconv t}: t\ge1\}\subset\Prob(\RR_{\ge0})$  \cite{BB2005partially}. These semigroups are characterized by 
\begin{equation*}
C_{\mu^{\addconv t}} = t C_\mu,  \qquad S_{\nu^{\multconv t}} = S_\nu^t\quad(\nu\ne\delta_0),  
\end{equation*}
where the equalities hold in the intersection of the domains of two functions.

If $\mu\in \Prob(\RR)$ can be written as $\mu=\lambda^{\addconv s}$ for some $\lambda \in \Prob(\RR)$ and $s>1$, then a measure $\mu^{\addconv t}$ exists for all  $t\ge1/s$ because we can define $\mu^{\addconv t}\coloneqq\lambda^{\addconv (ts)}$, and so we obtain a prolongation of the semigroup $\{\mu^{\addconv t}: t\ge1\}$. Formula $C_{\mu^{\addconv t}} = t C_\mu$ still holds for $0<t<1$ as long as $\mu^{\addconv t}$ exists. A similar observation holds for $\multconv$. 
The quantities $\inf\{t>0: \mu^{\boxplus t} \text{ exists in }\Prob(\RR) \}$ and $\inf\{t>0: \nu^{\boxtimes t} \text{ exists in }\Prob(\RR_{\ge0}) \}$ are studied in details in  \cite{BN2008remarkable,AH13semigroup}.

Similarly and more easily, one can define the Boolean convolution powers. Let $\mu^{\uplus n}$ be the $n$-fold Boolean convolution of $\mu\in \Prob(\RR)$, that is, 
\begin{align*}
    \mu^{\uplus n}\coloneqq\underbrace{\mu\uplus\mu\uplus\cdots\uplus\mu}_{n \text{ times}}, 
\end{align*}
 where $\uplus$ denotes the Boolean convolution \cite{SW1997Boolean}. This notion can also be extended to continuous parameters. Actually, we can define $\mu^{\uplus t}$ for all $t\ge0$ \cite{SW1997Boolean} as a unique probability measure on $\RR$ satisfying 
    \begin{align*}
        \eta_{\mu^{\uplus t}}(z)=t\eta_\mu(z),\qquad z\in\Hup.
    \end{align*}

\subsection{Boundary behaviors of general analytic functions}\label{sec:boundary}

  In fact, an analytic function $f\colon\Hup\to\Hup$ has a nontangential limit at $x\in\RR$ if there is a limit $\lim_{z\to x, z\in\Gamma_{x,\theta}} f(z)$ for \emph{some} $\theta\in(0,\pi/2)$. Furthermore, the following stronger statement holds. See  \cite[Theorem~2.20]{CL2004theory} for a proof.
  
  \begin{theorem}[Lindel\"{o}f's theorem]
      \label{thm:Lind}
      Let $f\colon \Hup\to\CC\cup \{\infty\}$ be a meromorphic function such that $(\CC\cup\{\infty\})\setminus f(\Hup)$ contains at least three distinct points. If there exists a continuous map $\gamma\colon[0,1)\to\Hup$ such that the limits $x \coloneq \lim_{t\to1}\gamma(t)\in\RR\cup\{\infty\}$ and $l\coloneq\lim_{t\to1}f(\gamma(t))\in\CC\cup\{\infty\}$ exist, then the nontangential limit of $f$ at $x$ exists and equals $l$.
  \end{theorem}

  Notably, the curve $\gamma$ treated in the theorem does not need to approach the point $x$ in a nontangential way. In other words, regardless of how arbitrarily the curve meanders, as long as it is continuous and approaches the boundary point, there is no problem.

  Although nontangential limits do not necessarily exist everywhere, the following theorem states that it exists at almost all points. See  \cite[Theorem~2.4]{CL2004theory} for a proof.
  
  \begin{theorem}[Fatou's theorem]
      \label{thm:fatou}
      Let  $f\colon\Hup\to\Hup$ be an analytic function. Then, for almost every $x\in\RR$, the nontangential limit of $f$ at $x$ exists in $\Hup\cup\RR\cup\{\infty\}$.
  \end{theorem}
  
Analytic functions are in fact uniquely determined by their nontangential limits on the boundary as stated below. See  \cite[Theorem~8.1]{CL2004theory} for a proof.
  
  \begin{theorem}[Privalov's theorem]
      \label{thm:priv}
      Let $f\colon\Hup\to\CC$ be an analytic function. If there exists a set $F\subseteq\RR$ of positive Lebesgue measure such that the nontangential limit of $f$ exists and equals $0$ at every $x\in F$, then $f\equiv0$.
  \end{theorem}
  
  The next corollary follows immediately from Privalov's theorem.
  
  \begin{corollary}
      Let $f,g\colon\Hup\to\Hup$ be analytic functions. If there exists a set $F\subseteq\RR$ of positive Lebesgue measure such that, for every $x \in F$, the nontangential limits $\nontanglim_{z\to x}f(z)$ and $\nontanglim_{z\to x}g(z)$ both exist in $\Hup\cup\RR\cup\{\infty\}$ and are equal, then $f \equiv g$. 
  \end{corollary}

  \begin{remark}
  \label{rem:conf_equiv}
    The above propositions are stated for analytic functions on the upper half-plane, but they can naturally be translated into the results for functions defined on the unit disk or $\CC\setminus[0,+\infty)$. This observation follows from the conformal equivalences of these domains and is implicitly used in the following arguments.
  \end{remark}

 Lastly, we present an additional boundary behavior of analytic functions on the upper half-plane. For a proof, see \cite[Lemma 2.13]{B2008lebesgue}. It serves as a crucial tool for establishing regularity properties for free convolution, which will be explored in Section \ref{sec:cont_ext}. 

  \begin{theorem}[Part of the Julia--Wolff--Carath{\'e}odory theorem for $\Hup$]
  \label{thm:JC}
      Let $f$ be an analytic self-map of $\Hup$ and let $a\in\RR$. If 
     $\nontanglim_{z\to a}f(z)=c\in\RR$ 
      then 
      \begin{align*}
        \nontanglim_{z\to a}\frac{f(z)-c}{z-a}=\liminf_{z\to a}\frac{\Im f(z)}{\Im z} \in \CC\cup\{\infty\}. 
      \end{align*}
  \end{theorem}

\subsection{Denjoy--Wolff point}  
  A key concept is the notion of a generalized absorbing fixed point, known as the Denjoy--Wolff point, which characterizes the asymptotic behavior of the iterates of an analytic self-map. The results are mostly stated for the unit disk $\DD\coloneqq\{z\in\CC\colon|z|<1\}$ centered at the origin. 

  \begin{definition}[Denjoy--Wolff point]
    Let $f\colon \DD\to\cl{\DD}$ be an analytic function. A point $w\in\cl{\DD}$ is called a Denjoy--Wolff point (DW-point for short) of $f$ if one of the following holds:
    \begin{enumerate}
      \item $|w|<1$ and $f(w)=w$, 
      \item $|w|=1$, $\nontanglim_{z\to w} f(z)=w$, and
      $\nontanglim_{z\to w}\frac{f(z)-w}{z-w}\leq 1.$ 
    \end{enumerate}
  \end{definition}
  \begin{remark}\label{rem:DW}
      One can also formulate the above definition for the upper half-plane. Specifically, we will use the following fact:  an analytic map $f\colon \Hup \to \Hup\cup\RR$ has a DW-point at $\infty$ if and only if $\nontanglim_{z\to\infty}f(z)/z \in [1,+\infty)$. 
     This can be proved by identifying $\Hup$ with $\DD$ and using the fact $\nontanglim_{z\to\infty}f'(z) = \nontanglim_{z\to\infty} f(z)/z \in [0,+\infty)$ that can be proved e.g.~by the Pick--Nevanlinna representation  \eqref{eq:PN} in which the linear term $z$ needs to be multiplied by some $b\ge0$.   
  \end{remark}

The following fact can be extracted from  \cite[Corollary 1.2.4, Theorems 1.7.3 and 1.8.4]{BCDM20}.

  \begin{theorem}[Denjoy--Wolff theorem] 
      \label{thm:denjoy_wolff}
      Any analytic function $f\colon\DD\to\cl{\DD}$, which is not the identity map, has a unique DW-point. If $w\in\DD$ is the DW-point of $f$, then $|f^\prime(w)|\leq 1$, where equality holds if and only if $f$ is a conformal automorphism. Moreover, if $f$ is not a conformal automorphism or a constant function taking a value in $\partial \DD$, then $\{f^{\circ n}\}_{n=1}^\infty$ converges to the DW-point of $f$ uniformly on each compact subset of $\DD$. 
      \end{theorem}

If the DW-point is located on the boundary, any horodisk tangent at the point is known to be invariant under the function. A similar phenomenon can be observed when the DW-point in the interior of the unit disk, which is known as the Schwarz lemma. 
Relying on this invariance, Heins proved the continuity of the DW-point with respect to a sequence of analytic maps.  This continuity result, combined with the previous theorems, will be essential in the analysis of subordination functions and their boundary behavior below. 

\begin{theorem}[{\cite[Theorem P.3]{H1941iteration}}]
    \label{thm:conv_of_dwp}
    Let $f_n, f\colon\DD\to\cl{\DD}$ be analytic such that $f_n$ converges pointwisely to $f$.  If $f$ is not the identity map, then the DW-point of $f_n$ converges to that of $f$ as $n\to\infty$. 
\end{theorem}

\begin{remark} 
\begin{enumerate}
\item[(i)] 
It is worth noting that an alternative proof of Heins' theorem is provided in  \cite[Theorem 1.1]{BBH2022convergence} in case $f( \DD)\subseteq \DD$.  

\item[(ii)]
By selecting and fixing a conformal isomorphism $\DD \simeq \Hup$, we can translate the result to holomorphic functions $f_n, f \colon \Hup \to \Hup \cup \RR \cup \{\infty\}$ (regarded as functions into the Riemann sphere). Note that if a holomorphic function $f\colon \Hup \to \Hup \cup \RR \cup \{\infty\}$ takes a value in the boundary $\RR \cup \{\infty\}$ then $f$ must be a constant function by the maximum principle. Actually, later we will often use the result in case $f$ takes a (constant) value in $\RR \cup \{\infty\}$. 
\end{enumerate}

\end{remark}

\section{Existence of subordination functions} 
\label{sec:ex_of_subord}

In this section, we first prove Theorem  \ref{thm:subordination} in the bounded case, and subsequently extend the result to the unbounded case. An important step is to find suitable analytic self-maps with parameter $z\in \Hup$, whose DW-points give the desired subordination functions of $z$. 

\subsection{Constructing parametrized analytic self-maps}

As in previous studies (e.g.,  \cite{BB2007new,J2021regularity}), the desired subordination functions appear as fixed points of parametrized analytic functions defined on the upper half-plane $\Hup$. The following theorem guarantees the general existence, uniqueness, and analyticity of such fixed points, which is utilized extensively throughout this section.

  \begin{theorem}[Theorem 2.4 of  \cite{BB2007new}]
      \label{thm:dwfix}
      Let $\Omega$ be an open subset of $\Hup$ and let $f\colon\Omega\times\Hup\to\Hup$ be an analytic function. Suppose that for some $z\in\Omega$, the analytic self-map $f_z(w)\coloneqq f(z,w)$ of $\Hup$ is not a conformal automorphism and has a fixed point in $\Hup$. Then there exists a unique analytic function $\omega\colon\Hup\to\Hup$ satisfying $f(z,\omega(z))=\omega(z)$ for $z\in\Omega$. Moreover, for any $w$, we have $\omega(z)=\lim_{n\to\infty} f_z^{\circ n}(w)$ uniformly on compact subsets of $\Hup$.
  \end{theorem}

We prepare a lemma that will be used to check one of the assumptions of Theorem \ref{thm:dwfix}. 

 \begin{lemma} \label{lem:auto}
  Let $f,g$ be analytic self-maps of $\Hup$. If $f \circ g$ is a conformal automorphism of $\Hup$ then both $f$ and $g$ are conformal automorphisms of $\Hup$. 
  \end{lemma}
  \begin{proof} Let $d(z,w)$ be the pseudohyperbolic distance between $z,w \in \Hup$. It is known that $d(h(z),h(w)) \le d(z,w)$ for any analytic self-map $h$ of $\Hup$, where the equality holds for some $z\ne w$ if and only if $h$ is a conformal automorphism (e.g.,  \cite[Lemma 1.2]{G2006bounded} or  \cite[Theorem 1.3.7]{BCDM20}). Hence, 
  \[
   d(f(g(z)), f(g(z))) \le d(g(z),g(w)) \le d(z,w).
  \]
  By the assumption, the above inequalities are all equalities, so that $f$ and $g$ are conformal automorphisms. 
  \end{proof}

We then construct appropriate parametrized analytic self-maps, inspired by \cite[Theorem 1.27]{B2005complex} and \cite[Theorem 2.2]{BSTV2015operator}.
  
  \begin{proposition}
      \label{prop:not_conf}
We consider the transform
\begin{equation*}
h_\mu(z)\coloneqq\frac{\eta_\mu(z)}{z}=\frac{1}{z}-\frac{1}{G_\mu\left(1/z\right)}. 
\end{equation*}
     Let $\mu\in\Prob(\RR)$ and $\nu\in\Prob(\RR_{\geq0})$ be nondegenerate. Then, the following hold. 
      \begin{enumerate}
          \item\label{item:not_auto1} $h_\mu(\Hup) \subseteq \Hup$.
          \item\label{item:not_auto2} For $z, w \in \Hup$, we have $zh_\mu(w) \in \CC \setminus [0,+\infty)$ and  
          \[
          f_z(w)\coloneqq zh_\nu(zh_\mu(w))\in \Hup.
          \]
          
          \item\label{item:not_auto3} For $z, w \in \Hup$, we have $zh_\nu(zw) \in \Hup$ and \[
          g_z(w)\coloneqq h_\mu(zh_\nu(zw)) \in \Hup.\]
          \item\label{item:not_auto4}  For any $z\in\Hup$, the analytic self-maps $f_z, g_z\colon\Hup\to\Hup$ are not conformal automorphisms.
      \end{enumerate}
  \end{proposition}

  \begin{proof}
      \ref{item:not_auto1}\, This follows immediately from the inequality $\Im z< \Im F_\mu(z)=\Im (1/G_\mu(z)) $ for $z\in\Hup$. 

      \vspace{2mm}
      \noindent
      \ref{item:not_auto2}, \ref{item:not_auto3}\,   We first observe that the fact $h_\mu(w) \in \Hup$ implies  $zh_\mu(w)\in\CC\setminus[0,+\infty)$ for $z,w\in\Hup$, so that the expression $zh_\nu(zh_\mu(w))$ is well-defined. It suffices to show that $k_z(w)\coloneqq zh_\nu(zw) \in \Hup$ for all $z,w \in \Hup$ since it ensures $f_z = k_z \circ h_\mu$ and $g_z=h_\mu\circ k_z$ are self-maps of $\Hup$. 
  
      \begin{enumerate}
          \item[(i)] When $0<\arg zw \leq\pi$, it holds that $\arg zw\leq\arg\eta_\nu(zw)\leq\pi$. Consequently, we have
          \begin{align*}
              \arg k_z(w)&=\arg \frac{\eta_\nu(zw)}{w} \in [\arg z, \pi-\arg w] \subseteq(0,\pi).
          \end{align*}

          \item[(ii)] When $\pi <\arg zw ~(<2\pi)$, it holds that  $\pi\le \arg \eta_\nu(zw)\le\arg zw$. Thus, we obtain
          \begin{align*}
             \arg k_z(w)&=\arg \frac{\eta_\nu(zw)}{w} \in [\pi - \arg w, \arg z] \subseteq(0,\pi). 
          \end{align*}
       
      \end{enumerate}

\vspace{2mm}  \noindent
\ref{item:not_auto4}\, Suppose that $f_z=k_z \circ h_\mu$ is a conformal automorphism of the upper half-plane. By Lemma \ref{lem:auto}, the map $k_z$ must be a conformal automorphism. Recall that the group of conformal automorphisms of $\Hup$ is given by
      \begin{align}
          \Aut(\Hup)=\left\{\frac{az+b}{cz+d}\ \colon a,b,c,d\in\RR,\ ad-bc>0\right\}. \label{eq:auto}
      \end{align}
      Observe that every conformal automorphism takes real values or infinity on the real line. However, letting $w\to 1$, we have $k_z(w)=\eta_\nu(zw)/w\to\eta_\nu(z)\in\Hup$ for $z\in\Hup$, which is a contradiction. The proof for $g_z$ is the same.  
  \end{proof}

  \subsection{Proof of Theorem { \ref{thm:subordination}}: the case of compact support}
  We first show the existence of subordination functions for compactly supported probability measures with nonzero expectation. Although this result has been established in a more general bounded operator-valued setting   \cite{BSTV2015operator}, we provide a simplified proof tailored to our scalar-valued case for the sake of completeness.
  
  \begin{proposition}
      \label{prop:bdd_subordination}
      For any nondegenerate compactly supported measures $\mu\in\Prob(\RR)$ and $\nu\in\Prob(\RR_{\geq0})$, each of which has nonzero expectation, there exist analytic functions $\omega_1\colon\Hup\to\Hup$ and $\omega_2\colon\Hup\to\CC\setminus[0,+\infty)$ satisfying the following at any point $z\in \Hup$:
      \begin{enumerate}
    
        \item $\eta_{\mu\multconv\nu}(z) =\eta_\mu(\omega_1(z))=\eta_\nu(\omega_2(z))= \omega_1(z)\omega_2(z)/z$,
       
        \item\label{item:cond2_bdd} $\arg z\leq \arg \omega_2(z)\leq \arg z+\pi$, or equivalently, $\omega_2(z)/z \in \Hup$, 
      
        \item $\omega_1(z)$ and $\omega_2(z)/z$ are the fixed points in $\Hup$ of the analytic self-maps $f_z$ and $g_z$ defined in Proposition  \ref{prop:not_conf}, respectively. 
    \end{enumerate}
    
\end{proposition}

\begin{remark}
    We really need two different fixed point equations because some $\omega_1$ does not satisfy $\omega_1(z)/z \in \Hup$, and some $\omega_2$ is not a self-map of $\Hup$; see Remark \ref{eq:omega_1_not_eta} for counterexamples.  
\end{remark}

  \begin{proof}
  Since the means of $\mu$ and $\nu$ are nonzero, we have $\eta_\mu^\prime(0) \neq 0 \neq \eta_\nu^\prime(0)=:m_\nu$, and by the inverse function theorem, $\eta_\mu$ and $\eta_\nu$ are univalent in some neighborhood of $0$. Let $\eta_\mu^\inv$ and $\eta_\nu^\inv$ denote their local inverses. In a neighborhood of $0$, we have $z \eta_{\mu \multconv \nu}^\inv(z) = \eta_\mu^\inv(z) \eta_\nu^\inv(z)$. We now define two analytic functions $\omega_1(z)$ and $\omega_2(z)$
      \begin{align*}
          \omega_1(z)\coloneqq\eta_\mu^\inv\circ\eta_{\mu\multconv\nu}(z),\quad\omega_2(z)\coloneqq\eta_\nu^\inv\circ\eta_{\mu\multconv\nu}(z)
      \end{align*}
      around the origin. In particular, they satisfy $z\eta_{\mu \multconv \nu}(z)=\omega_1(z)\omega_2(z)$. Following the calculations in  \cite[Section 2.2.1]{BSTV2015operator}, for sufficiently small $z$, we have 
      \begin{align*}
          f_z(\omega_1(z))&=zh_\nu(zh_\mu(\omega_1(z)))=\frac{\eta_\nu(zh_\mu(\omega_1(z)))}{h_\mu(\omega_1(z))}=\frac{\eta_\nu(z\eta_\mu(\omega_1(z))/\omega_1(z))\omega_1(z)}{\eta_\mu(\omega_1(z))}\\
          &= \frac{\eta_\nu(z\eta_{\mu\multconv\nu}(z)/\omega_1(z))\omega_1(z)}{\eta_{\mu\multconv\nu}(z)}=\frac{\eta_\nu(\omega_2(z))\omega_1(z)}{\eta_{\mu\multconv\nu}(z)}= \omega_1(z),
      \end{align*}
      which implies that $\omega_1(z)$ is a fixed point of $f_z$ in a sufficiently small neighborhood of $0$. Since $\eta_\mu^\inv(z) = \frac{z}{m_\mu}+O(z^2)$  and $\eta_{\mu\boxtimes\nu}(z) = m_{\mu\boxtimes\nu}z +O(z^2) = m_{\mu}m_\nu z +O(z^2)$ as convergent power series expansions, we obtain $\omega_1(z) = m_\nu z + O(z^2)$ as $z\to0$. The fact $m_\nu>0$ yields $\omega_1(\iu y) \in \Hup$ for sufficiently small $y >0$.  Together with Proposition \ref{prop:not_conf}, this implies that the assumptions of Theorem \ref{thm:dwfix} are fulfilled. Hence, $\omega_1(z)$ can be extended to an analytic self-map of $\Hup$ satisfying $\omega_1(z)=f_z( \omega_1(z))$ for $z\in\Hup$. Furthermore, the domain of $\omega_2(z)$ is also extendable to $\Hup$ by the relation $\omega_2(z)=zh_\mu(\omega_1(z))$, which also verifies that $\omega_2(z)/z \in \Hup$. These functions $\omega_1$ and $\omega_2$ satisfy $\eta_\mu(\omega_1(z))=\eta_\nu(\omega_2(z))=\eta_{\mu \multconv \nu}(z)$ and $g_z(\omega_2(z)/z)=\omega_2(z)/z$ in a neighborhood of $0$, and hence on $\Hup$. 
  \end{proof}

\subsection{Proof of Theorem { \ref{thm:subordination}}: the case of unbounded support}

We next discuss the free multiplicative convolution of probability measures with unbounded support. Unfortunately, in this case, the $\eta$-transforms of probability measures are not necessarily analytic around the origin, and the arguments in the bounded case do not work. We circumvent this difficulty by approximating probability measures with those with bounded support in an appropriate manner.

Lemma \ref{lem:ex_of_fp} below plays a crucial role in ensuring that the subordination functions emerge as fixed points (or DW-points) of $f_z$ and $g_z$: as we see in the proof of Theorem \ref{thm:subordination} below, this lemma guarantees the existence of the DW-point of $f_z$ in $\Hup$ for some $z\in\Hup$, thereby confirming one of the assumptions of Theorem \ref{thm:dwfix}. 

Before stating the next lemma, let us observe that for $\nu\ne\delta_0$ the function $\eta_\nu|_{\iu\Hup}$ extends to a continuous and injective function on the closed left half-plane $\cl{\iu\Hup}$, analytic on $\cl{\iu\Hup}\setminus\{0\}$, which maps $\iu\Hup$ onto a Jordan domain symmetric about the real line, and maps $\iu \RR$ onto the boundary of $\eta_\nu(\iu\Hup)$; see \cite[Proposition 6.2]{BV1993free}. In particular, $\eta_\nu|_{\cl{\iu\Hup}}$ has a  compositional inverse map that we denote by $\eta_\nu^\inv$. For the reader's convenience, here is an outline of the proof. By direct calculations we can see that $\Im[\psi_\nu']>0$ on $(\iu\Hup\cap \Hup) \cup \iu(0, +\infty)$ and so $\psi_\nu$ is injective on this region by Noshiro-Warschawski's theorem \cite[Theorem 3.1.3]{BCDM20}, and it takes values in $\Hup$. Also we can easily see that $(-\infty,0] \ni t\mapsto\psi_\nu(t) \in (-1+\nu(\{0\}),0]$ is a strictly increasing bijection. Therefore, $\psi_\nu$ is homeomorphic from $\cl{\iu\Hup\cap \Hup}$ onto its range. By symmetry, the same holds on $\cl{\iu\Hup}$.

  \begin{lemma} Let $\mu\in\Prob(\RR)$ and $\nu\in\Prob(\RR_{\geq0})$ be nondegenerate. 
      \label{lem:ex_of_fp}
      There exists a complex number $w\in\Hup$ satisfying
      \begin{enumerate}
          \item\label{item:cond1} $\eta_\mu(w)\in\eta_\nu(\cl{\iu\Hup}\setminus\{0\})$, and 
          \item\label{item:cond2} $z:= {w\eta_\nu^\inv(\eta_\mu(w))}/{\eta_\mu(w)} \in \Hup$, 
      \end{enumerate}
      where $\eta_\nu^\inv$ stands for the compositional inverse of $\eta_\nu$ described above. 
      The first condition is necessary for the second formula to be well-defined.
  \end{lemma}
 \begin{proof}  
      Since $\lim_{z\to 0,z\in \iu\Hup}\eta_\nu(z)=0$ and $\lim_{z\to \infty,z\in \iu\Hup}\eta_\nu(z)=1-1/\nu(\{0\})$, the boundary of $\eta_\nu(\iu\Hup)$ consists of the smooth open simple curves $\{\eta_\nu(\iu t): t>0\} \subseteq\Hup\cap \iu\Hup$, $\{\eta_\nu(-\iu t): t>0\} \subseteq\CC^-\cap\iu\Hup$, and their common endpoints $0$ and $1-1/\nu(\{0\})$ (in the Riemann sphere if $\nu(\{0\})=0$).
      
      Choose and fix any positive real number $0<r< 1/\nu(\{0\})-1$ and set the disk $D:=\{z: |z| < r\}$ and the first hitting time $s\coloneqq\min\{t>0: \eta_\nu(\iu t) \in \partial D\}$. We also denote the first hitting point as $re^{\iu\phi}:= \eta_{\nu}(\iu s), \phi \in (\pi/2,\pi)$. 
      The three curves $\eta_\nu([0,\iu s])$, $\{re^{\iu\theta}: \pi/2 \le\theta \le \phi\}$ and $[0,\iu r]$ form a Jordan closed curve $\partial E$ that surrounds a bounded Jordan domain  $E\subseteq \Hup$, see Figure \ref{fig:E}. 

\begin{figure}[t]
    \centering
    \begin{tikzpicture}[scale=1.5]
        \begin{scope}
  \clip (0,0) circle[radius=1.3];
  \fill[gray!30]
    (0,0) -- (0,1.5) -- (-2.5,1.5)
    to[out=40,in=240] (-1.3,0.8)
    to[out=60,in=150,looseness=1.2] (-1,1.5)
    to[out=-30,in=120,looseness=1.2] (-0.5,0.7)
    to[out=-60,in=150] (0,0) -- cycle;
\end{scope}
    \draw (-2.5, 0) -- (1.5, 0); 
    \draw (0,-2) -- (0,2);
  \draw[thick] (0,0) to[out=150,in=-60] (-0.5,0.7)
    to[out=120,in=-30, looseness=1.2] (-1,1.5) to[out=150,in=60,looseness=1.2] (-1.3,0.8)
    to[out=240,in=40] (-2,0); 
     \draw (0,0) to[out=-150,in=60] (-0.5,-0.7)
    to[out=-120,in=30, looseness=1.2] (-1,-1.5) to[out=-150,in=-60,looseness=1.2] (-1.3,-0.8)
    to[out=-240,in=-40] (-2,0); 
    \node at (0,0) [below right] {$0$};
    \node at (-2,0) [below left] {\small $1-\frac1{\nu(\{0\})}$};
    \draw (0,0) circle [radius = 1.3];
    \node at (-0.2,0.7) {$E$};

     \node at (-0.66,1.1) {\small$\bullet$};
      \node at (-0.8,1.2) [above right] {$r e^{\iu \phi}$};
       \node at (0,1.3) [above right] {$\iu r$};
       \node at (0,1.3) {\small$\bullet$};
       \draw[->] (-2,1.5) -- (-1.35,0.85);
       \node at (-2,1.5) [above]{\small$\eta_\nu(\iu(0,+\infty))$};
     \end{tikzpicture}
           \caption{The domain $E$} \label{fig:E}
           \end{figure}

      For sufficiently small $\delta>0$ and for any $z$ with $0<|z|\le \delta$ and $ \pi-\phi \le \arg z \le \phi$, it holds that $\eta_\mu(z)\in D$, simply because $\nontanglim_{z\to 0}\eta_\mu(z)=0$. By the inequality $\arg z \le \arg\eta_\mu(z) \le \arg z +\pi$ and by Theorems \ref{thm:fatou} and \ref{thm:priv}, there exist real numbers $0<x<\delta$ and $-\delta<y<0$ such that the nontangential limit of $\eta_\mu(z)$ at $x$ exists in $(\Hup\cup\RR)\setminus\{0\}$, and at $y$ in $(\CC^-\cup\RR)\setminus\{0\}$.

   In the following $\Arg z$ denotes the argument of $z \in \CC\setminus\{0\}$ such that $\Arg z \in(-\pi,\pi]$.   

   \vspace{2mm}\noindent
  \textbf{Case 1.} 
    If $\eta_\mu(\iu t)\in\eta_\nu(\cl{\iu\Hup}\setminus\{0\})$ for some $0<t<\delta$, then $w\coloneqq\iu t$ satisfies the two conditions in the statement. Indeed, the first condition is clearly satisfied. For the second condition, we take  $\zeta\in \cl{\iu\Hup}\setminus\{0\}$ such that $\eta_\mu(\iu t) = \eta_\nu(\zeta)$. If $ \arg \zeta \in[\pi/2, \pi)$ then, since $\arg \eta_\nu(\zeta) \in (\arg \zeta, \pi)$, we have  
    \begin{align*}
              \Argg\left[\frac{\eta_\nu^\inv(\eta_\mu(\iu t))}{\eta_\mu(\iu t) }\right]= \Argg\left[\frac{\zeta}{\eta_\nu(\zeta)}\right] \in \left( -\frac{\pi}{2},0\right). 
          \end{align*}
 If $ \arg \zeta \in(\pi, 3\pi/2]$, by the symmetry $\eta_\nu(\conj{z})=\conj{\eta_\nu(z)}$ we have   
$ \Argg(\zeta/ \eta_\nu(\zeta)) \in ( 0,\frac{\pi}{2}).$ 
 If $\arg \zeta =\pi$ then $\Argg(\zeta/ \eta_\nu(\zeta))=0$.          
 Therefore, in any case the inequality           
          \begin{align}\label{eq:arg}
              -\frac{\pi}{2}<\Argg\left[\frac{\eta_\nu^\inv(\eta_\mu(\iu t))}{\eta_\mu(\iu t) }\right]<\frac{\pi}{2} 
          \end{align}
holds.  The fact $\arg \iu t=\pi/2$ and  \eqref{eq:arg}  immediately imply the second condition.

           \vspace{2mm}\noindent
   \textbf{Case 2.}  If $\eta_\mu(\iu t)\notin\eta_\nu(\cl{\iu\Hup}\setminus\{0\})$ for any $t\in(0,\delta)$, then the curve $C\coloneqq\{\eta_\mu(\iu t)\colon 0<t<\delta\}$ is entirely contained in either $E$ or $\conj{E}:=\{z:\conj{z} \in E \}$ because $C \subseteq \iu\Hup \cap D$. We discuss these cases separately. 

\begin{figure}[t]
    \centering
    \begin{tikzpicture}[scale=1.5]
    \draw (-2.5, 0) -- (1.5, 0); 
    \draw (0,-2) -- (0,2);
  \draw (0,0) to[out=150,in=-60] (-0.5,0.7)
    to[out=120,in=-30, looseness=1.2] (-1,1.5) to[out=150,in=60,looseness=1.2] (-1.3,0.8)
    to[out=240,in=40] (-2,0); 
     \draw (0,0) to[out=-150,in=60] (-0.5,-0.7)
    to[out=-120,in=30, looseness=1.2] (-1,-1.5) to[out=-150,in=-60,looseness=1.2] (-1.3,-0.8)
    to[out=-240,in=-40] (-2,0); 
    \node at (0,0) [below right] {$0$};
    \node at (-2,0) [below left] {\small $1-\frac1{\nu(\{0\})}$};
    \draw (0,0) circle [radius = 1.3];
    \draw[thick] (0,0)  to[out=110,in=-50] (-0.2,0.5) 
    to[out=130,in=-70, looseness=1.2] (-0.3,0.8);
    \node at (-0.3,0.8) {\small$\bullet$};
     \draw[thick] (-0.3,0.8)  to[out=200,in=70] (-0.9,0.2) 
    to[out=250,in=140, looseness=1.2] (-0.7,-0.4);
    \node at (-0.7,-0.4) {\small$\bullet$};
    \draw[->] (-1.7,1) -- (-0.9,0.3);
       \node at (-1.7,1) [above]{\small$\eta_\mu(\gamma)$};
     \node at (-0.66,1.1) {\small$\bullet$};
     \node at (-0.49,0.71) {$\star$};
      \node at (-0.8,1.2) [above right] {$r e^{\iu \phi}$};
       \node at (0,1.3) [above right] {$\iu r$};
       \node at (0,1.3) {\small$\bullet$};
       \draw[->] (1.5,1) -- (-0.1,0.4);
       \node at (1.5,1) [above]{\small$\eta_\mu((0,\iu|y|])$};
     \end{tikzpicture}
           \caption{In case 2 (a), the curve $\eta_\mu((0,\iu|y|])$ is contained in $E$. Then the curve $\eta_\mu(\gamma)$ traverses $\eta_\nu((0,\iu s))$ at a point denoted by $\star$.} \label{fig:E2}
           \end{figure}

   \begin{enumerate}[label=(\alph*),topsep=0.3em,itemsep=0.3em]
      \item   When $C$ is contained in $E$, one can find $w$ on the arc $\gamma\coloneqq\{|y|e^{\iu\theta}: \pi/2<\theta<\pi\}$ such that $\eta_\mu(w) \in\eta_\nu((0,\iu s))$ holds, as its image $\eta_\mu(\gamma)$ has to traverse the simple curve $\eta_\nu((0,\iu s))\subseteq\Hup\cap \iu\Hup$ to approach the nontangential limit $\eta_\mu(y)\in (\CC^-\cup\RR)\setminus\{0\}$ that is located outside $\cl{E}$, see Figure \ref{fig:E2}. 
          More precisely, we first consider the subarc $\gamma':= \{|y|e^{\iu\theta}: \pi/2<\theta\le \phi\}\subseteq \gamma$. If $\eta_{\mu}(\gamma')$ intersects $\eta_\nu((0,\iu s))$, we are done. Otherwise, $\eta_{\mu}(\gamma')$ is still  contained in $E$ as it intersects neither  $\{re^{\iu\theta}: \pi/2 \le\theta \le \phi\}$ nor $[0,\iu r]$ because $\eta_\mu(z) \in D$ and $\arg z \le \arg \eta_\mu(z) \le \arg z + \pi$ on $\gamma'$, respectively. 
          In this case we further look at $\gamma'':= \{|y|e^{\iu\theta}: \phi\le\theta <\pi\}$; the curve $\eta_\mu(\gamma'')$ starts from the point $\eta_\mu(|y|e^{\iu \phi})\in E$ and ends at $\eta_\mu(y)\notin\cl{E}$. Although $\eta_\mu(\gamma'')$ might not be contained in $D$ anymore, it cannot ``circumvent'' the curve $\eta_\nu((0,\iu s))$ to approach to $\eta_\mu(y)$, i.e., it intersects neither  $\{re^{\iu\theta}: \pi/2 \le\theta \le \phi\}$ nor $[0,\iu r]$ now because $\eta_\mu(|y|e^{\iu\phi})\in D$ and $\phi < \arg \eta_\mu(z) < 2\pi$ for all $z\in \gamma''$. Therefore, $\eta_\mu(\gamma'')$ must  intersect $\eta_\nu((0,\iu s))$ to approach to $\eta_\mu(y) \notin \cl{E}$. 
         
          Any $w$ on $\gamma$ such that $\eta_\mu(w) \in \eta_\nu((0,\iu s))$ also satisfies condition \ref{item:cond2}. Indeed, the fact $\arg w,\arg \eta_\mu(w) \in (\pi/2,\pi)$ and the general inequality $\arg \eta_\mu(w)\in (\arg w, \arg w +\pi)$ imply $\Argg (w/\eta_\mu(w))\in(-\pi/2,0)$, so that considering  $ \arg{\eta_\nu^\inv(\eta_\mu(w))}=\pi/2$ we obtain 
          \[
          \Argg\left[ \frac{w}{\eta_\mu(w)}\cdot \eta_\nu^\inv(\eta_\mu(w))\right]\in \left(0,\frac{\pi}{2}\right). 
          \]
          
          \item Similarly, in the case where $C$ is contained in $\overline{E}$, the arc $\{x e^{\iu\theta}: 0<\theta<\pi/2\}$ contains a point $w$ such that $\eta_\mu(w) \in \eta_\nu((0,-\iu s))$.  
          We can also check condition \ref{item:cond2} analogously to the previous case. 
  \end{enumerate}
    Combining the above arguments, we obtain the conclusion.
  \end{proof}

  \begin{proof}[\textbf{Proof of Theorem  \ref{thm:subordination}}] If $\mu$ or $\nu$ is degenerate then the statements are trivial; for example if $\mu=\delta_a~(a\ne0)$ then $\eta_\mu(z)=az, \eta_{\mu\multconv\nu}(z) =\eta_\nu(az)$, and so the subordination functions are uniquely given by $\omega_1(z)=\eta_\nu(az)/a, \omega_2(z) = a z$. We therefore assume that $\mu, \nu$ are nondegenerate.
     Let $f_z,g_z$ be the functions defined in Proposition  \ref{prop:not_conf}. We take $w,z$ as in Lemma  \ref{lem:ex_of_fp}. We have
      \begin{align*}
          f_z(w)&=zh_\nu(zh_\mu(w))=\frac{w\eta_\nu(z\eta_\mu(w)/w)}{\eta_\mu(w)}=\frac{w\eta_\nu(\eta_\nu^\inv(\eta_\mu(w)))}{\eta_\mu(w)}=w,
      \end{align*}
      so that $w$ is the DW-point of the analytic function $f_z\colon\Hup\to\Hup$.
      Thus, by Theorem  \ref{thm:dwfix}, together with the fact that $f_z(w)$ cannot be a conformal automorphism by Proposition  \ref{prop:not_conf}, there exists an analytic function $\omega_1(z)\colon\Hup\to\Hup$ such that $f_z(\omega_1(z))=\omega_1(z)$ for all $z\in\Hup$. Defining $\omega_2(z)\coloneqq zh_\mu(\omega_1(z))$, we have
      \begin{align}
          z\eta_\mu(\omega_1(z))=z\eta_\nu(\omega_2(z))=\omega_1(z)\omega_2(z), \label{eq:omegaA}
      \end{align}
      for all $z\in\Hup$. Indeed, by the definitions of $\omega_1(z)$ and $\omega_2(z)$, it follows that 
      \begin{align*}
          z\eta_\nu(\omega_2(z))&=zh_\nu(\omega_2(z))\omega_2(z)=zh_\nu(zh_\mu(\omega_1(z)))\omega_2(z) \notag \\
          &=f_z(\omega_1(z))\omega_2(z)=\omega_1(z)\omega_2(z), 
      \end{align*}
      and
      \begin{align*}
          z\eta_\mu(\omega_1(z))&=zh_\mu(\omega_1(z))\omega_1(z)=\omega_1(z)\omega_2(z) 
      \end{align*}
      for all $z\in\Hup$.
      Note that 
      \begin{align*}
          h_\mu(zh_\nu(\omega_2(z)))&=h_\mu(z\eta_\nu(\omega_2(z))/\omega_2(z))=h_\mu(\omega_1(z))\\
          &=\eta_\mu(\omega_1(z))/\omega_1(z)=\omega_2(z)/z, 
      \end{align*}
      and hence $\omega_2(z)/z\ (=h_\mu(\omega_1(z)))$ is an analytic self-map of $\Hup$ and is the DW-point of $g_z(w)= h_\mu(zh_\nu(zw))$. 
Conversely, for any analytic self-maps $\omega_1$ and $\omega_2$ satisfying equation \eqref{eq:omegaA}, the points $\omega_1(z)$ and $\omega_2(z)/z$ can be easily shown to be fixed points of $f_z$ and $g_z$, respectively, so the uniqueness follows by Theorem  \ref{thm:denjoy_wolff}.

      Next, we will prove that $\eta_\nu(\omega_2(z))=\eta_\mu(\omega_1(z))=\eta_{\mu\multconv\nu}(z)$. Let $\{\mu_n\}_{n=1}^\infty$ and $\{\nu_n\}_{n=1}^\infty$ be sequences of compactly supported probability measures with nonzero mean, converging to $\mu$ and $\nu$, respectively, with respect to the Kolmogorov--Smirnov distance, i.e.,
      \begin{align*}
          &\sup_{x\in\RR} |\mu_n((-\infty,x])-\mu((-\infty,x])|\longrightarrow0\quad(n\to\infty),\\
          &\sup_{x\in\RR} |\nu_n((-\infty,x])-\nu((-\infty,x])|\longrightarrow0\quad(n\to\infty).
      \end{align*} 
      Then, by Theorem  \ref{thm:ct_unifconv}, we obtain
      \begin{align*}
          f_z^{(n)}(w)\to f_z(w), g_z^{(n)}(w)\to g_z(w), 
      \end{align*}
      for all $w\in\Hup$. Since neither of $f_z(w),g_z(w)$ is the identity map, we have by Theorem  \ref{thm:conv_of_dwp}
      \begin{align*}
          \omega_1^{(n)}(z)\to\omega_1(z),\ \omega_2^{(n)}(z)\to\omega_2(z),
      \end{align*}
      where $\omega_1^{(n)}(z),\omega_2^{(n)}(z)$ denote the subordination functions associated with the probability measures $\mu_n,\nu_n$.
      By Theorem 4.11 of  \cite{BV1993free} and Theorem  \ref{thm:ct_unifconv}, we have
    \begin{align*}
        \eta_{\mu_n\multconv\nu_n}(z)\to\eta_{\mu\multconv\nu}(z),\quad (n\to\infty).
    \end{align*}
    Utilizing the relation $z\eta_{\mu_n\multconv\nu_n}(z)=\omega_1^{(n)}(z)\omega_2^{(n)}(z)$, we deduce that
    \begin{align*}
        z\eta_\mu(\omega_1(z))=z\eta_\nu(\omega_2(z))=\omega_1(z)\omega_2(z)=z\eta_{\mu\multconv\nu}(z)
    \end{align*}
    for all $z\in\Hup$.

    The relation $\omega_2(z)=zh_\mu(\omega_1(z))$ and Proposition  \ref{prop:not_conf} imply that $\arg z\leq\arg\omega_2(z)\leq\arg z+\pi$. 
    Observe that $f_z(w)\to 0$ as $z\nontangtou 0$. 
    Taking any sequence $\{z_n\}_{n=1}^\infty$ in a fixed domain $\Gamma_{0,\theta}$ such that $z_n\to0$ and applying the upper half-plane version of Theorem  \ref{thm:conv_of_dwp} to the functions $f_n:=f_{z_n}$ and $f\equiv0$, we deduce that  $\lim_{n\to\infty}\omega_1(z_n)=0$, and thus     $\nontanglim_{z\to0}\omega_1(z) = 0$. Using the facts that $\eta_\nu(z)\neq0$ on $\CC\setminus[0,+\infty)$, the subordination relation $\eta_\nu(\omega_2(z)) = \eta_{\mu\multconv\nu}(z)\to0$ as $z\nontangtou0$, $\eta_\nu(w)\to 1-1/\nu(\{0\})\ne0$ as $w\to \infty$ with $w \in \Delta_\theta$ for a fixed $\theta \in(0,\pi)$, together with the inequality $\arg z\leq \arg\omega_2(z)\leq\arg z+\pi$, we conclude that $\nontanglim_{z\to0}\omega_2(z)=0$, thereby completing the proof.
  \end{proof}

\section{Subordination functions as  $\eta$-transforms} \label{sec:subord2}

For $\mu \in \Prob(\RR)\setminus\{\delta_0\}$ and $\nu \in \Prob(\RR_{\ge0})\setminus\{\delta_0\}$, by the converse statement of Proposition  \ref{prop:prop_of_eta} and by Theorem \ref{thm:subordination}, 
the subordination function $\omega_2$ is the $\eta$-transform of some probability measure on $\RR$, which we denote as $\mathbb{B}_{\nu}(\mu)$. 
Recall from Section \ref{sec:ex_of_subord} that $\omega_2(z)/z = \eta_{\mathbb{B}_\nu(\mu)}(z)/z$ is the DW-point of the function $g_z(w) = h_\mu(z h_\nu(zw))$ if $\mu,\nu$ are nondegenerate. 
It is easy to generalize this to possibly degenerate measures as long as $\nu\ne\delta_0$. In particular, if $\mu=\delta_a$ then $g_z(w)\equiv a$, so that its DW-point is $\omega_2(z)/z\equiv a$, i.e., $\mathbb{B}_\nu(\delta_a)=\delta_a$. 
In this way, we can define $\mathbb B_\nu(\mu)$ for all $\mu \in \Prob(\RR)$ and $\nu \in \Prob(\RR_{\ge0})\setminus\{\delta_0\}$.  
We provide properties of $\mathbb{B}_{\nu}(\mu)$, in particular, generalizations of \cite[Theorems 1.1 and 1.2]{AH2016free}.

\begin{remark} \label{eq:omega_1_not_eta}

\begin{enumerate}[label=\rm(\roman*)]

\item  The subordination function $\omega_1$
for $\mu \coloneq \mathbf{b}_{\alpha,\rho}$ and $\nu\coloneq (\mathbf{b}_{\beta,1})^{\multconv \frac1{\alpha}}$ can be explicitly given as 
\begin{equation*}
\omega_1(z) = e^{\iu\rho\pi}(e^{-\iu\rho\pi}z)^\beta.  
\end{equation*}
As soon as $0<\rho <1$ and $0<\beta<1$, this is not an $\eta$-transform of a probability measure on $\RR$ since the inequality $\arg z \le \arg \omega_1(z) \le \arg z + \pi$ fails to hold. 
\item The function $\omega_2$ is not a self-map of  $\Hup$ in general. For example let $\mu=\frac1{2}(\delta_{-1}+\delta_1)$ and $\nu=\delta_1$. Then $\eta_\mu(z)=z^2$, $\eta_\nu(z)=z$ and $\eta_{\mu\boxtimes\nu}(z)=z^2$, so that $\omega_2(z) = z^2$ and we see $\omega_2(\Hup) = \CC\setminus [0,+\infty)$.

\end{enumerate}
\end{remark}

\subsection{Proofs of Theorems \ref{thm:subordination_hom} and \ref{thm:pde_for_semigroup}} \label{sec:subordination_hom}

As preparations, we first study the weak continuity. 

\begin{proposition}\label{prop:conv_omega2} The map $\Prob(\RR) \times (\Prob(\RR_{\ge0})\setminus\{\delta_0\})\to \Prob(\RR)$, $(\mu,\nu)\mapsto \mathbb{B}_{\nu}(\mu)$, is continuous  with respect to the weak topologies. 
\end{proposition}
\begin{proof}
First we observe that $g_z$ is not the identity map for any $z\in \Hup$. Indeed, if $\mu,\nu$ are nondegenerate, then this is the case by Proposition  \ref{prop:not_conf}  \ref{item:not_auto4}. If $\mu$ or $\nu$ is degenerate, then $g_z(w)$ is a constant function of $w$. 

Since $g_z$ depends continuously on $(\mu, \nu)$, by Theorem  \ref{thm:conv_of_dwp}, its DW-point $\eta_{\mathbb{B}_\nu(\mu)}(z)/z$ depends continuously on $\mu$ and $\nu$ as well. Then Theorem  \ref{thm:ct_unifconv} finishes the proof. 
\end{proof}

\begin{remark}  There is no weakly continuous extension of  the map $(\mu,\nu)\mapsto \mathbb{B}_{\nu}(\mu)$ to the domain $\Prob(\RR) \times \Prob(\RR_{\ge0})$. Take $\mu = \mathbf{b}_{\alpha,1}, 0<\alpha<1$ and $\nu=\delta_s~(s>0)$, for example. Then $h_\mu(z) = (-z)^{\alpha-1}$ and $h_\nu(z)\equiv s$, so that $g_z(w) =  s^{\alpha-1} (-z)^{\alpha-1}$ and $\omega_2(z) =  - s^{\alpha-1} (-z)^{\alpha}$. As  $s\to0^+$, the measure $\nu$ goes to $\delta_0$, but $\omega_2$  diverges to $\infty$, so that $\mathbb B_\nu(\mu)$ does not converge to a probability measure.  
\end{remark}

Proposition  \ref{prop:conv_omega2} implies a partial result for Corollary  \ref{cor:wconti}, which is  enough to prove Theorem  \ref{thm:subordination_hom}. 
\begin{corollary} \label{cor:wconti0}
The restriction of free multiplicative convolution $\multconv\colon (\Prob(\RR)\setminus\{\delta_0\}) \times (\Prob(\RR_{\ge0})\setminus\{\delta_0\})\to\Prob(\RR)\setminus\{\delta_0\}$ is continuous with respect to the weak topology. 
\end{corollary}
\begin{proof} Suppose that $\mu_n, \mu \in \Prob(\RR)\setminus\{\delta_0\}$ and $\nu_n,\nu\in \Prob(\RR_{\ge0})\setminus\{\delta_0\}$ such that $\mu_n \to \mu$ and $\nu_n\to\nu$. By Montel's theorem, $\eta_{\nu_n}$ converges to $\eta_\nu$ locally uniformly on $\CC\setminus [0,+\infty)$. Let $\omega_{2,n}, \omega_2$ be subordination functions for $\mu_n \multconv \nu_n, \mu\multconv \nu$, respectively. By Proposition  \ref{prop:conv_omega2}, $\omega_{2,n} (z)\to\omega_2(z)$. Since $\omega_2(z) \in \CC\setminus[0,+\infty)$ for all $z \in \Hup$, we conclude that $\eta_{\mu_n\multconv\nu_n}(z) = \eta_{\nu_n}(\omega_{2,n}(z))\to \eta_\nu(\omega_2(z))=\eta_{\mu\multconv\nu}(z)$. Then Theorem  \ref{thm:ct_unifconv} finishes the proof. 
\end{proof}

\begin{proof}[\textbf{Proof of Theorem  \ref{thm:subordination_hom}}]
If $\mu, \nu,\rho, \sigma$ have compact support and nonvanishing means, then formulas \eqref{eq:hom2} and \eqref{eq:hom} can be proved by means of $S$- or $\Sigma$-transform as in  \cite[Theorem 1.1]{AH2016free}. In the general case, formula \eqref{eq:hom2} can be proved by using Proposition  \ref{prop:conv_omega2} and by approximating $\mu,\rho, \sigma$ with measures having nonzero mean and compact support. If $\mu, \nu \ne\delta_0$, formula \eqref{eq:hom} can be proved by approximation arguments and Proposition \ref{prop:conv_omega2} and Corollary  \ref{cor:wconti0}. If $\mu=\delta_0$ or $\nu=\delta_0$ then the both hands sides of \eqref{eq:hom} are trivially $\delta_0$. 
\end{proof}

\begin{proof}[\textbf{Proof of Theorem  \ref{thm:pde_for_semigroup}}]
For notational convenience, we set $ \eta(t,z) = \eta_{\nu_t}(z)$. The fixed point equation for $\omega$ reads 
\[
\eta_\mu\left( \frac{z\eta(t,\omega)}{\omega} \right) = \eta(t,\omega). 
\]
Taking the derivative with respect to $z$ and $t$ yields the equality 
\[
\frac{\frac{\partial }{\partial t} \left[ \eta_\mu\left( \frac{z\eta(t,\omega)}{\omega} \right) \right]}{ \frac{\partial }{\partial z} \left[  \eta_\mu \left( \frac{z\eta(t,\omega)}{\omega} \right) \right] }   =   \frac{\frac{\partial }{\partial t}  \left[ \eta(t,\omega)  \right]}{\frac{\partial }{\partial z} \left[  \eta(t,\omega) \right]}. 
\]
After some calculations this equality can be simplified into 
\begin{equation}\label{eq:omega}
(\partial_z \eta)(t,\omega)  \partial_t \omega  = -(\partial_t \eta)(t,\omega) + \frac{z(\partial_t \eta)(t,\omega) }{\omega} \partial_z \omega.
\end{equation}
On the other hand, taking the derivative $\left. \frac{\di{}}{\di s}\right|_{s=0}$ of $\eta(t+s,z) = \eta(t, \eta(s,z))$ gives 
\begin{equation}\label{eq:eta}
(\partial_t\eta) (t, z)  = zK(z)(\partial_z \eta)(t, z). 
\end{equation}
Substituting \eqref{eq:eta} into \eqref{eq:omega} finishes the proof of the desired PDE \eqref{eq:omega1}. 
\end{proof}

\subsection{Three examples related to convolution powers}

Let $(\nu_t)_{t\ge0}$ be a weakly continuous $\circlearrowright$-convolution semigroup with $\nu_0=\delta_1$. Then the one-parameter family of self-maps $(\mathbb{B}_{\nu_t})_{t\ge0}$ on $\Prob(\RR)$ is a compositional semigroup consisting of homomorphisms with respect to $\multconv$.  
Three major examples are given in  \cite{AH2016free}, which we reproduce and extend below. For convolution powers appearing below, the reader is referred to Section \ref{sec:conv_powers} and references therein.

\begin{example}
If we set 
\[
\sigma_t \coloneq \frac{t}{1+t}\delta_0  + \frac{1}{1+t}\delta_{1+t}, \qquad t\ge0, 
\]
then $(\mathbb{B}_t)_{t\ge0}\coloneq(\mathbb{B}_{\sigma_t})_{t\ge0}$ coincides with the maps of Belinschi and Nica  \cite{BN2008remarkable}, i.e., 
\begin{equation}\label{eq:BN}
\mathbb{B}_{t} (\mu) = (\mu^{\addconv(1+t)})^{\uplus \frac{1}{1+t}}. 
\end{equation} 
This fact can be proved via the other two examples below. 
\end{example}

\begin{example}\label{ex:B}
The second example is $(\rho_t)_{t>0} = (\delta_{1/t})_{t>0}$. This  is not a $\circlearrowright$-convolution semigroup; $(\rho_{e^{-t}})_{t\ge0}$ or  $(\rho_{e^{t}})_{t\ge0}$ is a precise $\circlearrowright$-convolution semigroup. One can easily see that 
\begin{equation}\label{eq:B}
\mathbb{B}_{\rho_t} (\mu) = \Dil_{\frac1{t}}(\mu^{\uplus t}), 
\end{equation}
where $D_s(\mu)$ is the dilation of $\mu$ by $s\in\RR$, i.e., the pushforward of $\mu$ by the map $x\mapsto sx$. 
Hence, \eqref{eq:hom} reads 
\[
 \Dil_{\frac1{t}}((\mu \multconv \nu) ^{\uplus t})  =  \Dil_{\frac1{t}}(\mu^{\uplus t}) \multconv  \Dil_{\frac1{t}}(\nu^{\uplus t}), \qquad t > 0, 
\]
which was known in \cite[Proposition 3.5]{BN2008remarkable}. 
\end{example}

\begin{example}The third example is $(\tau_{t})_{t\ge1}$ consisting of $\tau_t = (1-\frac1{t})\delta_0 + \frac{1}{t} \delta_1$ which is again not exactly a $\circlearrowright$-convolution semigroup but $(\tau_{e^{t}})_{t\ge0}$ is. One can see that
\begin{equation}\label{eq:free_conv}
\mathbb{B}_{\tau_t}(\mu) =  \Dil_{\frac1{t}}(\mu^{\addconv t}). 
\end{equation}
One can prove this first by using $\Sigma$- or $S$-transforms  when $\mu$ has compact support and nonvanishing mean (formulas (3.11) and (2.15) in  \cite{BN2008remarkable} are helpful), and then resort to approximation when $\mu$ is arbitrary.  
Formula \eqref{eq:hom} reads
\begin{equation}\label{eq:BN_free}
 \Dil_{\frac1{t}}((\mu \multconv \nu) ^{\addconv t})  =  \Dil_{\frac1{t}}(\mu^{\addconv t}) \multconv  \Dil_{\frac1{t}}(\nu^{\addconv t}), \qquad t\ge1, 
\end{equation}
which was also known in  \cite[Proposition 3.5]{BN2008remarkable}. Notice that this identity also makes sense for $t$ smaller than $1$ whenever $\mu^{\boxplus t} \in \Prob(\RR)$ and $\nu^{\boxplus t} \in \Prob(\RR_{\ge0})$ exist. 
\end{example}

Coming back to the first example, one can check from \eqref{eq:monotone} that $\tau_{1+t} \circlearrowright \rho_{1/(1+t)} =\sigma_t$, and so formula \eqref{eq:hom2}, together with \eqref{eq:B} and \eqref{eq:free_conv}, implies \eqref{eq:BN}. 

As shown in \cite[Theorem 1.5]{BN2008remarkable}, the function $h(t,z) \coloneq z - F_{\mathbb{B}_{t}(\mu)}(z)$ satisfies the complex Burgers equation
\[
\frac{\partial h}{\partial t} = -h(t,z) \frac{\partial h}{\partial z}, \qquad h(0,z) = z - F_\mu(z).  
\]
This corresponds to the case $K(z)=z$ in Theorem  \ref{thm:pde_for_semigroup}. Some more examples are provided in \cite{AH2016free}.

\section{$T$-transform and $S$-transform}  \label{sec:T} 
In this section, we consider $F_\mu$ on $\CC^-$ for the sake of  consistency with other transforms related via  \eqref{eq:F_psi_eta}. 
We begin by restating the definition of $T$-transform in \eqref{eq:T-transform}, confirming that the limit indeed exists. 

For $\mu \in \Prob(\RR)$, $t\ge1$ and $z \in \CC^-\cup \RR$, the function
\[
L_{\mu,t,z}(w) \coloneq  z +(t-1)(F_\mu(w)-w) 
\]
 is an analytic map from $\CC^-$ into $\CC^-\cup \RR$. If $z \in \CC^-$ or $\mu$ is nondegenerate and $t>1$ then $L_{\mu,t,z}$ is a self-map of $\CC^-$. Let $\sof_{\mu,t}(z)$ be  the DW-point of $L_{\mu,t,z}$. Note that $\sof_{\mu,1}(z) =z$. 
It is shown in  \cite[Section 5]{BB2005partially} that $\sof_{\mu,t}$ is a continuous self-map of $\CC^-\cup \RR$ and its restriction to $\CC^-$ is an analytic self-map of $\CC^-$. 
Moreover, by Theorem  \ref{thm:conv_of_dwp}, the map 
\begin{equation}\label{eq:conti}
\begin{split}
\Prob(\RR) \times [1,+\infty) \times (\CC^-\cup \RR) &\to  \CC^-\cup \RR, \\
 (\mu,t,z) &\mapsto \sof_{\mu,t}(z)
\end{split}
\end{equation}
 is also continuous. The function $\sof_{\mu,t}$ satisfies  \cite[Section 5]{BB2005partially}
\begin{equation}\label{eq:F}
F_{\mu^{\addconv t}} (z) = F_\mu (\sof_{\mu,t}(z)) = \frac{t\sof_{\mu,t}(z) -z}{t-1},  \qquad t>1, ~z \in \CC^-, 
\end{equation}
which allows us to extend $F_{\mu^{\addconv t}}$ to a continuous self-map of $\CC^-\cup \RR$, still denoted by the same symbol. For $x \in \RR$ such that $\sof_{\mu,t}(x) \in \RR$, by Lindel\"of's theorem and \eqref{eq:F}, $F_\mu$ has a nontangential limit at $\sof_{\mu,t}(x)$ and the limit equals $\frac{t\sof_{\mu,t}(x) -x}{t-1}$. 
Thus \eqref{eq:F} can be extended to all $z \in \CC^- \cup \RR$.

\begin{definition}
The \emph{$T$-transform of $\mu \in \Prob(\RR)$} and \emph{$S$-transform of $\mu\in \Prob(\RR)\setminus\{\delta_0\}$} are defined by 
\begin{align}
&T_\mu(-s) \coloneq - s F_{\mu^{\addconv 1/s}} (0) = - F_{\Dil_s(\mu^{\addconv 1/s})} (0)= \frac{s \sof_{\mu,1/s}(0)}{s-1}, &\hspace{-10mm} 0<s<1, \label{def:T} \\ 
&S_\mu(-s) \coloneq \frac{1}{T_\mu(-s)},&\hspace{-10mm} 0<s<1-\mu(\{0\}).  \label{def:S}
\end{align}
Recall that $D_s$ is the  dilation operator introduced in Example \ref{ex:B}. 
\end{definition}

\subsection{Proof of Theorem \ref{thm:T} and some properties}
We begin with some basic properties of the $T$-transform that arises immediately from the definition and known facts. 

\begin{proposition}\label{prop:omega} Let $\mu\in \Prob(\RR)\setminus\{\delta_0\}$.  
\begin{enumerate}
\item\label{itemO1} $ \sof_{\mu,t}(0) =0$ for $1\le t \le 1/(1-\mu(\{0\}))$ and $ \sof_{\mu,t}(0)\ne 0$ for $t > 1/(1-\mu(\{0\}))$. 
\item\label{itemO2} The function $t\mapsto \sof_{\mu,t}(0)$ is continuous on $[1,+\infty)$ and injective on $[1/(1-\mu(\{0\})),+\infty)$. 
\item\label{itemO3} For any $z \in \CC^- \cup \RR$, $\lim_{t\to+\infty} \sof_{\mu,t}(z) =\infty$. 
\end{enumerate}
\end{proposition}
\begin{proof}{\bf \ref{itemO1}} This is a direct consequence of  \cite[Proposition 5.1]{BB2005partially} showing that, for $t>1$, $F_{\mu^{\addconv t}} (0) = 0$ if and only if $0$ is an atom of $\mu$ with $\mu(\{0\})\ge (t-1)/t$.   

\vspace{2mm}
\noindent
{\bf  \ref{itemO2}} Continuity is a special case of the continuity of the map \eqref{eq:conti}. For injectivity, suppose that $\sof_{\mu,t_1} (0) =\sof_{\mu,t_2}(0)$, $t_1, t_2 > 1/(1-\mu(\{0\}))$. By \eqref{eq:F} we obtain  
\[
\frac{t_1 \sof_{\mu,t_1}(0)}{t_1-1} = F_\mu(\sof_{\mu,t_1} (0)) = F_\mu(\sof_{\mu,t_2} (0)) =\frac{t_2 \sof_{\mu,t_2}(0)}{t_2-1}. 
\]
This implies $t_1/(t_1-1) = t_2/(t_2-1)$, i.e., $t_1 = t_2$ as desired. 

\vspace{2mm}
\noindent
{\bf \ref{itemO3}} By Theorem  \ref{thm:conv_of_dwp}, it suffices to show that $L_{\mu,t,z}(w) \to \infty$ as $t\to+\infty$. 
If $\mu=\delta_a ~(a\ne0)$ then $L_{\mu,t,z}(w) = z - (t-1)a \to \infty$. If $\mu$ is nondegenerate then $F_\mu(w)-w$ has negative imaginary parts on $\CC^-$, and so $L_{\mu,t,z}(w) \to \infty$. 
\end{proof}

Proposition \ref{prop:omega} implies that $T_\mu(u)\ne0$ for $\mu(\{0\})-1<u<0$,  $T_\mu$ and $S_\mu$ are continuous functions into $\Hup\cup\RR$ and $(\CC^- \cup \RR)\setminus\{0\}$, respectively, and additionally 
\begin{equation}\label{eq:Tat0}
\lim_{u\to0^-}\frac{T_\mu(u)}{u}=\infty\qquad \text{and} \qquad \lim_{u\to0^-} uS_\mu(u)=0,\qquad \mu\ne \delta_0. 
\end{equation}
On the other hand, one can easily see $\sof_{\delta_0,t}(z) =z$ and so $T_{\delta_0}\equiv0$.

We give a proof of Theorem  \ref{thm:T} below. The continuity \ref{itemT0} is already proved. The most important \ref{itemT1}, the multiplicativity $T_{\mu\multconv\nu}=T_\mu T_\nu$,  is based on the following crucial lemma. 

\begin{lemma}\label{lem:T} 
Suppose that $\mu \in \Prob(\RR), \nu \in \Prob(\RR_{\ge0})$ such that $\mu(\{0\})=\nu(\{0\})=0$ and  $F_\mu, F_\nu$ extend to continuous functions from $\CC^-\cup\{0\}$ into $\CC^- \cup \RR$, still denoted as $F_\mu, F_\nu$. Then $F_{\mu \multconv \nu}$ has a nontangential limit at $0$, denoted as $F_{\mu \multconv \nu}(0)$, and 
\begin{equation}\label{eq:inverse_mean}
F_{\mu \multconv \nu}(0) = - F_\mu(0) F_\nu(0). 
\end{equation}
\end{lemma}
\begin{remark}
\begin{enumerate}

\item[(i)] If we use Theorem \ref{thm:leb_decomp} \ref{item:leb_decomp3} then the proof can be shortened as Step 1 below would be unnecessary. In addition, the use of Theorem \ref{thm:leb_decomp} \ref{item:leb_decomp3} would allow us to drop the assumption $\mu(\{0\})=\nu(\{0\})=0$: if $\mu$ or $\nu$ has a positive mass at 0 then both sides of \eqref{eq:inverse_mean} are zero. However,  in order to give a self-contained proof of Theorem  \ref{thm:leb_decomp} \ref{item:leb_decomp3} as addressed in Remark \ref{rem:regularity}, we avoid using Theorem \ref{thm:leb_decomp} \ref{item:leb_decomp3} at this moment. \item[(ii)] If $\mu$ and $\nu$ have no mass in a neighborhood of $0$, then the claim is much easier to show: we take freely independent self-adjoint operators $X,Y$ that have distributions $\mu,\nu$ respectively, and affiliated to a $W^*$-probability space $(\mathcal A, \tau)$. The assumption implies that $X^{-1}, Y^{-1}$ are bounded and $- G_\mu(0) = \tau(X^{-1}), -G_\nu(0)=\tau(Y^{-1}), -G_{\mu\boxtimes\nu}(0) = \tau((\sqrt{X} Y \sqrt{X})^{-1})$. Using the equality $(\sqrt{X} Y \sqrt{X})^{-1} = \sqrt{X^{-1}} Y^{-1}\sqrt{X^{-1}}$, we obtain $\tau(\sqrt{X} Y \sqrt{X})^{-1}) = \tau(X^{-1}Y^{-1})=\tau(X^{-1})\tau(Y^{-1})$, which is the desired \eqref{eq:inverse_mean}. We are not sure if this kind of operator-theoretic argument or an approximation argument works to deal with the general setting of Lemma \ref{lem:T}.  
\end{enumerate}
\end{remark}

\begin{proof}

Let $\omega_1, \omega_2$ be the subordination functions for $\eta_{\mu \multconv \nu}$ as in Theorem \ref{thm:subordination}. 

{
\vspace{2mm}\noindent
{\bf Step 1.} We show $(\mu\boxtimes\nu)(\{0\})=0$. 
Recall that $\omega_1(z)$ is the DW-point of $f_z$ defined in Proposition \ref{prop:not_conf}. Observe that 
\[
f_{z}(w) = z h_\nu(z h_\mu(w)) = \frac{\eta_\nu(z h_\mu(w))}{h_\mu(w)} \to \infty   \quad \text{as} \quad z\nontangtou \infty,  
\]
where the last convergence holds by the assumption $\nu(\{0\})=0$ and Proposition  \ref{prop:prop_of_eta0}  \ref{item:eta_positive3}. Together with this fact, applying Theorem \ref{thm:conv_of_dwp} to the functions $f_n:=f_{z_n}$ and $f\equiv \infty$ for an arbitrarily selected sequence $z_n\nontangtou \infty$ in $\Hup$ shows 
\begin{equation}\label{eq:subord1_infty}
\nontanglim_{\substack{z\to \infty \\ z \in \Hup}} \omega_1(z)  =\infty. 
\end{equation}

We take any continuous curve $\{\gamma(t)\}_{t\in[0,1)} \subseteq \Hup$ such that $\gamma(t)\nontangtou \infty$ as $t\uparrow1$. 
Then we have, by Proposition  \ref{prop:prop_of_eta}  \ref{eq:eta_at_infty_zero},   
\begin{equation}\label{eq:atom0}
1-\frac{1}{(\mu\boxtimes\nu)(\{0\})}=\lim_{t\uparrow1}\eta_{\mu\boxtimes\nu}(\gamma(t)) =\lim_{t\uparrow1}\eta_{\mu}(\omega_1(\gamma(t))).  
\end{equation}
Since   $\omega_1(\gamma(t))\to\infty$  as shown in \eqref{eq:subord1_infty},    
by the Lindel\"of theorem, the above limits \eqref{eq:atom0} also equal $\nontanglim_{z\to\infty}\eta_\mu(z)$, which is $\infty$ by the assumption $\mu(\{0\})=0$. Therefore, we have $(\mu\boxtimes\nu)(\{0\})=0$. 
}


\vspace{2mm}
\noindent
{\bf Step 2.} We prove 
\begin{equation}\label{eq:subord_infty} 
\nontanglim_{\substack{z\to \infty \\ z \in \Hup}} \omega_2(z) =\infty. 
\end{equation}
Suppose $\omega_2(z_n) \to c \in \CC \cup \{\infty\}$ for some sequence $\{z_n\} \subset \Hup$ tending to $\infty$ nontangentially. It suffices to show $c =\infty$. 
Note first that $c \in\{\infty\} \cup \{0\}\cup \Delta_\theta$ for some $\theta\in(0,\pi)$ since $\arg z \le \arg \omega_2(z) \le \arg z +\pi$.  The case $c \ne \infty$ leads to a contradiction: in the identity $\eta_{\mu\multconv \nu}(z_n) = \eta_\nu(\omega_2(z_n))$, the left hand side would tend to infinity because $\mu\multconv \nu$ has no atom at 0 and by Proposition  \ref{prop:prop_of_eta}  \ref{eq:eta_at_infty_zero}, while the right hand side would converge to the finite value $\eta_{\nu}(c)$. 

\vspace{2mm}
\noindent
{\bf Step 3.}  Recalling from \eqref{eq:F_psi_eta} that 
\[
F_\mu\left(\frac1{z} \right) = \frac{1-\eta_\mu(z)}{z}, \qquad \psi_\mu(z) = \frac{\eta_\mu(z)}{1-\eta_\mu(z)},  \qquad z\in \Hup
\]
we obtain 
\begin{align*}
F_\mu\left(\frac1{\omega_1(z)} \right)F_\nu\left(\frac1{\omega_2(z)} \right) \psi_{\mu\multconv \nu}(z)
&= \frac{1-\eta_\mu(\omega_1(z))}{\omega_1(z)} \cdot \frac{1-\eta_\nu(\omega_2(z))}{\omega_2(z)}  \cdot \frac{\eta_{\mu\multconv \nu}(z)}{1-\eta_{\mu\multconv \nu}(z)} \\
&= \frac{1-\eta_{\mu\multconv \nu}(z)}{z} = F_{\mu\multconv\nu}\left(\frac1{z}\right). 
\end{align*}
As $z\nontangtou\infty~(z\in \Hup)$, by \eqref{eq:psi_at_infty_zero}, we have $\psi_{\mu\multconv \nu}(z)\to -1+(\mu\multconv \nu)(\{0\})=-1$, so that the desired \eqref{eq:inverse_mean} follows from \eqref{eq:subord1_infty} and \eqref{eq:subord_infty}. 
\end{proof}

\begin{proof}[\textbf{Proof of Theorem \ref{thm:T}  \ref{itemT1}}] If either $\mu$ or $\nu$ equals $\delta_0$ then $\mu\boxtimes\nu =\delta_0$, so that the identity $T_{\mu\boxtimes\nu}=T_\mu T_\nu\equiv0$ holds obviously. We henceforth assume $\mu,\nu\ne\delta_0$ and set $c:= \max\{\mu(\{0\}),\nu(\{0\})\} \in[0,1)$. By \cite[Proposition 5.1]{BB2005partially}, $\mu^{\boxplus 1/s}(\{0\}) = \nu^{\boxplus 1/s}(\{0\})=0$ for all $0< s < 1-c$. 
Using the identity
\eqref{eq:BN_free} and applying Lemma \ref{lem:T} for measures $\Dil_s(\mu^{\addconv 1/s})$ and $\Dil_s(\nu^{\addconv 1/s})$, we deduce 
\begin{align}
T_{\mu\multconv\nu}(-s) 
&= - F_{\Dil_s((\mu\multconv\nu)^{\addconv 1/s})} (0) 
=  -  F_{\Dil_s(\mu^{\addconv 1/s})\multconv \Dil_s(\nu^{\addconv 1/s})} (0) \notag \\ 
&= F_{\Dil_s(\mu^{\addconv 1/s})}(0) F_{\Dil_s(\nu^{\addconv 1/s})} (0) = T_\mu(-s) T_\nu(-s) \label{eq:multiplicativity_T}
\end{align}
for all $0<s < 1-c$. 

If $c=0$, we are already done. If $c>0$, then by the continuity of $T$-transform (Proposition \ref{prop:omega} \ref{itemO2}), the formula \eqref{eq:multiplicativity_T} also holds for $s=1-c$ and the value is 0 by Proposition \ref{prop:omega} \ref{itemO1}. The same proposition also implies that \eqref{eq:multiplicativity_T}  holds for $1-c <s <1$ and the value is also 0.  
\end{proof}

\begin{proof}[\textbf{Proof of Theorem  \ref{thm:T}  \ref{itemT2}}] 
Assume that $T_\lambda = T_\mu$ on $(-\ee,0)$ for some $\ee \in(0,1)$. This is equivalent to $\sof_{\lambda,t}(0) = \sof_{\mu,t}(0) =:\sof(t)$ for $t>1/\ee$. Moreover, \eqref{eq:F} entails $F_\lambda(\sof(t)) = F_\mu(\sof(t))$,  in the sense of nontangential limits from $\CC^-$ if $\sigma(t) \in \RR$. 

\vspace{2mm}
\noindent
{\bf Case 1:} $T_\mu=0$. This means $\sof_{\mu,t}(0)=0$ for all $t>1/\ee$.  On the other hand, recall that $\sof_{\mu,t}(0)$ is the DW-point of $L_{\mu, t,0}(w) = (t-1)(F_{\mu}(w)-w)$. If $F_{\mu} \ne \text{id}$ then $L_{\mu,t,0}(w)\to \infty$ as $t\to+\infty$, so that  $\sof_{\mu,t}(0) \to \infty$, a contradiction. Therefore, it must hold that $F_{\mu} = \text{id}$, i.e., $\mu =\delta_0$. 

\vspace{2mm}
\noindent
{\bf Case 2:} $T_\mu\ne0$ and there exists $t >1/\ee$ such that $\sof(t) \in \CC^-$. By Proposition  \ref{prop:omega}  \ref{itemO2}, the set $\{\sof(t): t>1/\ee\}\cap \CC^-$ contains an accumulative point in $\CC^-$ and so by the identity theorem $F_\lambda = F_\mu$ on $\CC^-$, which implies $\lambda=\mu$. 

\vspace{2mm}
\noindent
{\bf Case 3:} $T_\mu\ne0$ and $\sof(t) \in \RR$ for all $t>1/\ee$. 
By Proposition  \ref{prop:omega}  \ref{itemO2}, the trajectory $\{\sof(t): t>1/\ee\}$ contains an open interval $I \subseteq \RR$. By the Privalov theorem, $F_\lambda = F_\mu$ on $\CC^-$, so that $\lambda=\mu$. 
\end{proof}

\begin{proof}[\textbf{Proof of Theorem  \ref{thm:T}  \ref{itemT3}}]
Using the relations
\[
\frac{-s}{1-s} S_\mu(-s) = \frac1{\sof_{\mu,1/s}(0)}, \qquad \psi_\mu(z) = \frac1{z F_\mu(1/z)}-1~~(z\in \Hup)
\]
and \eqref{eq:F}, we obtain 
\begin{align}\label{eq:psi}
\psi_\mu\left( \frac{-s}{1-s} S_\mu(-s)\right)  &= \psi_\mu\left( \frac{1}{\sof_{\mu,1/s}(0)}\right) = \frac{\sof_{\mu,1/s}(0) }{F_\mu(\sof_{\mu,1/s}(0))}-1 =-s. 
\end{align}
\end{proof}

\begin{proof}[\textbf{Proof of Theorem \ref{thm:T}  \ref{itemT4} and  \ref{itemT5}}]
Assertion \ref{itemT4} is a special case of the continuity of the map \eqref{eq:conti}. As preparations for the proof of  \ref{itemT5}, let $\widehat{\RR}=\RR \cup\{\infty\}$ be the one-point compactification of $\RR$ and let $G_\lambda$ for $\lambda \in \Prob(\widehat{\RR})$ be defined as $G_{\lambda}(z) = \int_{\RR} \frac1{z-x}\lambda(\di x)$. Then $\lambda \mapsto G_\lambda$ is injective, $F_\lambda \coloneq 1/G_\lambda$ maps $\CC^-$ into $\CC^-\cup\{\infty\}$, and 
 \[
\nontanglim_{\substack{w\to\infty \\ w\in \CC^-}} \frac{F_{\lambda}(w)}{w} = \frac{1}{\lambda(\RR)} \in [1,+\infty]. 
\]

We regard $\{\mu_n: n\ge1\}$ as probability measures on $\widehat{\RR}$. Since $\{\mu_n: n\ge1\}$ is a uniformly bounded family of finite measures on the compact space $\widehat{\RR}$, it is tight and so by Prokhorov's theorem it has a subsequence (denoted as $\mu_{n'}$) that weakly converges to some $\mu \in \Prob(\widehat{\RR})$.  Since for each $z\in \CC^-$ the function $x\mapsto 1/(z-x)$ can be extended to a bounded continuous function on $\widehat{\RR}$, we have 
\[
\lim_{n'\to\infty} G_{\mu_{n'}}(z) = G_{\mu}(z), \qquad z \in \CC^-.  
\]
For each $t>1/\ee$ and $z\in \CC^-\cup \RR, w\in \CC^-$, we have $L_{\mu_{n'},t,z}(w) \to L_{\mu, t,z}(w)$, so that $\sof_{\mu_{n'},t}(z)$ converges to the DW-point $\sof_{\mu,t}(z)$ of $L_{\mu, t,z}$. By the assumption, we obtain $\sof_{\mu_{n'}, t}(0) = (1-t) T_{\mu_{n'}} (-1/t) \to (1-t) T(-1/t)  = \sof_{\mu,t}(0)$.  
If $\mu$ charged a positive mass at $\infty$, then it would entail  $\nontanglim_{w\to\infty}F_\mu(w)/w = 1/\mu(\RR)\in(1,+\infty]$, so that by Remark \ref{rem:DW}, $L_{\mu,t,0}(w) =(t-1)(F_\mu(w)-w)$ would have DW-point at $\infty$ whenever $(t-1)(1/\mu(\RR)-1)\ge1$.  
However, this would contradict the fact that the DW-point $\sof_{\mu,t}(0)= (1-t) T(-1/t)$ is finite. Hence, $\mu$ must be a probability measure on $\RR$. Because of $T_{\mu} = T$ on $(-\ee,0)$ and the already established  \ref{itemT2}, the limiting measure $\mu$ is independent of a choice of a convergent subsequence of $\{\mu_n: n\ge1\}$. The standard compactness argument shows that the original sequence $\mu_n$ weakly converges to $\mu$.   
\end{proof}

The following identity generalizes that in  \cite[Lemma 7.1]{AFPU} since the Cauchy distribution $\mathbf{c}_{a,b}$ converges weakly to $\delta_a$ as $b\to 0^+$. 
\begin{proposition}  Let $a\in \RR$, $b>0$, and $\mathbf{c}_{a,b}$ be the Cauchy distribution with density 
\begin{equation}\label{eq:cauchy_distribution}
\frac{\di{\mathbf{c}_{a,b}}}{\di{x}}=\frac{1}{\pi} \cdot \frac{b}{(x-a)^2 +b^2}, \qquad x \in \RR. 
\end{equation}
For any  $\mu\in \Prob(\RR)$ we have 
\[
S_{\mu\addconv \mathbf{c}_{a,b}}(-t) = - G_{\Dil_{t}(\mu^{\addconv \frac1{t}})}(-a - \iu b), \qquad 0<t<1. 
\]
\end{proposition}
\begin{proof} 
This follows from the definition of $S$-transform and the facts $\Dil_{t}(\mathbf{c}_{a,b}^{\addconv \frac1{t}}) = \mathbf{c}_{a,b}$ and  
\[
G_{\mu \addconv  \mathbf{c}_{a,b}} (w) = G_{\mu \ast  \mathbf{c}_{a,b}} (w)  = G_{\mu} (w -a - \iu b), \qquad w \in \CC^-. 
\]
\end{proof}

\subsection{Positive and symmetric cases}

We characterize measures supported on $\RR_{\ge0}$ and symmetric measures on $\RR$, and then observe that our definition coincides with the previous ones in the literature.  

\begin{proposition}\label{prop:PS} 
Let $\mu \in \Prob(\RR)$. 
\begin{enumerate} 
\item\label{itemP}
$\mu$ is supported on $\RR_{\ge0}$ if and only if $T_\mu \ge0$ on $(-1,0)$. 
\item\label{itemS} $\mu$ is symmetric with respect to the origin if and only if $-\iu T_\mu \ge  0$ on $(-1,0)$.   
\end{enumerate}
\end{proposition}
\begin{proof}   We assume that $\mu$ is nondegenerate because the statement is obvious otherwise.

\vspace{2mm}
\noindent
{\bf ``Only if'' part of  \ref{itemP}:} We fix $t>1$.  Suppose that $\mu$ is supported on $\RR_{\ge0}$.   Then $F_\mu$ extends to an analytic function in $\CC\setminus [0,+\infty)$ with $F_\mu(\conj{z}) = \conj{F_\mu(z)}$, $F_\mu(0-) \le0$ (see e.g.~\cite[Proposition 2.5]{Has10}) and $x \mapsto F_\mu(x) -x$ is strictly increasing on $(-\infty,0)$; these facts can be proved from 
the Pick--Nevanlinna representation \eqref{eq:PN}
and the Stieltjes inversion.  
Hence,  for any $x<0$ and $z <0$, we have $L_{\mu,t,z}(x)<0$. This shows $L_{\mu,t,z}$ is an analytic self-map of $\CC\setminus [0,+\infty)$ for each $z \in \CC\setminus [0,+\infty)$. Moreover, for $z<0$ we have $\lim_{x\to-\infty} (L_{\mu,t,z}(x) -x) =+\infty$ and $\lim_{x\to0^-} (L_{\mu,t,z}(x) -x) \le z <0$, showing that $L_{\mu,t,z}$, as a self-map of $\CC\setminus [0,+\infty)$, has a DW-point $\sof_{\mu,t}(z)$ in $(-\infty,0)$.  Passing to the limit $z\to 0^-$ and again by Theorem  \ref{thm:conv_of_dwp}, we get $\sof_{\mu,t}(0) \in (-\infty,0]$ and so $T_\mu(-1/t) = \sof_{\mu,t}(0)/(1-t) \ge0$. 

\vspace{2mm}
\noindent
{\bf ``If'' part of  \ref{itemP}:} Recall from Proposition \ref{prop:omega} that $\sof_{\mu,1}(0)=0, \lim_{t\to+\infty}\sof_{\mu,t}(0)=\infty$ and $\sof_{\mu,t}(0)$ is continuous on $[1,+\infty)$. Together with assumption $T_\mu \ge0$, this yields $\{\sof_{\mu,t}(0): t\ge1\} = (-\infty,0]$. From \eqref{eq:F} $F_\mu$ has nontangential limits on $(-\infty,0)$ that are nonzero real values, so that by the Stieltjes inversion (Theorem \ref{thm:ct_nontang}), $\mu((-\infty,0))=0$.

\vspace{2mm}
\noindent
{\bf ``Only if'' part of  \ref{itemS}:} We fix $t>1$. Suppose that $\mu$ is symmetric with respect to $0$. For all $w \in \iu \RR_{<0}$,  $F_\mu(w) -w$ lies in $\iu \RR_{<0}$, so that for all  $ z \in \iu \RR_{\le 0}$ and $w  \in \iu \RR_{<0}$, we have $L_{\mu,t,z}(w) \in \iu \RR_{<0}$. Moreover, for $z \in \iu \RR_{<0}$, we have $\lim_{y\to - \infty} (L_{\mu,t,z}(\iu y) -\iu y) / \iu=+\infty$ and $\lim_{y\to 0^-} (L_{\mu,t,z}(\iu y) -\iu y) / \iu = \Im(z) +(t-1)\Im[F_\mu(\iu 0^-)]<0$, so that there is a fixed point of $L_{\mu,t,z}$ in $\iu \RR_{<0}$, which must be $\sof_{\mu,t}(z)$. Passing to the limit $z\to0~(z\in \iu\RR_{<0})$, we conclude that $\sof_{\mu,t}(0) \in \iu \RR_{\le 0}$ and hence $T_\mu(-1/t) \in  \iu \RR_{\ge 0}$.

\vspace{2mm}
\noindent
{\bf ``If'' part of  \ref{itemS}:} From reasonings similar to the proof of  \ref{itemP}, $\psi_\mu$  takes real values  on $\iu \RR_{>0}$. By the Schwarz reflection principle, $\psi_\mu(-x + \iu y) = \psi_\mu(x +\iu y)$ holds for all $x \in \RR$ and $y>0$. The Stieltjes inversion formula implies that $\mu$ is symmetric. 
\end{proof}

Using Theorem \ref{thm:T}  \ref{itemT3}, Equation \eqref{eq:Tat0} and Proposition  \ref{prop:PS}, we can prove that our $S$-transform coincides with the former definitions in special cases: 
\begin{itemize}[itemsep=0.2em,topsep=0.3em]
\item measures with compact support and nonvanishing mean  \cite{V1987multiplication}, 
\item measures with compact support and vanishing mean  \cite{RS07}, 
\item measures supported on $\RR_{\ge0}$  \cite{BV1993free}, 
\item  measures symmetric with respect to the origin  \cite{AP2009transform}. 
\end{itemize} 
In fact, for measures on $\RR_{\ge0}$, Arizmendi et al.~\cite[Lemma 7.1]{AFPU} already proved that our definition of $S$-transform coincides with that of   \cite{BV1993free}.

\subsection{Examples of $S$-transform}

Before calculating examples, we note a useful general method.

\begin{remark}[Calculations of $S$-transform]  \label{rem:calc_S}
It is easy to see that the equation $L_{\mu,t,0}(w) =w$ is equivalent to  $\psi_\mu(1/w) = -1/t$ (with understanding that $L_{\mu,t,0}(w)$ and $\psi_\mu(1/w)$ are nontangential limits when $w\in \RR$).  From Theorem   \ref{thm:denjoy_wolff}, the equation $\psi_\mu(1/w) = -1/t$ has at most one solution in $\CC^-$ and, if a solution exists, it coincides with the DW-point $\sof_{\mu,t}(0)$. On the other hand, if $\psi_\mu(1/w) = -1/t$ has no solutions in $\CC^-$ and has more than one solutions $w$ in $\RR$, then  we need further reasonings to identify the DW-point on the boundary.  
\end{remark}

\begin{example} \label{ex:BS} For the Boolean stable law \eqref{eq:stable}, we have 
\begin{align*}
L_{\mathbf{b}_{\alpha,\rho}, t, z} (w)  
&= z+ (t-1) (F_{\mathbf{b}_{\alpha,\rho}}(w)-w) = z -(t-1) w \eta_{\mathbf{b}_{\alpha,\rho}}\left(\frac1{w}\right)\\
&=  z + (t-1) w \left(\frac{e^{-\iu \rho\pi}}{w}\right)^\alpha, \qquad w \in \CC^-\cup\RR,~ t>1,~z \in \CC^-\cup\RR. 
\end{align*}
The fixed point equation $L_{\mathbf{b}_{\alpha,\rho}, t, 0} (w) =w$ has solutions  
$
w =0, e^{-\iu \rho \pi} (t-1)^\frac{1}{\alpha}
$ 
 in $\CC^-\cup\RR$ ($w=0$ is a solution only if $0<\alpha<1$). The solution $0$ is not a DW-point because $\sof_{\mathbf{b}_{\alpha,\rho}, t}(0)$ must be continuous on $[1,+\infty)$ and is not identically zero.  Hence, $\sof_{\mathbf{b}_{\alpha,\rho}, t}(0) = e^{-\iu \rho \pi} (t-1)^\frac{1}{\alpha}$
and  
\[
S_{\mathbf{b}_{\alpha,\rho}} (u)  = - \frac{1+u}{-u} \cdot \frac{1}{\sigma_{\mathbf{b}_{\alpha,\rho}, -1/u}(0)} = -e^{\iu \rho \pi}  \left(\frac{-u}{1+u} \right)^\frac{1-\alpha}{\alpha}, \qquad -1 < u < 0. 
\]
\end{example}

\begin{example}\label{ex:FS} For notational brevity, let $ G\coloneq G_{\mathbf{f}_{\alpha,\rho}}$. Recall from \eqref{eq:stable} that the free stable law $\mathbf{f}_{\alpha,\rho}$ is characterized by  
$C_{\mathbf{f}_{\alpha,\rho}}(z) = z G^\inv(z) - 1 =  - (e^{-\iu \rho \pi}z)^\alpha.$ 
Putting  $z= G(w)$ yields  
\begin{equation}\label{eq:FS1}
wG(w) - 1 =  -(e^{-\iu \rho \pi}G(w))^\alpha, \qquad  w \in \CC^-\cup\RR. 
\end{equation}
Note that $G$ extends to a continuous map from $\CC^-\cup\RR$ to $\Hup \cup \RR\cup\{\infty\}$ since $\mathbf{f}_{\alpha,\rho}$ is freely infinitely divisible, so that the above equation can be considered for $ w \in \CC^-\cup\RR$. 
The DW-point $\sigma_{\mathbf{f}_{\alpha,\rho},t}(0)$ of $L_{\mathbf{f}_{\alpha,\rho}, t,0}$ is a solution $w$ to the equation $w = (t-1)(1/G(w)-w)$, so that $G(w) = (t-1)/(tw)$. Hence, we have either $w=0$ or, by substituting $G(w) = (t-1)/(tw)$ into \eqref{eq:FS1},  
\[
\frac{1}{t} = \left(e^{-\iu\rho\pi} \frac{t-1}{t w} \right)^\alpha, \qquad w\ne0,  
\]
and so $w= e^{-\iu \rho \pi} t^{\frac{1-\alpha}{\alpha}} (t-1)$. Because $w=0$ can be excluded for the same reason as Example \ref{ex:BS}, we have $\sof_{\mathbf{f}_{\alpha,\rho},t}(0)= e^{-\iu \rho \pi} t^{\frac{1-\alpha}{\alpha}} (t-1)$ and so 
\[
S_{\mathbf{f}_{\alpha,\rho}}(u) = - \frac{1+u}{-u}\cdot \frac{1}{ \sigma_{\mathbf{f}_{\alpha,\rho},-1/u}(0)} = - e^{\iu \rho \pi} (-u)^{\frac{1-\alpha}{\alpha}}, \qquad -1 < u <0. 
\]  
\end{example}

The following example shows that the $S$-transform need not be analytic in the domain. 

\begin{example}\label{ex:SS}
Let $\mu$ be the semicircle law of mean $a \ge0$ and variance $1$, i.e., 
\[
\mu(\di x)= \frac{1}{2\pi}\sqrt{4-(x-a)^2} \mathbf{1}_{(-2+a,2+a)}(x)\di{x}. 
\]
Using $F_\mu(w) = [w-a + \sqrt{(w-a)^2-4}]/2$, we can solve the equation $L_{\mu,t,0}(w) =w$ as 
\[
w = \frac{t-1}{2t}(-a t \pm \sqrt{a^2t^2-4t}).  
\]
Since the DW-point is in $\CC^-\cup\RR$,  continuous and tends to $\infty$ as $t\to+\infty$,  the correct choice of the square root and the sign in front leads to 
\[
\sigma_{\mu,t}(0)= 
\begin{cases}
\frac{t-1}{2t}(-a t - \sqrt{a^2t^2-4t})& \text{if~}  t> \max\{1,4/a^2\},  \\ 
\frac{t-1}{2t}(-a t - \iu \sqrt{4t-a^2t^2})& \text{if~} 0< a < 2~\text{and}~1< t \le 4/a^2,  \\
-\iu (t-1)/\sqrt{t} & \text{if~} a=0~\text{and}~t>1,  
\end{cases}
\]
where $\sqrt{x} \ge0$ is the standard square root when $x\ge0$. From this we can find the $S$-transform $S_\mu(u) = (1+u)/(u \sigma_{\mu,-1/u}(0))$ as 
\[
S_{\mu}(u)= 
\begin{cases}
\frac{1}{2u}(-a+ \sqrt{a^2+4u})& \text{if~}  -\min\{1,a^2/4\}<u<0,  \\
\frac{1}{2u}(-a + \iu \sqrt{-a^2-4u})& \text{if~} 0< a < 2~\text{and}~-1<u\le-a^2/4,  \\
-\iu/\sqrt{-u}&  \text{if~} a=0~\text{and}~-1<u<0.   
\end{cases}
\] 
The case $a<0$ can be handled similarly. Notice that  $S_\mu(u)$ is not analytic at the point $u=-a^2/4$ in case $|a|<2$.  
\end{example}

\section{Boolean and free stable laws}   \label{sec:S}

With the $S$-transform in hand, we are now in a position to prove Theorems \ref{thm:reproducing} and \ref{thm:mixture}. Then we apply the results to study the free L\'evy measure of Boolean stable mixtures. One can also use subordination functions in many cases, but the $S$-transform makes the proofs simpler. 

\subsection{Proofs of Theorems  \ref{thm:reproducing},  \ref{thm:mixture} and related identities} \label{sec:boole1}

\begin{proof}[\textbf{Proof of Theorem  \ref{thm:reproducing}}] 
Examples  \ref{ex:BS} and  \ref{ex:FS} readily yield $S_{\mathbf{b}_{\alpha\beta,\rho}}(u)= S_{\mathbf{b}_{\alpha,1}}(u) S_{\mathbf{b}_{\beta,1}}(u)^\frac1{\alpha}$ and $S_{\mathbf{f}_{\alpha\beta,\rho}}(u)= S_{\mathbf{f}_{\alpha,1}}(u) S_{\mathbf{f}_{\beta,1}}(u)^\frac1{\alpha}$, so that Theorem  \ref{thm:T} implies the desired formulas.  
\end{proof}

\begin{proof}[\textbf{Proof of Theorem  \ref{thm:mixture}}]
Since the formula is obvious for $\nu=\delta_0$, we assume that $\nu \ne \delta_0$. We begin with computing the $S$-transform of $\mathbf{b}_{\alpha,\rho} \circledast (\nu^{\frac1{\alpha}})$. Recall from \eqref{eq:eta_b} that 
\begin{equation*} 
\psi_{\mathbf{b}_{\alpha,\rho} \circledast (\nu^{\frac1{\alpha}})} (1/w) = \psi_\nu(-(e^{-\iu\rho\pi}/w)^\alpha).  
\end{equation*}
The function $\psi_\nu$ restricts to a bijection from $(-\infty,0)$ onto $(-1+\nu(\{0\}), 0)$ and the inverse map $\psi_\nu^\inv(u)$ is given by $\frac{u}{1+u}S_\nu(u)$. 
So, if $0<\rho<1$ and $u \in (-1+\nu(\{0\}), 0)$ then the equation $\psi_{\mathbf{b}_{\alpha,\rho} \circledast (\nu^{\frac1{\alpha}})} (1/w) = u$ has a solution $w = e^{-\iu \rho \pi}(- \psi^\inv_\nu(u))^{-1/\alpha}\in\CC^-$. From Remark  \ref{rem:calc_S}, the DW-point of $L_{\mathbf{b}_{\alpha,\rho} \circledast (\nu^{\frac1{\alpha}}), -1/u,0}$ is $w$, and so 
\begin{equation}\label{eq.Stranmixture}
S_{\mathbf{b}_{\alpha,\rho} \circledast (\nu^{\frac1{\alpha}})}(u) = -\frac{1+u}{-uw} = -e^{\iu \rho \pi} \frac{1+u}{-u}\left(\frac{-u}{1+u}\right)^{\frac1{\alpha}} S_\nu(u)^\frac1{\alpha} = S_{\mathbf{b}_{\alpha,\rho}}(u)S_{\nu^{\multconv\frac1{\alpha}}}(u).  
\end{equation}
This formula still holds for $\rho=0,1$ because the $S$-transform is continuous with respect to the weak convergence, see Theorem  \ref{thm:T} \ref{itemT4}. 
Then Theorem  \ref{thm:T}  \ref{itemT1} and  \ref{itemT2} finish the proof. 
\end{proof}

The special case $\alpha=1$ of Theorem \ref{thm:mixture} yields the following ``folklore formula''.  
\begin{corollary} For any $\nu \in \Prob(\RR_{\ge0})$ we have 
\[
\mathbf{c}_{a,b} \circledast \nu = \mathbf{c}_{a,b}\multconv \nu, 
\]
where $\mathbf{c}_{a,b}$ is the Cauchy distribution defined in \eqref{eq:cauchy_distribution}. 
\end{corollary}
\begin{proof}
Notice that $\mathbf{c}_{a,b} = \mathbf{b}_{1,\rho}$ for $0 < \rho <1$ and $a=\cos\rho\pi, b=\sin\rho\pi$, so that the desired formula is exactly Theorem  \ref{thm:mixture} for $\alpha=1$ in this case. The general case follows from the fact that $\Dil_{r}(\mathbf{c}_{a,b}) = \mathbf{c}_{ar,br}$ and $\Dil_r (\mu \multconv \nu) = (\Dil_r \mu)\multconv\nu$, $r>0$. 
\end{proof}

Another related convolution formula is obtained as follows. 

\begin{proposition}\label{cor:bs-fs} Let $0<\alpha \le 1, 0\le \rho \le 1$. Then 
$\mathbf{b}_{\alpha,\rho} =  \mathbf{f}_{\alpha,\rho} \multconv (\MP^{\multconv \frac{1-\alpha}{\alpha}})$, 
where $\MP$ is the standard Marchenko--Pastur law with density 
\begin{align} \label{eq:MP}
   \frac{\di{\MP}}{\di{x}}= \frac{1}{2\pi}\sqrt{\frac{4-x}{x}} \mathbf{1}_{(0,4)}(x). 
\end{align}

\end{proposition}
\begin{proof}
Similar to Theorem  \ref{thm:reproducing}, this follows from Examples \ref{ex:BS},   \ref{ex:FS} and the well-known formula $S_\MP(u) = 1/(1+u)$. 
\end{proof}

\begin{remark}[Connection to classical stable laws]\label{rem:classical}
As mentioned in \cite[Proposition 4.12]{AH2016classical}, the measure $ \MP^{\multconv \frac{1-\alpha}{\alpha}}$ equals $ (\mathbf{f}_{\alpha,1})^{-1}$, i.e., the law of $1/Y$ when a random variable $Y$ has distribution $\mathbf{f}_{\alpha,1}$. Hence, Proposition \ref{cor:bs-fs} combined with \cite[Proposition 4.12]{AH2016classical} yields 
\[
\mathbf{b}_{\alpha,\rho} =  \mathbf{f}_{\alpha,\rho} \multconv (\mathbf{f}_{\alpha,1})^{-1} = \mathbf{n}_{\alpha,\rho} \circledast (\mathbf{n}_{\alpha,1})^{-1}, \qquad 0<\alpha \le 1,\quad 0\le \rho \le 1,  
\]
where $\mathbf{n}_{\alpha,\rho}$ is the classical stable law characterized by 
\[
\int_\RR e^{zx} \mathbf{n}_{\alpha,\rho}(\di x) = \exp\left[- (e^{-\iu\rho\pi}z)^\alpha\right], \qquad z\in \iu \RR_{>0}. 
\]
\end{remark}

\subsection{Characterizing Boolean stable mixtures in terms of $S$-transform} \label{sec:characBoole}
As a generalization of Proposition  \ref{prop:PS}, we can get a characterization of Boolean stable mixtures in terms of the behaviour of the $S$-transform. We need the following result of J. R. Guerrero \cite[Theorem 4]{thesisGuerrero}. Since the original proof contained an error (the symmetric Bernoulli law gives a counterexample to \cite[Lemma 4]{thesisGuerrero}), we give another proof below. 

\begin{lemma} \label{prop:characBoolMix} Let $\mu \in \Prob(\RR)\setminus\{\delta_0\}$ and $\rho \in (0,1)$. We denote $\mathcal{L}_{\rho}\coloneqq \{e^{\iu\rho\pi}t: t>0\}$ and 
\[
\alpha(\rho)\coloneq \max\{\alpha: (\alpha,\rho) \in \mathfrak{A}\}= \min\left\{\frac1{\rho},\frac1{1-\rho}\right\}.\]
The following are equivalent. 

\begin{enumerate}

\item\label{item:BM1} There exists some $\nu \in \Prob(\RR_{\ge0})$ such that $\mu= \mathbf{b}_{\alpha(\rho),\rho}\circledast (\nu^{1/\alpha(\rho)})$. 

\item\label{item:BM2} $\psi_{\mu}(\mathcal{L}_{\rho})\subseteq(-1,0)$. 

\end{enumerate}
\end{lemma}

\begin{proof}
It is easy from \eqref{eq:eta_b} to see that \ref{item:BM1} implies \ref{item:BM2}. 
For the converse implication, we may assume that $\rho\in[1/2,1)$; the other case $\rho\in(0,1/2)$ can be treated by the symmetry $x\mapsto -x$. For $\rho\in[1/2,1)$ we have $\alpha(\rho)=1/\rho$. Then we define 
\[
F(z)\coloneq z^{-\rho+1} F_\mu(z^\rho), \qquad z \in \CC^+.  
\]
We show that $F$ is a reciprocal Cauchy transform of a measure $\nu \in \Prob(\RR_{\ge0})$. Once this is proved then we have $F_\mu(z) = z^{1-\frac1{\rho}} F_\nu(z^{\frac1{\rho}})$ or equivalently $\eta_\mu(z) = \eta_\nu(z^\frac1{\rho})$, which implies from \eqref{eq:eta_b} that $\mu = \mathbf{b}_{1/\rho,\rho} \circledast (\nu^{\rho})$ as desired.  

To find such $\nu$, the main task is to show that $F(\Hup)\subseteq \Hup\cup\RR$. This can be achieved as follows. 

\vspace{2mm}
\noindent
\textbf{Case 1.} For a technical reason, first we assume $\mu(\{0\})>0$. We study the behavior of $F$ on the curve $\gamma:=  (-R,-\delta) \cup\{\delta e^{\iu\theta}: \ee \le \theta \le \pi\} \cup \{x e^{\iu\ee}: \delta \le x \le R \} \cup  \{Re^{\iu\theta}: \ee \le \theta \le \pi\}$ for $\ee \in (0,\pi)$  and $0<\delta<R$ as follows.  

\begin{itemize}[itemsep=0.2em, topsep=0.3em]
\item First observe that $F$ extends continuously to $(-\infty,0)$. 

\item Since assumption \ref{item:BM2} implies $F_\mu(\mathcal{L}_{\rho}) \subseteq \mathcal{L}_{\rho}$, we have, for all $x>0$,  
\begin{equation}\label{eq:F_real}
    F(-x) = x^{1-\rho}e^{\iu (1-\rho) \pi} F_\mu(x^{\rho}e^{\iu \rho \pi}) \in (-\infty,0).
\end{equation}

\item For $x>0$, we have $F(xe^{\iu \ee}) = x^{1-\rho} e^{\iu \ee(1-\rho)} F_\mu(x^\rho e^{\iu \rho \ee})$, so that 
\[\ee(1-\rho) <
\arg F(xe^{\iu \ee})< \ee(1-\rho)+\pi.\] 

\item Using Proposition \ref{prop:characF} \ref{item:F3} we see that 
\begin{equation}\label{eq:F_infty}
\lim_{\substack{z\to\infty\\\arg z \in [\ee, \pi]}} \frac{F(z)}{z}=\lim_{\substack{w\to\infty\\ \arg w \in [\ee\rho, \pi \rho]} }\frac{F_\mu(w)}{w}=1.
\end{equation}

\item Using Theorem  \ref{thm:ct_nontang} \ref{item:G2}, we see that 
\[
F(z) =  \frac{z}{z^\rho G_\mu(z^\rho)} = \frac{z}{\mu(\{0\})}(1+o(1))\quad \text{as}\quad  z \to 0~(\arg z \in [\ee, \pi]).
\] 

\end{itemize}
Combining the above behaviors, we have 
$\arg F(z) \in [(1-\rho)\ee, (1-\rho)\ee +\pi]$ on the curve $\gamma$ for sufficiently large $R>0$ and small $\delta>0$. Since $z\mapsto \arg F(z) = \Im[\log F(z)]$ is harmonic, it takes the maximum and minimum on the boundary, so that $\arg F(z) \in [(1-\rho)\ee, (1-\rho)\ee +\pi]$ for all $z$ in the bounded domain surrounded by $\gamma$. By letting $\delta,\ee\to0^+$ and $R\to+\infty$, we conclude that $F(z) \in \Hup \cup \RR$ for all $z \in \Hup$ as desired. 

\vspace{2mm}\noindent
\textbf{Case 2.} If $\mu(\{0\})=0$, then let $\mu_n:= (1/n)\delta_0 + (1-1/n)\mu$. This measure still satisfies \ref{item:BM2}, so that the function $F_n(z) \coloneqq z^{-\rho+1} F_{\mu_n}(z^\rho)$ takes values in $\Hup\cup\RR$. Since $F_n$ tends to $F$, the limiting function $F$ also takes values in $\Hup \cup \RR$. 

\vspace{2mm}\noindent
\textbf{Conclusion.} 
From $F(\Hup)\subseteq \Hup\cup\RR$ and \eqref{eq:F_infty} above, together with Proposition \ref{prop:characF}, there exists a probability measure $\nu$ on $\RR$ such that $F=F_\nu$. Moreover, since $G_\nu=1/F_\nu$ takes negative values on $(-\infty,0)$ as observed in \eqref{eq:F_real}, by the Stieltjes inversion $\nu$ must be supported on $\RR_{\ge0}.$
\end{proof}

\begin{theorem}\label{prop:characBoole} 
 Let $\rho \in [0,1]$, $\mu \in \Prob(\RR)$ and $\alpha(\rho)$ be the number defined in Lemma \ref{prop:characBoolMix}. The following are equivalent. 
 
 \begin{enumerate}
 
 \item $\mu$ is a Boolean stable mixture with asymmetry parameter $\rho$.
 
 \item $\mu = \mathbf{b}_{\alpha(\rho),\rho} \circledast (\nu^{1/\alpha(\rho)})$ for some $\nu \in \Prob(\RR_{\ge0})$. 
 \item $e^{-\iu \pi \rho}  T_\mu \ge  0$ on $(-1,0)$. 
 
\end{enumerate}

\end{theorem}

\begin{proof} 
The equivalence between (1) and (2) is a consequence of \cite[Proposition 4.11]{AH2016classical} that is an immediate consequence of the formula
\[
\mathbf{b}_{\alpha,\rho} \circledast (\mathbf{b}_{\beta,1})^\frac{1}{\alpha} = \mathbf{b}_{\alpha\beta,\rho}, \qquad (\alpha,\rho)\in \mathfrak{A},~\beta \in (0,1], 
\]
which is an immediate consequence of \eqref{eq:eta_b}.  
The part $(2) \Longrightarrow (3)$ follows from formula \eqref{eq.Stranmixture} for the $S$-transform of $\mathbf{b}_{\alpha,\rho} \circledast (\nu^{\frac1{\alpha}})$.

For the remaining part $(3) \Longrightarrow (2)$, we may assume that $\mu\ne\delta_0$. For $\rho=1$ the result is easy because $\mathbf{b}_{1,1}=\delta_1$ and (3) just says $\mu$ is supported on $\RR_{\ge0}$ by Proposition \ref{prop:PS}. The case $\rho=0$ follows by symmetry. It remains to discuss $\rho\in(0,1)$. 
Observe that $\frac{u}{1+u}S_\mu(u) \in \mathcal{L}_\rho$ for all $\mu(\{0\})-1<u<0$. By Proposition \ref{prop:omega}, $1/\sigma_{\mu,-1/u}(0)=\frac{u}{1+u}S_\mu(u)$ is a bijection from $(\mu(\{0\})-1,0)$ onto $\mathcal{L}_\rho$. From Theorem \ref{thm:T} \ref{itemT3}, we have $\psi_\mu(\mathcal{L}_\rho)=(\mu(\{0\})-1,0)$ 
 and hence (2) follows by Lemma \ref{prop:characBoolMix}. 
\end{proof}

\begin{remark}
   This result is a generalization of Proposition \ref{prop:PS}. In case $\rho=1/2$, the measure $\mathbf{b}_{\alpha(\rho),\rho}$ is the symmetric Bernoulli law $\frac1{2}(\delta_{-1}+\delta_1)$, so that (2) is the well-known  characterization of symmetric probability measures. For $1/2 < \rho \le 1$, the extremal Boolean stable law $\mathbf{b}_{\alpha(\rho),\rho}$ has the form 
 \[
 \mathbf{b}_{1/\rho,\rho}(\di{x})=q_{1/\rho,1-\rho}(-x)\mathbf{1}_{(-\infty,0)}(x)\di{x} + \rho\delta_{1}, 
 \]  
   see \cite[Proposition 4 (6)]{HS2015unimodality} for detailed calculations.  
\end{remark}

\subsection{Boolean stable mixtures as compound free Poisson distributions}\label{sec:compound}

 Recall that the \emph{compound free Poisson distribution} $\MP(\lambda,\mu)$ with rate $\lambda >0$ and jump distribution $\mu \in \Prob(\RR\setminus\{0\})$ is the freely infinitely divisible distribution characterized by the free cumulant transform
\[
C_{\MP(\lambda,\mu)}(z)  = \lambda\int_{\RR} \frac{zx}{1-zx} \mu(\di x). 
\]
The special case $\MP(1,\delta_1)$ coincides with the Marchenko--Pastur law $\MP$ defined in \eqref{eq:MP}. For every $\mu \in \Prob(\RR)\setminus\{\delta_0\}$, $\mu\multconv \MP$ coincides with $\MP(1-\mu(\{0\}), \frac{1}{1-\mu(\{0\})}\mu|_{\RR\setminus\{0\}})$ (see  \cite[Remark 2.2(2)]{AH2016classical}). 

\begin{remark} 
In order to guarantee the uniqueness of jump distribution,  our definition of compound free Poisson distributions does not allow the jump distribution to have an atom at 0,  differing from  \cite{AH2016classical} and  \cite{NS2006lectures}. Our definition followed  \cite[Definition 3]{PaS2012free}. 
\end{remark}

The measure $\mathbf{b}_{1,\rho}$ is a delta measure or a Cauchy distribution, so it is rather trivially freely infinitely divisible. For other cases, it is known \cite{AH2014classical} that $\mathbf{b}_{\alpha,\rho}~(\alpha\ne1)$ is freely infinitely divisible if and only if $(\alpha,\rho) \in \mathfrak{A}_0$, 
where
\[
\mathfrak{A}_0:=\left\{(\alpha,\rho): 0 < \alpha \le \frac2{3}, \rho \in \left[2-\frac{1}{\alpha}, \frac1{\alpha}-1\right]\cap[0,1]\right\} \subseteq \mathfrak{A}.
\]
Moreover, according to \cite[Theorem 1.1]{AH2016classical}, Boolean stable mixtures $\mathbf{b}_{\alpha,\rho} \circledast (\nu^\frac1{\alpha})$ are freely infinitely divisible whenever $\nu\in\Prob(\RR_{\ge0})\setminus\{\delta_0\}$ and $(\alpha,\rho) \in \mathfrak{A}_0$. By \cite[Corollary 3.6]{HS2017unimodality}, $\mathbf{b}_{\alpha,\rho} \circledast (\nu^\frac1{\alpha})$ is a compound free Poisson distribution with rate one because its density is divergent at $0$, see \eqref{eq:density_boole}. 
In \cite[Proposition 4.21]{AH2016classical} we derived some expressions for the jump distribution in case $\rho =0,\frac{1}{2}$ or $1$. Below we provide more information on the jump distribution.

First, the following is the extension of  \cite[Proposition 4.21(1)]{AH2016classical} to general $\rho$'s. 
\begin{proposition}\label{cor:FL} Let $0<\alpha \le \frac{1}{2}, 0\le \rho \le 1$ and $\nu\in\Prob(\RR_{\ge0})\setminus\{\delta_0\}$. Then 
$\mathbf{b}_{\alpha,\rho} \circledast (\nu^{\frac1{\alpha}})$ is the compound free Poisson distribution of rate $1-\nu(\{0\})$ and jump distribution $\frac{1}{1-\nu(\{0\})}(\mathbf{f}_{\alpha,\rho} \multconv \MP^{\multconv\frac{1-2\alpha}{\alpha}} \multconv \nu^{\multconv \frac1{\alpha}})|_{\RR\setminus\{0\}}$. 
 \end{proposition}
 \begin{proof}
 The idea is the same as  \cite[Proposition 4.21]{AH2016classical}, i.e., Theorem  \ref{thm:mixture} and Proposition \ref{cor:bs-fs} yield 
 \[
 \mathbf{b}_{\alpha,\rho} \circledast (\nu^{\frac1{\alpha}} ) =  \mathbf{b}_{\alpha,\rho} \multconv (\nu^{\multconv \frac1{\alpha}})  =  (\mathbf{f}_{\alpha,\rho} \multconv \MP^{\multconv \frac{1-\alpha}{\alpha}}) \multconv \nu^{\multconv \frac1{\alpha}} = (  \mathbf{f}_{\alpha,\rho} \multconv  \MP^{\multconv \frac{1-2\alpha}{\alpha}} \multconv \nu^{\multconv \frac1{\alpha}} )\multconv \MP. 
 \]
 The atom of the measure $ \mathbf{f}_{\alpha,\rho} \multconv  \MP^{\multconv \frac{1-2\alpha}{\alpha}} \multconv \nu^{\multconv \frac1{\alpha}}$ at $0$ equals that of $\nu$ since $ \mathbf{f}_{\alpha,\rho}$ and $\MP^{\multconv \frac{1-2\alpha}{\alpha}}$ are Lebesgue absolutely continuous. 
 \end{proof}

Second, we calculate the jump distribution of $\mathbf{b}_{\alpha,\rho} \circledast (\nu^{\frac1{\alpha}})$ with a different method that works for all parameters $(\alpha,\rho)\in\mathfrak{A}_0$. We begin with $\nu=\delta_1$.

\begin{theorem}\label{thm:boole_LM} Let $(\alpha,\rho) \in \mathfrak{A}_0$ and  $p_{\beta,\sigma}$ be the density of the free stable law $\mathbf{f}_{\beta,\sigma}$.   
Then $\mathbf{b}_{\alpha,\rho} = \mu_{\alpha,\rho} \multconv \MP$, where $ \mu_{\alpha,\rho}$ is a probability measure on $\RR\setminus\{0\}$ that is Lebesgue absolutely continuous with density 
\begin{equation}\label{levy density}
\frac{\di\mu_{\alpha,\rho}}{\di x} = 
\begin{cases}
x^{\frac{\alpha}{1-\alpha}-1}p_{1-\alpha, \frac{\alpha\rho}{1-\alpha}}(x^{\frac{\alpha}{1-\alpha}}),&x>0, \\
(-x)^{\frac{\alpha}{1-\alpha}-1}p_{1-\alpha, \frac{\alpha\rho+1-2\alpha}{1-\alpha}}(-(-x)^{\frac{\alpha}{1-\alpha}}),&x<0. 
\end{cases}
\end{equation}
Notice that $\frac{\alpha\rho}{1-\alpha}, \frac{\alpha\rho+1-2\alpha}{1-\alpha} \in [0,1]$ whenever $(\alpha,\rho) \in \mathfrak{A}_0$. 
\end{theorem}
\begin{proof} The power functions $w\mapsto w^\beta$ below are all the principal values.  For the moment, we take an arbitrary $(\alpha,\rho) \in \mathfrak{A}$. 
We define a function $f$ in a cone $(-\Gamma_{0,\theta})\cap \{|z|>R\}$ via  $F_{\mathbf{b}_{\alpha,\rho}}^\inv(z) = z (1 + f(z))$. Note that $f(z)=o(1)$ as $z\to\infty, z\in -\Gamma_{0,\theta}$.  Using $F_{\mathbf{b}_{\alpha,\rho}}(z) = z (1 + (e^{\iu\rho\pi}z)^{-\alpha})$ on $\CC^-$, formula $F_{\mathbf{b}_{\alpha,\rho}}( z (1+f(z)) =z $ amounts to 
\[
- (e^{-\iu\rho \pi}z)^\alpha = \frac{f(1/z)}{(1+f(1/z))^{1-\alpha}}. 
\]
Let $\xi(w)$ be the inverse function of $z\mapsto \frac{z}{(1+z)^{1-\alpha}}$ that can be defined in a neighborhood of zero. Using the Lagrange inversion formula, we have
\[
\xi(w) =  \sum_{n\ge1} \frac{((1-\alpha)n)_{n-1} }{n!} w^n, 
\]
where $(x)_n \coloneq x (x-1) \cdots (x-n+1)~(n\ge1)$ and $(x)_0\coloneq1$. Since $f(1/z) = \xi(- (e^{-\iu\rho \pi}z)^\alpha) $, the Voiculescu transform $\varphi_{\mathbf{b}_{\alpha,\rho}}(z) =  F_{\mathbf{b}_{\alpha,\rho}}^\inv(z) -z= z C_{\mathbf{b}_{\alpha,\rho}}(1/z)$ has the series expansion
\[
\varphi_{\mathbf{b}_{\alpha,\rho}}(z) = z f(z) = z \sum_{n\ge1} \frac{((1-\alpha)n)_{n-1} }{n!} (-1)^n (e^{\iu\rho \pi}z)^{-n\alpha}, \qquad z\in -\Gamma_{0,\theta}
\]
for sufficiently large $|z|$. (An explicit convergence radius can be calculated, but it is not needed for the arguments below.)  

Next, we compute the series expansion of the Cauchy transform of the free stable laws in a similar manner.  We set a function $g(z)$ via $G_{\mathbf{f}_{\alpha,\rho}} (z) = \frac{1}{z}(1 + g(z))$ in $\CC^-$, substitute $1/G_{\mathbf{f}_{\alpha,\rho}} (z) $ into $F^\inv_{\mathbf{f}_{\alpha,\rho}}(z) = z - z (e^{\iu\rho \pi} z)^{-\alpha}$ $(z \in \CC^-)$, and change the variable $z$ to $1/z$, to obtain the equation 
\[
- (e^{-\iu\rho \pi}z)^\alpha = \frac{g(1/z)}{(1+g(1/z))^{\alpha}}. 
\]
Then applying the previous method amounts to 
\begin{equation}\label{eq:Cauchy-FS}
G_{\mathbf{f}_{\alpha,\rho}}(z) = \frac{1}{z} + \frac1{z} \sum_{n\ge1} \frac{(\alpha n)_{n-1} }{n!} (-1)^n (e^{\iu\rho \pi}z)^{-n\alpha}, \qquad z \in \CC^- 
\end{equation}
for large $|z|$. 

From now on, we restrict the parameters to $(\alpha,\rho) \in \mathfrak{A}_0$.  Replacing the parameters in \eqref{eq:Cauchy-FS} yields
\begin{align*}
&G_{\mathbf{f}_{1-\alpha,\frac{\alpha\rho}{1-\alpha}}}(z) = \frac{1}{z} + \frac1{z} \sum_{n\ge1} \frac{((1-\alpha) n)_{n-1} }{n!} (-1)^n (e^{\iu\rho \pi}z^{\frac{1-\alpha}{\alpha}})^{-n\alpha}, \\ 
&G_{\mathbf{f}_{1-\alpha,\frac{\alpha\rho+1-2\alpha}{1-\alpha}}}(z) = \frac{1}{z} + \frac1{z} \sum_{n\ge1} \frac{((1-\alpha) n)_{n-1} }{n!} (-1)^n (-e^{\iu\rho \pi}(-z)^{\frac{1-\alpha}{\alpha}})^{-n\alpha}. 
\end{align*} 
Comparing the series expansions, we obtain 
\begin{equation}\label{eq:V-BS}
\varphi_{\mathbf{b}_{\alpha,\rho}}(z) = 
\begin{cases}
\left. z \left(w G_{\mathbf{f}_{1-\alpha,\frac{\alpha\rho}{1-\alpha}}}(w) -1\right) \right|_{w= z^{\frac{\alpha}{1-\alpha}}}, &  \Re(z) > 0, \Im(z) < 0, \\[5mm]
  \left. z \left(w G_{\mathbf{f}_{1-\alpha,\frac{\alpha\rho+1-2\alpha}{1-\alpha}}}(w) - 1\right) \right|_{w= -(-z)^{\frac{\alpha}{1-\alpha}}}, & \Re(z)<0, \Im(z)<0. 
  \end{cases}
\end{equation}

By  \cite[Corollary 3.6]{HS2017unimodality}, the measure $\mathbf{b}_{\alpha,\rho}$ is of the form $\mu_{\alpha,\rho} \multconv \MP$, where $\mu_{\alpha,\rho}$ is the free L\'evy measure of  $\mathbf{b}_{\alpha,\rho}$ (and therefore $\mu_{\alpha,\rho}(\{0\})=0$).  By the Stieltjes inversion, $\mu_{\alpha,\rho}$ is given by the limit
\[
\frac1{\pi x^2} \Im \varphi_{\mathbf{b}_{\alpha,\rho}}(x-\iu0) \di x
\]
which has density \eqref{levy density} due to \eqref{eq:V-BS}.  
 \end{proof}
 
\begin{remark}
The series expansion \eqref{eq:Cauchy-FS} and the Stieltjes inversion formula yield a series expansion of $p_{\alpha,\rho}(x)$ for large $|x|$, which was already obtained in \cite{HK2014free} by means of the Mellin transform. Notice that a Taylor series expansion of $p_{\alpha,\rho}(x)$  in a neighborhood of $0$ was also obtained in  \cite{HK2014free}. 
\end{remark} 
\begin{remark}
We used the fact that $\mathbf{b}_{\alpha,\rho}$ is freely infinitely divisible in the proof of Theorem \ref{thm:boole_LM}. In fact, this fact can also be proved from \eqref{eq:V-BS} with some further elaboration. The idea is that \eqref{eq:V-BS} gives an analytic continuation of $\varphi_{\mathbf{b}_{\alpha,\rho}}$ to $\CC^-$. The Stieltjes inversion yields $\Im \varphi_{\mathbf{b}_{\alpha,\rho}}(x-i0) \ge0$. By the maximum principle for harmonic functions, we obtain $\Im \varphi_{\mathbf{b}_{\alpha,\rho}} \ge0$ on $\CC^-$, which guarantees the free infinite divisibility by \cite[Theorem 5.10]{BV1993free} 
\end{remark} 

\begin{corollary} If $0 <\alpha \le \frac1{2}$ and $0\le \rho \le1$ then 
$\mathbf{f}_{\alpha,\rho} \multconv (\MP^{\multconv\frac{1-2\alpha}{\alpha}})$ coincides with $\mu_{\alpha,\rho}$. 
\end{corollary} 
 \begin{proof}
    Proposition \ref{cor:FL} for $\nu=\delta_1$ and Theorem \ref{thm:boole_LM} give two formulas for the free L\'evy measure of $\mathbf{b}_{\alpha,\rho}$ and so they must coincide. 
 \end{proof}

With the help of $\mu_{\alpha,\rho}$,  Proposition \ref{cor:FL}
can be generalized as follows. 
\begin{proposition}
 Let $(\alpha,\rho)\in\mathfrak{A}_0$ and $\nu \in \Prob(\RR_{\ge0})\setminus\{\delta_0\}$. Then $\mathbf{b}_{\alpha,\rho} \circledast (\nu^{\frac1{\alpha}})$ is the compound free Poisson distribution of rate $1-\nu(\{0\})$ and jump distribution $\frac{1}{1-\nu(\{0\})}(\mu_{\alpha,\rho} \multconv \nu^{\multconv \frac1{\alpha}})|_{\RR\setminus\{0\}}$. 
\end{proposition}
 \begin{proof}
     This follows by Theorem \ref{thm:boole_LM} and calculations similar to Proposition \ref{cor:FL}. 
 \end{proof}

\section{Regularity of free  convolution: proof of Theorem  \ref{thm:leb_decomp}}
\label{sec:cont_ext}

We now focus on the Lebesgue decomposition of the free multiplicative convolution. The subordination functions derived in the preceding section illuminate this enigmatic mathematical operation. Indeed, once these functions are obtained, the task of establishing regularity properties in terms of the Lebesgue decomposition reduces to a purely analytic problem.  

In this section, we show that the subordination function can be continuously extended to the boundary, specifically to the real line excluding the origin. Furthermore, we demonstrate that it can be analytically continued in a neighborhood of any point where the imaginary part of the subordination function is positive. To reduce the emergence of unnecessary reciprocals in the subsequent arguments and simplify expressions, we make extensive use of the $M$-transform introduced in \cite{J2021regularity}. This approach does not fundamentally alter the underlying arguments but rather clarifies their connection to the Cauchy transform, thereby improving the overall exposition.

We define the $M$-transform of the nondegenerate probability measures $\mu\in\Prob(\RR)$ and $\nu\in\Prob(\RR_{\ge0})$, and their corresponding subordination functions as 
\begin{align*}
   M_{\mu}(z)&\coloneqq 1/\eta_\mu(1/z),\quad M_{\nu}(z)\coloneqq 1/\eta_\nu(1/z),\\
  \Omega_1(z)&\coloneqq 1/\omega_1(1/z),\quad\Omega_2(z)\coloneqq 1/\omega_2(1/z).
\end{align*}
In this section, we define $\eta_\mu$ on $\CC^-$ by the reflection $\eta_\mu(z)\coloneq\conj{\eta_{\mu}(\conj{z})}$, and correspondingly consider the functions $M_\mu$ and $\Omega_1,\Omega_2$ on $\Hup$. 
  Theorem \ref{thm:subordination}, Propositions  \ref{prop:prop_of_eta} and \ref{prop:prop_of_eta0} can be translated into the language of $M_\mu$, $\Omega_1$ and $\Omega_2$; for example we have for $z\in\Hup$ 
 \begin{enumerate}
     \item $\arg z\leq M_\mu(z)\leq\arg z+\pi$, i.e., $M_\mu(z)/z \in \Hup$, 
     \item $\arg z\leq M_\nu(z)<\pi$,
     \item $\Omega_1\colon\Hup\to\Hup$
     \item $\Omega_2\colon\Hup\to\CC\setminus[0,+\infty)$,
     \item $\arg z\leq\Omega_2(z)\leq\arg z +\pi$, i.e., $\Omega_2(z)/z \in \Hup$, 
     \item $M_\mu(\Omega_1(z))=M_\nu(\Omega_2(z))=M_{\mu\multconv\nu}(z)$,
      \item $\Omega_1(z)\Omega_2(z)=zM_{\mu\multconv\nu}(z)$, 
 \end{enumerate}
  and 
   \begin{align}
     \nontanglim_{z\to 0} M_\mu(z) &=\frac{\mu(\{0\})}{\mu(\{0\}) - 1},\quad\nontanglim_{z\to\infty} M_\mu(z)  =\infty.\label{eq:M_infty}
 \end{align}
  Additionally, defining
 \begin{align*}
      &H_\mu(z)\coloneqq\frac{M_\mu(z)}{z}=\frac{1}{h_\mu(1/z)}, \qquad H_\nu(z)\coloneqq\frac{M_\nu(z)}{z}=\frac{1}{h_\nu(1/z)},\\
      &F_z(w)\coloneqq zH_\nu(zH_\mu(w)),  
     \!\qquad\qquad G_z(w)\coloneqq H_\mu(zH_\nu(zw)),
 \end{align*}
  we can conclude that $\Omega_1(z)$ and $\Omega_2(z)/z$ are the DW-points of $F_z(w)$ and $G_z(w)$, respectively. Observe that by Proposition \ref{prop:not_conf} both $F_z$ and $G_z$ are analytic self-maps of $\Hup$ for any $z\in(\Hup\cup\RR)\setminus\{0\}$ because we are assuming $\mu,\nu$ are nondegenerate.
  
  When considering the DW-point of a function, we need to divide the arguments into cases based on whether it is a conformal automorphism or not. The functions considered here are $F_z(w)$ and $G_z(w)$, and therefore we need a claim concerning $z$ that makes these functions conformal automorphisms. However, if such a complex number $z\in(\Hup\cup\RR)\setminus\{0\}$ exists, it is not only restricted to $z\in\RR$ (by Proposition \ref{prop:not_conf}), but it can also be concluded that $\mu$ and $\nu$ are supported on two points. More precisely, we can establish the following proposition.
 \begin{proposition}
     \label{prop:conf_auto}
     \begin{enumerate}
         \item\label{item:conf_auto1} Let $\mu\in\Prob(\RR)$ and $\nu\in\Prob(\RR_{\ge0})$ be nondegenerate probability measures. Suppose that there exists an $x\in\RR\setminus\{0\}$ such that $F_x(w)=xH_\nu(xH_\mu(w))$ is a conformal automorphism. Then, $\mu$ and $\nu$ are probability measures consisting of two atoms, i.e., there exist some real numbers $\lambda_\mu,\lambda_\nu\in(0,1)$ and $a_\mu,a_\nu,b_\mu,b_\nu\in\RR$ satisfying
         \begin{align*}
             \mu=\lambda_\mu\delta_{a_\mu}+(1-\lambda_\mu)\delta_{b_\mu} 
             \qquad\text{and}\qquad\nu=\lambda_\nu\delta_{a_\nu}+(1-\lambda_\nu)\delta_{b_\nu}.
         \end{align*}
          
          Similarly, if there exists an $x\in\RR\setminus\{0\}$ such that $G_x(w)=H_\mu(xH_\nu(xw))$ is a conformal automorphism, then $\mu$ and $\nu$ are probability measures consisting of two atoms.
         \item\label{item:conf_auto2} Under the conditions of \ref{item:conf_auto1}, if $\Omega_1(z)\in\Hup$ or $\Omega_2(z)\in\CC\setminus\RR$ for $z\in\RR\setminus\{0\}$, then $\Omega_1$, $\Omega_2$, and $M_{\mu\multconv\nu}$ can be analytically continued in a neighborhood of $z$.
     \end{enumerate}
 \end{proposition}
 \begin{proof}
 \ref{item:conf_auto1}\,
  The proof is almost identical to that in \cite{B2008lebesgue}, but we will provide it here for completeness. Suppose that $w\mapsto xH_\nu(xH_\mu(w))$ is a conformal automorphism of $\Hup$. 
  By Lemma \ref{lem:auto}, both  $H_\mu$ and $H_\nu$ are conformal automorphisms of $\Hup$. By \eqref{eq:M_infty}, we have $\lim_{y\to +\infty}\iu yH_\mu(\iu y) =\lim_{y\to +\infty}\iu y H_\mu(\iu y) =\infty$ and so, in view of \eqref{eq:auto}, there exist real numbers $b_\mu, c_\mu, d_\mu, b_\nu, c_\nu, d_\nu$ such that $d_\mu - b_\mu c_\mu > 0,  d_\nu - b_\mu c_\nu > 0$ and 
      \begin{gather*}
           H_\mu(z) =\frac{z + b_\mu}{c_\mu z + d_\mu},\quad H_\nu(z) =\frac{z + b_\nu}{c_\nu z + d_\nu},\quad z\in\Hup. 
      \end{gather*}
       It also follows that $\eta_\mu(z) = 1/H_\mu(1/z)$ and $\eta_\nu(z)=1/H_\nu(1/z)$ are conformal automorphisms. Thus, the argument reduces to the case where $a = 0$ in Proposition 3.1 of \cite{B2008lebesgue} concerning $h_\mu$ and $h_\nu$. Therefore, $\mu$ and $\nu$ are both convex combinations of two atoms. The same reasoning leads to a similar conclusion for $G_z(w)$.

\vspace{2mm}
\noindent
 \ref{item:conf_auto2}\,  The fixed point equation for $\Omega_1(z)$ is 
     \begin{align*}
         \Omega_1(z)&=zH_\nu(zH_\mu(\Omega_1(z))) =z\frac{zH_\mu(\Omega_1(z))+b_\nu}{c_\nu zH_\mu(\Omega_1(z))+d_\nu}\\
          &=z\frac{z\frac{\Omega_1(z)+b_\mu}{c_\mu\Omega_1(z)+d_\mu}+b_\nu}{c_\nu z\frac{\Omega_1(z)+b_\mu}{c_\mu\Omega_1(z)+d_\mu}+d_\nu} =\frac{z^2(\Omega_1(z)+b_\mu)+b_\nu z(c_\mu\Omega_1(z)+d_\mu)}{c_\nu z(\Omega_1(z)+b_\mu)+d_\nu(c_\mu\Omega_1(z)+d_\mu)}.
     \end{align*}
      From this formula, for $z\in\Hup$, we have
     \begin{align*}
         \Omega_1(z)^2(c_\nu z+d_\nu c_\mu)-\Omega_1(z)\left(z^2+(b_\nu c_\mu-c_\nu b_\mu)z-d_\mu d_\nu\right)-z(zb_\mu+b_\nu d_\mu)=0,
     \end{align*}
      which shows that $\Omega_1(z)$ satisfies a quadratic equation in terms of $z$ and the real parameters $b_\mu, c_\mu, d_\mu, b_\nu,c_\nu, d_\nu$. 
      Hence, $\Omega_1(z)$ can be continuously extended to the real line, taking values in $\Hup\cup\RR\cup\{\infty\}$. 
      Moreover, when $\Omega_1(z)\in\Hup$ for $z\in\RR\setminus\{0\}$, it immediately follows that $\Omega_1(z)$ can be analytically continued in a neighborhood of $z$. 
      In this case, using the relations $zM_{\mu\multconv\nu}(z)=\Omega_1(z)\Omega_2(z)$ and $M_\mu(\Omega_1(z))=M_\nu(\Omega_2(z))=M_{\mu\multconv\nu}(z)$, we can also conclude that both $M_{\mu\multconv\nu}(z)$ and $\Omega_2(z)$ can be analytically continued around $z\in\RR\setminus\{0\}$. 
      
      Next, for $\Omega_2(z)$, simplifying the expression by letting $w_z\coloneqq\Omega_2(z) / z$, we obtain
     \begin{align*}
          w_z&=H_\mu(zH_\nu(zw_z)) =\frac{zH_\nu(zw_z)+b_\mu}{c_\mu zH_\nu(zw_z)+d_\mu}\\
          &=\frac{z\frac{zw_z+b_\nu}{c_\nu zw_z+d_\nu}+b_\mu}{c_\mu z\frac{zw_z+b_\nu}{c_\nu zw_z+d_\nu}+d_\mu}
          =\frac{z(zw_z+b_\nu)+b_\mu(c_\nu w_z z+d_\nu)}{c_\mu z(zw_z+b_\nu)+d_\mu(c_\nu zw_z+d_\nu)}.
     \end{align*}
      Thus, we arrive at the equation
     \begin{align*}
         \left(\frac{\Omega_2(z)}{z}\right)^2(c_\mu z^2+d_\mu c_\nu z)+(c_\mu b_\nu z+d_\mu d_\nu-z^2-b_\mu c_\nu z)\frac{\Omega_2(z)}{z} - zb_\nu-b_\mu d_\nu=0,
     \end{align*}
      and similar conclusions reached above hold for $\Omega_2(z)/z$ as well. 
 \end{proof}
  
  Using the above discussion, we can compute the free convolution of two probability measures with two atoms. Several specific examples are calculated in Appendix \ref{app:specific_example}. These computations can also be performed using the well-known $S$-transform. Since probability measures with two atoms are the simplest among nontrivial measures, the fact that their free convolution requires such elementary but complicated calculations suggests that the free convolution is intrinsically enigmatic.

  With these preparations in place, we now show that the subordination functions can be extended continuously to the real line, except at the origin. The behavior at the origin will be addressed in Appendix \ref{app:ext_at_origin}.
  
 \begin{proposition}
     \label{prop:analytic_cont}
      Let $\mu\in\Prob(\RR)$ and $\nu\in\Prob(\RR_{\geq0})$ be nondegenerate probability measures. Then, the functions $\Omega_1$ and $\Omega_2$ can be continuously extended to $(\Hup\cup\RR)\setminus\{0\}$ as functions taking values in $\Hup\cup\RR\cup\{
     \infty\}$.
      Moreover, if $\Omega_1(x)\in\Hup$ or $\Omega_2(x)\in\CC\setminus\RR$ for $x\in\RR\setminus\{0\}$, then $\Omega_1$, $\Omega_2$, and $M_{\mu\multconv\nu}$ have analytic continuation in a neighborhood of $x$.
 \end{proposition}
  
 \begin{proof}
      Observe that $F_z(w)\in\Hup$ for  $z\in(\Hup\cup\RR)\setminus\{0\}$ and $w\in\Hup$. For every $z\in\Hup$, the point $\Omega_1(z)$ was the DW-point of $F_z$. Since the map $z\mapsto F_z$ is continuous on $(\Hup\cup\RR)\setminus\{0\}$ with $w\in\Hup$ fixed, by applying Theorem \ref{thm:conv_of_dwp}, the DW-point of $F_z$ is continuous with respect to $z$ on $(\Hup\cup\RR)\setminus\{0\}$ if there is no $z\in\Hup$ such that $F_z(w)\equiv w$, and hence $\Omega_1(z)$ can be continuously extended. If there exists a point $z\in\RR\setminus\{0\}$ such that $F_z(w)\equiv w$, the claim follows directly from Proposition \ref{prop:conf_auto}. The same reasoning applies to $\Omega_2(z)/z$ and $G_z$, leading to the result for $\Omega_2(z)$.
      
      Next, consider the case where $\Omega_1(x)\in\Hup$ or $\Omega_2(x)\in\CC\setminus\RR$ for some $x\in\RR$. Recall that $\Omega_1(z)$ and $\Omega_2(z)/z$ are the DW-points of the functions $F_z$ and $G_z$, respectively. Suppose that $\Omega_1(x)\in\Hup$ and $F_x$ is not a conformal automorphism. Then, by Theorem \ref{thm:denjoy_wolff}, we have $|F'_x(\Omega_1(x))|<1$. By the analytic implicit function theorem \cite[p.~34]{FG02}, there exists an analytic map $\Omega_x(z)$ from a neighborhood $\mathsf{B}_x$ of $x$ to a neighborhood $\mathsf{B}_{\Omega_1(x)}$ of $\Omega_1(x)$ such that $F_z(\Omega_x(z))=\Omega_x(z)$. By uniqueness of the DW-point stated in Theorem \ref{thm:dwfix}, it follows that $\Omega_1(z)=\Omega_x(z)$ for $z\in\mathsf{B}_x\cap\Hup$, and hence $\Omega_1(x)$ admits analytic continuation near the point $x$. Using $M_{\mu\multconv\nu}(z)=M_\mu(\Omega_1(z))$ and $zM_{\mu\multconv\nu}(z)=\Omega_1(z)\Omega_2(z)$, we can deduce that $M_{\mu\multconv\nu}(z)$ and $\Omega_2(z)$ also admit analytic continuation. In the case where $F_x$ is a conformal automorphism, the result follows directly from Proposition \ref{prop:conf_auto}.
  
      A similar argument applies when $\Omega_2(z)\in\CC\setminus\RR$. If $G_x$ is not a conformal automorphism, then $\Omega_2(z)/z$ admits analytic continuation near the point $x$, and using this claim, $\Omega_1(z)$ and $M_{\mu\multconv\nu}(z)$ are also analytically continued near the point $x$. If $G_x$ is a conformal automorphism, the result follows from Proposition \ref{prop:conf_auto} again.
 \end{proof}

To complete the proof of Theorem \ref{thm:leb_decomp}, we utilize the following lemma, which is implicitly used in \cite{B2008lebesgue} and \cite{J2021regularity} to detect atoms in the original measures.
  
\begin{lemma}
   \label{lem:singJC}
    Let $\mu\in\Prob(\RR)$ and $\nu\in\Prob(\RR_{\geq0})$ be nondegenerate probability measures. Suppose that $\nontanglim_{z\to x} M_{\mu\multconv\nu}(z) = 1$ for $x\in\RR\setminus\{0\}$. Then, $\Omega_1(x)\in\RR\setminus\{0\}$, $\Omega_2(x)\in\RR\setminus\{0\}$ and $x =\Omega_1(x)\Omega_2(x)$ hold. Moreover, when $x > 0$, we have $\Omega_1(x)>0$, $\Omega_2(x)>0$ and 
       \begin{align}
            0&<\frac{1}{\nu(\{\Omega_2(x)\})}-1\leq\liminf_{y\to0^+}\frac{\arg\Omega_1(x+\iu y)}{\arg\Omega_2(x+\iu y)}<+\infty,  \label{eq:muomegaa} \\
            0&<\frac{1}{\mu(\{\Omega_1(x)\})}-1\leq\liminf_{y\to0^+}\frac{\arg\Omega_2(x+\iu y)}{\arg\Omega_1(x+\iu y)}<+\infty, \label{eq:nuomegaa}
       \end{align}
        and when $x < 0$, we have $\Omega_1(x)<0$, $\Omega_2(x)>0$ and 
       \begin{align}
            0&<\frac{1}{\nu(\{\Omega_2(x)\})}-1\leq\liminf_{y\to 0^+}\frac{\pi-\arg\Omega_1(x+\iu y)}{2\pi-\arg\Omega_2(x+\iu y)}<+\infty, \label{eq:muomegab}  \\
            0&<\frac{1}{\mu(\{\Omega_1(x)\})}-1\leq\liminf_{y\to 0^+}\frac{2\pi-\arg\Omega_2(x+\iu y)}{\pi-\arg\Omega_1(x+\iu y)}<+\infty. \label{eq:nuomegab}
       \end{align}
        In particular, it follows that $\mu(\{\Omega_1(x)\})>0$, $\nu(\{\Omega_2(x)\})>0$, and
       \begin{align*}
           \left(\frac{1}{\mu(\{\Omega_1(x)\})}-1\right)\left(\frac{1}{\nu(\{\Omega_2(x)\})}-1\right)\leq1.
       \end{align*}
        Consequently, $\mu(\{\Omega_1(x)\})+\nu(\{\Omega_2(x)\})\geq1$.
\end{lemma}

 \begin{proof}
        We first show that $\Omega_1(x)\in\RR\setminus\{0\}$ and $\Omega_2(x)\in\RR\setminus\{0\}$. The subordination relation $M_{\mu\multconv\nu}(z)=M_\mu(\Omega_1(z))=M_\nu(\Omega_2(z))$ and $\nontanglim_{z\to x}M_{\mu\multconv\nu}(z)=1$ show that $\Omega_1(x)\in\RR\cup\{\infty\}$ and $\Omega_2(x)\in\RR\cup\{\infty\}$.  The fact that $\nontanglim_{z\to\infty}M_\nu(z)=\nontanglim_{z\to\infty}M_\mu(z)=\infty$, together with Theorem \ref{thm:Lind}, implies that $\Omega_1(x)\neq\infty$ and $\Omega_2(x)\neq\infty$. Consequently, $\Omega_1(x)\neq0$ and $\Omega_2(x)\neq0$ follow from the relation $zM_{\mu\multconv\nu}(z)=\Omega_1(z)\Omega_2(z)$, leading to the claim.


         We now proceed to the second part using an argument similar to that in \cite{J2021regularity}.  Let us consider the case $x>0$. Note that $\Omega_1(x)<0,\Omega_2(x)<0$ or $\Omega_1(x)>0,\Omega_2(x)>0$ follow from $\Omega_1(x)\Omega_2(x)=x>0$. The case where $\Omega_1(x)<0$ and $\Omega_2(x)<0$ leads to a contradiction since $1=M_{\mu\multconv\nu}(x)=M_\nu(\Omega_2(x))<0$. Thus, we must have $\Omega_1(x)>0$ and $\Omega_2(x)>0$. 
          Observe then that for any nondegenerate probability measure $\lambda$ on $\RR$ and $a>0$ with $\nontanglim_{z\to a}M_{\lambda}(z)=1$,
       \begin{align}
           \frac{1}{\lambda(\{a\})}&=a\lim_{z\nontangtou a}\frac{M_{\lambda}(z)-1}{z-a}=a\lim_{z\nontangtou a}\frac{(\sqrt{M_{\lambda}(z)}-1)(\sqrt{M_{\lambda}(z)}+1)}{z-a} \notag \\
           &=2a\lim_{z\nontangtou a}\frac{\sqrt{M_{\lambda}(z)}-1}{z-a}
           = 2a\lim_{z\nontangtou a}\frac{\log \sqrt{M_{\lambda}(z)}}{z-a}
           = 2a\liminf_{z\to a}\frac{\Im\log \sqrt{M_{\lambda}(z)}}{\Im z} \notag \\ 
           &=a \liminf_{z\to a}\frac{\arg M_\lambda(z)}{|z|\sin \arg z} = \liminf_{z\to a}\frac{\arg M_\lambda(z)}{\arg z}. 
           \label{eq:atom}
       \end{align}
       Here, the branch of $\sqrt{M_\lambda(z)}$ is chosen so that the argument is halved,  Theorem \ref{thm:JC} is applied on the fifth equality, and recall that $\arg z=0$ for $z\in (0,+\infty)$, and it increases counterclockwise in $[0, 2\pi)$. Since $\nontanglim_{z\to \Omega_1(x)}M_\mu(z)=1$ by Lindel\"of's theorem, we may use \eqref{eq:atom} to obtain
       \begin{align}
           \frac{1}{\mu(\{\Omega_1(x)\})}-1
            &=\liminf_{z\to\Omega_1(x)}\frac{\arg M_\mu(z)-\arg z}{\arg z}
           \leq\liminf_{y\to0^+}\frac{\arg M_\mu(\Omega_1(x+\iu y))-\arg\Omega_1(x+\iu y)}{\arg\Omega_1(x+\iu y)} \notag \\
            &=\liminf_{y\to0^+}\frac{\arg\Omega_2(x+\iu y)-\arg(x+\iu y)}{\arg\Omega_1(x+\iu y)}\leq\liminf_{y\to0^+}\frac{\arg\Omega_2(x+\iu y)}{\arg\Omega_1(x+\iu y)}. \label{eq:muomega1}
       \end{align}
        By a similar reasoning, we deduce that
       \begin{align}
           \frac{1}{\nu(\{\Omega_2(x)\})}-1\leq\liminf_{y\to0^+}\frac{\arg\Omega_1(x+\iu y)}{\arg\Omega_2(x+\iu y)}.\label{eq:nuomega2}
       \end{align}
        Since $\mu$ and $\nu$ are nondegenerate measures, the left-most sides of \eqref{eq:muomega1} and \eqref{eq:nuomega2} are positive. This also implies that the two liminf's on the right-most sides are finite, and thus the desired \eqref{eq:muomegaa} and \eqref{eq:nuomegaa} follow.    
    
        Finally, we consider the case $x<0$. This case can be reduced to the case $x>0$ by the flip transform $\mu\mapsto \widehat{\mu}$ defined by $\widehat{\mu}(B):= \mu(-B)$ for Borel subsets $B \subseteq\RR$. Let $\widehat \Omega_1, \widehat\Omega_2$ be the subrodination functions for $\widehat{\mu}\boxtimes\nu$. Since $\widehat{\mu\boxtimes\nu}=\widehat{\mu}\boxtimes\nu$, we have $M_{\widehat\mu}(z)=\conj{M_\mu(-\conj{z})}$, $M_{\widehat\mu \boxtimes \nu }(z)=\conj{M_{\mu\boxtimes\nu}(-\conj{z})}$, 
        $\widehat \Omega_1(z)=-\conj{\Omega_1(-\conj{z})}$, and $\widehat{\Omega}_2(z)=\conj{\Omega_2(-\conj{z})}$ for $z\in \Hup$. From these relations we see that $\arg \widehat \Omega_1(-x+\iu y)= \pi- \arg \Omega_1(x+\iu y)$ and $\arg \widehat\Omega_2(-x+\iu y) =2\pi - \arg \Omega_2(x+\iu y)$ for $y>0$. With this observation,    \eqref{eq:muomegaa} and \eqref{eq:nuomegaa} obviously imply  \eqref{eq:muomegab}  and \eqref{eq:nuomegab}, respectively. 
         \end{proof}

 \begin{remark}
     \label{rem:argineq}
      From this proof, under the assumptions of Lemma \ref{lem:singJC}, we see that for $x>0$, we have $\arg\Omega_1(x+\iu y)>\arg(x+\iu y)$ as $y\to0^+$, and for $x<0$, we have $\arg\Omega_1(x+\iu y)<\arg(x+\iu y)$ as $y\to0^+$.
 \end{remark}
  
 \begin{proof}[\textbf{Proof of Theorem \ref{thm:leb_decomp}}]
     \ref{item:leb_decomp1}\,  
    This is already proved in Proposition  \ref{prop:analytic_cont}. 

\vspace{2mm}
\noindent
 \ref{item:leb_decomp2}\,  
          As mentioned earlier, this statement has already been established in \cite[Corollary 6.6]{ACSY2024universality}. It can also be shown similarly to the discussion in the proof of Theorem 4.1 (a) in \cite{B2003Atoms}, together with Remark \ref{rem:argineq}.

 \vspace{2mm}
\noindent
\ref{item:leb_decomp3}\,  
    This relation follows immediately from the properties of the $T$-transform; see Remark \ref{rem:regularity}.

\vspace{2mm}
\noindent
 \ref{item:leb_decomp4}\,  
              Assume to the contrary that the singular continuous part $(\mu\multconv\nu)^\scp$ is nonzero. Then, there exists an uncountable set $E\subseteq\RR\setminus\{0\}$ such that $(\mu\multconv\nu)^\scp(E)>0$, $(\mu\multconv\nu)^\acp(E)=(\mu\multconv\nu)^\ppp(E)=0$, and for any $x\in E$, $\nontanglim_{z\to x}M_{\mu\multconv\nu}(z)=1$. 
              By the relation $zM_{\mu\multconv\nu}(z)=\Omega_1(z)\Omega_2(z)$ and Lemma \ref{lem:singJC}, it follows that $\{\Omega_1(x)\colon x\in E\}$ or $\{\Omega_2(x)\colon x\in E\}$ is uncountable and these sets are contained in the pure point part of $\mu$ and $\nu$, respectively. 
              The cardinality of atoms of a probability measure is necessarily countable, leading to a contradiction. Therefore, we conclude that $(\mu\multconv\nu)^\scp=0$.

\vspace{2mm}
\noindent
 \ref{item:leb_decomp5}\, 
              When $\lim_{y\to0^+}G_{\mu\multconv\nu}(x+\iu y)\in\CC^-$ for $x\in\RR\setminus\{0\}$, the relation $xM_{\mu\multconv\nu}(x)=\Omega_1(x)\Omega_2(x)$ shows that either $\Omega_1(x)\in\Hup$ or $\Omega_2(x)\in\CC\setminus\RR$ holds. 
              In either case, $M_{\mu\multconv\nu}(x)$ is analytically continued in a neighborhood of $x$ by Proposition \ref{prop:analytic_cont}, and thus $G_{\mu\multconv\nu}(z)=M_{\mu\multconv\nu}(z)/(z(M_{\mu\multconv\nu}(z)-1))$ can also be analytically continued in the neighborhood of $x$. 
              Therefore, $\frac{\di{(\mu\multconv\nu)^{\acp}}}{\di{x}}$ has the  analytic version $-\frac1{\pi}\Im[G_{\mu\multconv\nu}]$ near the point $x$.
 \end{proof}


\section*{Acknowledgments}

The authors are grateful to Ricardo Guerrero and Daniel Perales for discussions in the early stage of work, and to the anonymous referees for helpful suggestions to improve the readability of the paper. This work was supported by the Japan Society for the Promotion of
Science (JSPS) through the Bilateral Joint Research Program (Open
Partnership Joint Research Projects) 120249936. O. A. thanks the support from CONACYT Grant CB-2017-2018-A1-S-9764. Part of this work was done during T. H.'s stays in CIMAT. Y. K. acknowledges the support of the Research Institute for Mathematical Sciences, an International Joint Usage/Research Center located in Kyoto University.

\appendix

\section{Continuity and boundedness of the density}
\label{app:ext_at_origin}

The following concepts and theorems are utilized in the discussion of cluster sets of analytic functions. For further details, see  \cite{CL2004theory}. These results will be employed in the proof of Theorem  \ref{thm:extorigin}.

\begin{definition}
    Given a function $f\colon\Hup\to\Hup$ and $x\in\RR\cup\{\infty\}$, we define the cluster set of $f$ at $x$ as 
    \begin{align*}
        C(f,x)\coloneqq \left\{z\in\CC\cup\{\infty\}\colon \exists{\{z_n\}_{n=1}^\infty}\subseteq \Hup,\ \lim_{n\to\infty}z_n=x,\ \lim_{n\to\infty}f(z_n)=z\right\}.
    \end{align*}
\end{definition}

\begin{proposition}
    \label{prop:conn}
    For a continuous function $f\colon\Hup\to\Hup$ and $x\in\RR\cup\{\infty\}$, $C(f,x)$ is connected.
\end{proposition}

\begin{proposition}[{\cite[Lemma 2.18]{B2008lebesgue}}]
    \label{prop:good_seq}
    Let $f$ be an analytic self-map of $\Hup$ and $x\in\RR\cup\{\infty\}$. Suppose that $C(f,x)\subseteq\RR\cup\{\infty\}$ and $C(f,x)$ contains at least two points. Then, there exists an open interval $I$ such that for any $c\in I$, there exists a sequence $\{z_n^{(c)}\}\subseteq\Hup$ such that $\lim_{n\to\infty}{z_n^{(c)}}=x, f(z_n^{(c)})=c+\iu y_n$ and $y_n\to 0$ as $n\to\infty$. 
\end{proposition}

\begin{lemma}\label{lem:infty}
     Let $\mu \in \Prob(\RR)$ and $\nu \in \Prob(\RR_{\geq0})$ be nondegenerate probability measures. Suppose that there exists a sequence $\{z_n\}_{n=1}^\infty\subseteq\Hup$ such that $M_{\mu\multconv\nu}(z_n)\to1$ and $z_n\to x$ for $x \in \RR \setminus \{0\}$ as $n\to\infty$, and we have
        \begin{align*}
            \lim_{\substack{z\to\infty, z\in\Hup}}M_\mu(z)=\lim_{\substack{z\to\infty, z\in\Hup}}M_\nu(z)=\infty.
        \end{align*}
        Then, the results of Lemma  \ref{lem:singJC} hold.
\end{lemma}

\begin{proof}
    By a similar argument to that in Lemma  \ref{lem:singJC}, the assumption $\lim_{\substack{z\to\infty}}M_\mu(z)=\lim_{\substack{z \to \infty}}M_\nu(z)=\infty$ implies that $\Omega_1(x)\in\RR\setminus\{0\}$ and $\Omega_2(x)\in\RR\setminus\{0\}$, and ultimately leads to the desired conclusion.
\end{proof}

In the multiplicative setting, various functions employed in the preceding arguments, especially the subordination functions, necessitate additional assumptions to preserve their regularity near the origin. Theorem  \ref{thm:extorigin} addresses the continuous extension of these subordination functions and provides a result analogous to  \cite[Corollary 8]{B2014infty} and  \cite[Theorem 3.2]{J2021regularity}.

\begin{theorem}
    \label{thm:extorigin}
    Let $\mu\in\Prob(\RR)$ and $\nu\in\Prob(\RR_{\ge0})$ be nondegenerate probability measures and let $M_\mu$ and $M_\nu$ be their $M$-transforms. 
    \begin{enumerate}
        \item\label{item:extorigin1}  Suppose that 
        \begin{align}
            \lim_{{z\to\infty, z\in\Hup}}M_\mu(z)=\lim_{{z\to\infty, z\in\Hup}}M_\nu(z)=\infty, \label{ass:m_inf_cond}
        \end{align}
        holds. Then, 
        \begin{enumerate}[label=\rm(\alph*)]
            \item\label{item:extorigin1a}  $\lim_{z\to\infty}\Omega_1(z)=\lim_{z\to\infty}\Omega_2(z)=\infty$, 
            \item\label{item:extorigin1b} $M_{\mu\multconv\nu}$ extends to a   continuous map from $(\Hup\cup\RR)\setminus\{0\}$ into $  \Hup\cup\RR\cup\{\infty\}$, and
            \item\label{item:extorigin1c} additionally, if $\mu(\{a\})+\nu(\{b\})<1$ for any $a,b\in\RR\setminus\{0\}$, then $x\cdot\frac{\di{(\mu\multconv\nu)}}{\di{x}}$ has a bounded and continuous version on $\RR\setminus(-\ee,\ee)$ for any $\ee>0$.
        \end{enumerate}
        \item\label{item:extorigin2} In addition to \eqref{ass:m_inf_cond}, suppose that 
        \begin{align*}
            \lim_{{z\to0, z\in\Hup}}M_\mu(z)=\frac{\mu(\{0\})}{\mu(\{0\})-1}, \quad \lim_{{z\to0, z\in\Hup}}M_\nu(z)=\frac{\nu(\{0\})}{\nu(\{0\})-1}.
        \end{align*}
        Then, 
        \begin{enumerate}[label=\rm(\alph*)]
            \item\label{item:extorigin2a} $\Omega_1, \Omega_2$ and $M_{\mu\multconv\nu}$ extend to  continuous maps from $\Hup\cup\RR$ into $\Hup\cup \RR \cup\{\infty\}$, and
            \item\label{item:extorigin2c} additionally, if $\mu(\{a\})+\nu(\{b\})<1$ for any $a,b\in\RR$, then $x\cdot\frac{\di{(\mu\multconv\nu)}}{\di{x}}$ has a bounded and continuous version on $\RR$.
        \end{enumerate}
    \end{enumerate}
\end{theorem}

Note that the condition  \eqref{ass:m_inf_cond} holds when $\mu$ and $\nu$ have compact support.

\begin{proof}
    \ref{item:extorigin1}\ref{item:extorigin1a}
         By Theorem  \ref{thm:conv_of_dwp}, together with the condition \eqref{ass:m_inf_cond}, we have $\Omega_1(z)\to\infty$ as $z\to\infty$, and hence $M_{\mu\multconv\nu}(z)=M_\mu(\Omega_1(z))\to\infty$ as $z\to\infty$. 
         Since $\nontanglim_{z\to\infty}\Omega_2(z)=\infty$ by Theorem  \ref{thm:subordination}, it is sufficient to show that $|C(\Omega_2,\infty)|=1$. 
         Suppose to the contrary that  $|C(\Omega_2,\infty)|\ge2$ holds. By the subordination relation $M_\nu(\Omega_2(z))=M_{\mu\multconv\nu}(z)$, we deduce that $C(\Omega_2,\infty)\subseteq\RR_{\geq0}\cup\{\infty\}$. Furthermore, in the light of Proposition  \ref{prop:good_seq}, we have some interval $I\subseteq\RR$ such that for all $x\in I$ at which the nontangential limit of $M_\nu$ exists,
            \begin{align*}
                \nontanglim_{z\to x}M_\nu(z)=\lim_{n\to\infty}M_{\nu}(\Omega_2(z_n^{(x)}))=\lim_{n\to\infty}M_{\mu\multconv\nu}(z_n^{(x)})=\infty.
            \end{align*}
            By Theorems  
            \ref{thm:fatou} and \ref{thm:priv}, it follows that $M_\nu$ is a constant function taking the value $\infty$, which is a contradiction.

    \vspace{2mm}
    \noindent    \ref{item:extorigin1}\ref{item:extorigin1b}
   Assume that there exists $c\in\RR\setminus\{0\}$ such that $\Omega_1(c)=0, \Omega_2(c)=\infty$. Using Theorem  \ref{thm:Lind} and Proposition  \ref{prop:analytic_cont}, we deduce that
            \begin{align*}
                \infty =\lim_{z\to c}M_\nu(\Omega_2(z)) = \lim_{z\to c}M_\mu(\Omega_1(z)) =  \nontanglim_{z\to c}M_\mu(\Omega_1(z))= 
                \frac{\mu(\{0\})}{\mu(\{0\})-1},
            \end{align*}
            which is a contradiction. The case where $\Omega_2(c)=0$ and $\Omega_1(c)=\infty$ results in a contradiction as well. Therefore, by virtue of the relation $M_{\mu\multconv \nu}(z)=\Omega_1(z)\Omega_2(z)/z$, the function $M_{\mu\multconv \nu}(z)$ extends continuously to $(\Hup\cup\RR)\setminus\{0\}$. 

\vspace{2mm}
    \noindent        
 \ref{item:extorigin1}\ref{item:extorigin1c}
            To show that $x\frac{\di{(\mu\multconv\nu)}}{\di{x}}$ is bounded on $\RR\setminus(-\ee,\ee)$, it suffices to demonstrate that there exists a positive real number $\delta$ such that $|M_{\mu\multconv\nu}(z)-1|>\delta$ holds uniformly for all $z\in\Hup\setminus B_\ee(0)$. 
            By applying Lemma  \ref{lem:infty}, for any sequence $\{z_n\}_{n=1}^\infty\subseteq\Hup$ such that $M_{\mu\multconv\nu}(z_n)\to1$ and $z_n\to x$ for $x \in \RR \setminus \{0\}$ as $n\to\infty$, we have $\Omega_1(x)\in\RR\setminus\{0\}$ and $\Omega_2(x)\in\RR\setminus\{0\}$ satisfying $\mu(\{\Omega_1(x)\})+\nu(\{\Omega_2(x)\})\geq1$, which again leads to a contradiction. 
            Furthermore, as $z\to\infty$, we have $M_{\mu\multconv\nu}(z)=M_\mu(\Omega_1(z))\to\infty$ by assumption. Thus, from the relation $zG_{\mu\multconv\nu}(z)=M_{\mu\multconv\nu}(z)/(M_{\mu\multconv\nu}(z)-1)$, it follows that $x\frac{\di{(\mu\multconv\nu)}}{\di{x}}$ is bounded. The continuity follows from that of $M_{\mu\multconv\nu}(z)$.

\vspace{2mm}  \noindent        
 \ref{item:extorigin2}\ref{item:extorigin2a} For $\Omega_1, \Omega_2$, in view of Proposition  \ref{prop:analytic_cont}, it is sufficient to show that $\Omega_1$ and $\Omega_2$ can be extended continuously at the origin. The continuity of $\Omega_1(z)$ at the origin follows directly from Theorem  \ref{thm:conv_of_dwp} since the function $z\mapsto F_z(w)=zH_\nu(zH_\mu(w))$ is continuous on the real line. A reasoning similar to   \ref{item:extorigin1}\ref{item:extorigin1a} shows the continuity of $\Omega_2$ at the origin.

 For $M_{\mu\multconv \nu}$, since  $(\mu\multconv\nu)(\{0\})<1$ we have 
            \begin{align*}
                \nontanglim_{z\to0}M_{\mu\multconv\nu}(z)=\frac{(\mu\multconv\nu)(\{0\})}{(\mu\multconv\nu)(\{0\})-1}\neq \infty,
            \end{align*}
            which implies that $\Omega_1(0)=0$ or $\Omega_2(0)=0$. In either case, we obtain
            \begin{align}
                \lim_{z\to 0}M_{\mu\multconv\nu}(z)=\lim_{z\to 0}M_{\mu}(\Omega_1(z))=\frac{\mu(\{0\})}{\mu(\{0\})-1}, \label{eq:zero1}\\
                \lim_{z\to 0}M_{\mu\multconv\nu}(z)=\lim_{z\to 0}M_{\nu}(\Omega_2(z))=\frac{\nu(\{0\})}{\nu(\{0\})-1},\label{eq:zero2}
            \end{align}
            and find that $M_{\mu\multconv\nu}$ can be continuously extended at the origin.

    \vspace{2mm}
   \noindent   \ref{item:extorigin2}\ref{item:extorigin2c}  
      First, from the calculations \eqref{eq:zero1} and \eqref{eq:zero2}, we have $\lim_{z\to 0}M_{\mu\multconv\nu}(z)\neq 1$. Thus, from the relation $zG_{\mu\multconv\nu}(z)=M_{\mu\multconv\nu}(z)/(M_{\mu\multconv\nu}(z)-1)$, together with  \ref{item:extorigin1}\ref{item:extorigin1c} and  \ref{item:extorigin2}\ref{item:extorigin2a}, it follows that $x\frac{\di{(\mu\multconv\nu)}}{\di{x}}$ is bounded and continuous on $\RR$. 
\end{proof}

\section{Examples: free multiplicative convolution of two-point measures}\label{app:specific_example}
Using the argument presented in the proof of Proposition  \ref{prop:analytic_cont}, the free convolution of probability measures with two atoms can be explicitly calculated. As noted in the subsequent remark to Proposition  \ref{prop:analytic_cont}, these results can also be derived using the  $S$-transform, as shown in  \cite{V1987multiplication}. Here, we employ the subordination functions to compute the Cauchy transform of the free convolution. Similar calculations can be found in \cite{B2005complex}, where the derivation is carried out in the same spirit.

We consider two probability measures given by
\begin{align*}
    \mu&\coloneqq \lambda_\mu\delta_{1}+(1-\lambda_\mu)\delta_{\alpha_\mu},\\
    \nu&\coloneqq \lambda_\nu\delta_{1}+(1-\lambda_\nu)\delta_{\alpha_\nu},
\end{align*}
where $\lambda_\mu,\lambda_\nu\in(0,1)$, $\alpha_\mu\in\RR$, and $\alpha_\nu\in\RR_{\geq0}$.
For simplicity, we assume that one of the atoms is located at $1$, which does not result in any loss of generality, as this condition can always be achieved by a simple scalar multiplication. In this case, the corresponding analytic functions can be readily computed as follows:
\begin{align*}
    H_\mu(z)&=\frac{z-((1-\lambda_\mu)+\lambda_\mu\alpha_\mu)}{z(\lambda_\mu+(1-\lambda_\mu)\alpha_\mu)-\alpha_\mu}\eqqcolon\frac{z+b_\mu}{c_\mu z+d_\mu},\\
    H_\nu(z)&=\frac{z-((1-\lambda_\nu)+\lambda_\nu\alpha_\nu)}{z(\lambda_\nu+(1-\lambda_\nu)\alpha_\nu)-\alpha_\nu}\eqqcolon\frac{z+b_\nu}{c_\nu z+d_\nu},
\end{align*}
\begin{align*}
    \Omega_1(z)&=\frac{z^2+(b_\nu c_\mu-c_\nu b_\mu)z-d_\mu d_\nu+\sqrt{P(z)}}{2(c_\nu z + d_\nu c_\mu)},\\
    \frac{\Omega_2(z)}{z}&=\frac{z^2+b_\mu c_\nu z-c_\mu b_\nu z-d_\mu d_\nu+\sqrt{P(z)}}{2(c_\mu z^2+d_\mu c_\nu z)}, 
\end{align*}
where
\begin{align*}
    P(z)&=(z^2+(b_\nu c_\mu-c_\nu b_\mu)z-d_\mu d_\nu)^2+4z(c_\nu z+d_\nu c_\mu)(zb_\mu + b_\nu d_\mu)\\
    &=(z^2+(b_\mu c_\nu -c_\mu b_\nu) z-d_\mu d_\nu)^2+4z(c_\mu z+d_\mu c_\nu)(zb_\nu+b_\mu d_\nu).
\end{align*}
Observe that the polynomials inside the square root appearing in $\Omega_1(z)$ and $\Omega_2(z)/z$ are identical.

\begin{figure}[htbp]
    \centering
    \includegraphics[width=9cm]{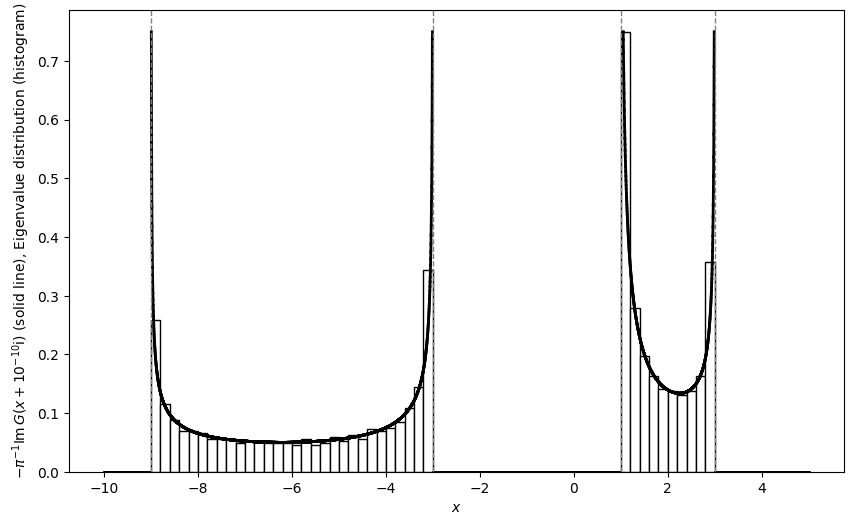}
    \caption{The density function of $(\frac{1}{2}\delta_1+\frac{1}{2}\delta_{-3})\multconv(\frac{1}{2}\delta_1+\frac{1}{2}\delta_3)$ and the scaled eigenvalue distribution of $BU^*AUB$, where $U$ is a realization of the Haar unitary random matrix, $A$ is a diagonal matrix with half of its eigenvalues equal to $1$ and the other half equal to $-3$, and $B$ is a diagonal matrix with half of its eigenvalues equal to $1$ and the other half equal to $\sqrt{3}$. The size of all matrices is $1500$. The black line represents $-\pi^{-1}\Im G_{\mu\multconv\nu}(x+10^{-10}\iu)$ for $x$. To improve clarity, we excluded data from a small neighborhood around the singularities.}
    \label{fig:two_atoms1}
\vspace{10mm}
    \centering
    \includegraphics[width=9cm]{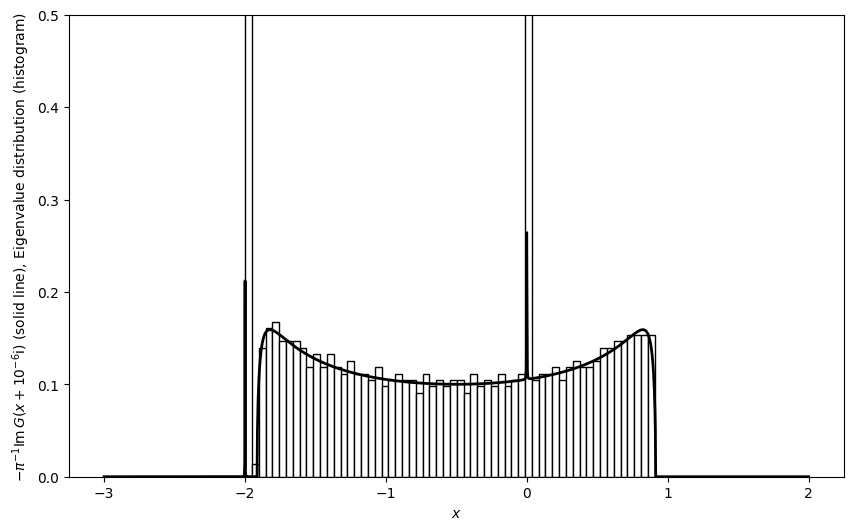}
    \caption{The density function of $(\frac{1}{3}\delta_1+\frac{2}{3}\delta_{-2})\multconv(\frac{1}{2}\delta_1+\frac{1}{2}\delta_0)$ and the scaled eigenvalue distribution of $BU^*AUB$, where $U$ is a realization of the Haar unitary random matrix, $A$ is a diagonal matrix with one-third of its eigenvalues equal to $1$ and the rest equal to $-2$, and $B$ is a diagonal matrix with half of its eigenvalues equal to $1$ and the other half equal to $0$. The size of all matrices is $3000$. The black line represents $-\pi^{-1}\Im G_{\mu\multconv\nu}(x+10^{-6}\iu)$ for $x$.}
    \label{fig:two_atoms2}
\end{figure}

For example, when $\lambda_\mu=\lambda_\nu=1/2$ and $\alpha_\mu=-3, \alpha_\nu=3$, we can deduce that 
\begin{align*}
    \Omega_1(z)&=\frac{z^2+9+\sqrt{(z-1)(z-3)(z+3)(z-9)}}{4z+6},\\
    \Omega_2(z)&=\frac{z^2+9+\sqrt{(z-1)(z-3)(z+3)(z-9)}}{12-2z},
\end{align*}
where the branch of the square root is chosen appropriately. From the relations $zM_{\mu\multconv\nu}(z)=\Omega_1(z)\Omega_2(z)$ and $zG_{\mu\multconv\nu}(z)=M_{\mu\multconv\nu}(z)/(M_{\mu\multconv\nu}(z)-1)$, we obtain
\begin{align*}
    G_{\mu\multconv\nu}(z)=\frac{1}{z}\frac{(z^2+9-\sqrt{Q(z)})^2}{(z^2+9-\sqrt{Q(z)})^2-4z(2z+3)(6-z)},
\end{align*}
where $Q(z)\coloneqq(z-1)(z-3)(z+3)(z-9)$ and the branch of the square root is again chosen appropriately. Using this formula, we can plot the density function of $\mu\multconv\nu$, as shown in Figure  \ref{fig:two_atoms1}, where it is compared with the histogram of the eigenvalue distribution of the corresponding random matrix model. The singularities observed in this graph arise from the zeros in the denominator of $G_{\mu\multconv\nu}$. Indeed, one can easily check that $-9,-1,1$ and $3$ are zeros of $(z^2+9-\sqrt{Q(z)})^2-4z(2z+3)(6-z)$.

Next, assume that $\lambda_\mu=\frac{1}{2}$, $\lambda_\nu=1/3$, $\alpha_\mu=0$ and $\alpha_\nu=-2$. In this case, we have 
\begin{align*}
    \Omega_1(z)&=\frac{2z+1+\sqrt{4z^2+4z-7}}{2},\\
    \frac{\Omega_2(z)}{z}&=\frac{2z-1+\sqrt{4z^2+4z-7}}{2(2-2z)},\\
\end{align*}
and 
\begin{align*}
    G_{\mu\multconv\nu}(z)=\frac{1}{z}\frac{(2z-1+\sqrt{4z^2+4z-7})(2z+1+\sqrt{4z^2+4z-7})}{(2z-1+\sqrt{4z^2+4z-7})(2z+1+\sqrt{4z^2+4z-7})-8(1-z)}.
\end{align*}

The density function in this case, along with the eigenvalue distribution of the corresponding random matrix model, is depicted  in Figure  \ref{fig:two_atoms2}. In this setting, atoms arise at $0$ and $-2$, which can also be observed as singularities of the Cauchy transform at these points.

\printbibliography
\end{document}